\documentclass[reqno,a4paper,11pt]{amsart}
\usepackage[T1]{fontenc} 
\usepackage{lmodern} 
\usepackage{amsmath,amsthm,amssymb,mathrsfs}
\usepackage[
centering,
left = 2.4cm,
right = 2.4cm,
top = 2.5 cm,
bottom = 2.3 cm
]{geometry}
\usepackage{graphicx,mathtools}
\usepackage[svgnames]{xcolor}
\usepackage{float}

\usepackage{todonotes}
\makeatletter
\define@key{todonotes}{Mingwen}[]{%
\setkeys{todonotes}{author=\textbf{Mingwen},inline,color=red!20}}%
\define@key{todonotes}{Xiang}[]{%
\setkeys{todonotes}{author=\textbf{Xiang},inline,color=green!20}}%
\define@key{todonotes}{Yadong}[]{%
\setkeys{todonotes}{author=\textbf{Yadong},inline,color=cyan!20}}%
\makeatother

\usepackage{tikz}
\usetikzlibrary{hobby} 
\usetikzlibrary{arrows.meta, positioning}
\usepackage{wrapfig}
\usepackage{pgfplots}
\pgfplotsset{compat=1.18} 
\usepackage{subcaption}

\usepackage{bm,stmaryrd}
\usepackage{enumitem}
\usepackage[toc,page]{appendix}
\usepackage[
pdfencoding = auto,
colorlinks = true,
citecolor = blue,
linkcolor = black,
anchorcolor = blue,
urlcolor = blue
]{hyperref}
\usepackage{cite}
\numberwithin{equation}{section}
\allowdisplaybreaks[4]

\theoremstyle{plain}
\newtheorem{theorem}{Theorem}[section]

\newtheorem{definition}[theorem]{Definition}
\newtheorem{lemma}[theorem]{Lemma}
\newtheorem{proposition}[theorem]{Proposition}
\newtheorem{assumption}[theorem]{Assumption}
\newtheorem{remark}[theorem]{Remark}

\newcommand{\bb}[1]{\mathbb{#1}}
\newcommand{\bbr}{\bb{R}}
\newcommand{\bbi}{\bb{I}}

\newcommand{\bbt}{\bb{T}}

\newcommand{\bbd}{\bb{D}}

\newcommand{\bbz}{\bb{Z}}

\newcommand{\bu}{\mathbf{u}}

\newcommand{\bw}{\mathbf{w}}

\newcommand{\be}{\mathbf{e}}

\newcommand{\bphi}{\boldsymbol{\varphi}}

\newcommand{\cB}{\mathcal{B}}

\newcommand{\cE}{\mathcal{E}}

\newcommand{\cL}{\mathcal{L}}
\newcommand{\cM}{\mathcal{M}}
\newcommand{\cN}{\mathcal{N}}

\newcommand{\cT}{\mathcal{T}}
\newcommand{\cS}{\mathcal{S}}

\newcommand{\tm}{\tilde{m}}

\renewcommand{\d}{\mathrm{d}}
\newcommand{\dx}{\d x}

\newcommand{\dxdt}{\d x\d t}

\newcommand{\dt}{\d t}
\newcommand{\ddt}{\frac{\d}{\d t}}
\newcommand{\dtau}{\d\tau}
\newcommand{\ptial}[1]{ \partial_{#1} }
\newcommand{\pt}{\ptial{t}}
\newcommand{\onehalf}{\frac{1}{2}}

\newcommand{\abs}[1]{\left\vert #1 \right \vert}
\newcommand{\absm}[1]{\vert #1 \vert}
\newcommand{\norm}[1]{\left\Vert #1 \right \Vert}
\newcommand{\normm}[1]{\Vert #1 \Vert}

\newcommand{\inner}[2]{\left\langle #1 , #2 \right\rangle} 

\newcommand{\doi}[1]{DOI: \href{https://doi.org/#1}{\texttt{#1}}}
\newcommand{\arxiv}[1]{arXiv: \href{https://arxiv.org/abs/#1}{\texttt{#1}}}

\usepackage{esint} 
\newcommand{\mean}[1]{\fint_{\bbt^3} #1 \,\dx}
\newcommand{\zeromean}[1]{#1 - \mean{#1}}

\DeclareMathOperator{\Div}{\mathrm{div}}

\DeclareMathOperator*{\supp}{\mathrm{supp}}

\DeclareMathOperator*{\meas}{\mathrm{meas}}

\title[NSCH system for quasi-incompressible two-phase flows]{Global weak solutions to a diffuse-interface model \\ for quasi-incompressible two-phase flows with \\ 
unmatched densities and singular potential}

\author{Mingwen Fei}
\address{School of  Mathematics and Statistics, Anhui Normal University, Wuhu 241002, P. R. China}
\email{mwfei@ahnu.edu.cn}
\author{Xiang Fei}
\address{School of  Mathematics and Statistics, Anhui Normal University, Wuhu 241002, P. R. China}
\email{feixiang@ahnu.edu.cn}
\author{Yadong Liu}
\address{School of Mathematical Sciences, Ministry of Education Key Laboratory of NSLSCS, and Key Laboratory of Jiangsu Provincial Universities of FDMTA, Nanjing Normal University, Nanjing 210023, P. R. China}
\email{ydliu@njnu.edu.cn}
\author{Hao Wu}
\address{School of Mathematical Sciences, Fudan University, Shanghai 200433, P. R. China}
\email{haowufd@fudan.edu.cn}

\date{\today}


\begin{document}	
\begin{abstract} 
	We study a thermodynamically consistent diffuse-interface model that describes the motion of two macroscopically immiscible, incompressible, and viscous Newtonian fluids with unmatched densities. This model is compatible with continuum mixture theory. It adopts a mass-averaged (barycentric) velocity so that the two-phase flow is quasi-incompressible: the velocity is no longer divergence-free, and the pressure enters the equation of the chemical potential. For the initial-boundary value problem in $\mathbb{T}^3$ with a class of physically relevant singular free energy densities, we prove the existence of global-in-time weak solutions. The proof relies on a suitable reduction of the original system to a Korteweg-type fluid model combined with a two-layer approximation, together with delicate estimates for the mass density and the phase-field variable inspired by the celebrated Bresch--Desjardins entropy. A key observation is that capillarity at the free interface provides a damping effect on the density evolution. For the limiting procedure, we derive delicate tail estimates to exclude possible concentrations of the singular potential, since no integrability of the pressure is available \textit{a priori}. 
	This work appears to be the first existence result for the Navier--Stokes/Cahn--Hilliard type system with unmatched densities and mass-averaged velocity without spatial regularization. 
\end{abstract}
\maketitle


\section{Introduction}

Multi-phase flows are of fundamental importance in physical, chemical, and biological processes, with a wide range of applications in engineering and industry. Over the decades, diffuse-interface (or phase-field) models have been successfully applied to describe the complex interfacial dynamics of macroscopically immiscible fluid mixtures, including breakup, coalescence, and drop formation \cite{AMW1998,GvdZ2016,LT1998,YPLS2004}. 
The idea, dating back to \cite{VDJ1979}, is to approximate the sharp interface with a thin transition layer (i.e., the diffuse interface) where partial mixing of immiscible fluids is allowed and to introduce a continuous order parameter (i.e., the phase-field variable) based on the volume or mass fraction. The diffuse-interface model avoids explicit tracking of the free interface in its mathematical formulation and provides a convenient way to capture topological changes of the interface \cite{DF2020}. As a typical example of diffuse-interface models for a conserved field, the Cahn--Hilliard equation \cite{CH1958} is a fundamental continuum model for phase separation in binary mixtures, driven by non-Fickian diffusion due to gradients in chemical potential. Describing the motion of macroscopic fluids also requires coupling with a momentum equation, e.g., the Navier--Stokes equations. The first hydrodynamic system consisting of the Navier--Stokes equations and the Cahn--Hilliard equation, i.e., the so-called ``Model H'', was proposed by Hohenberg and Halperin \cite{HH1977}. The ``Model H'' describes the dynamics of incompressible isothermal viscous two-phase flows under the assumption of matched densities, that is, the density of the fluid mixture and the density of the individual constituents are the same constant. Subsequently, extended diffuse-interface models have been developed for incompressible two-phase flows in the general scenario with unmatched densities. These models are based on a suitably defined average velocity field, for instance, the mass-averaged velocity \cite{ADGK2014,LT1998,SYW2013,SRSvBvdZ2018} and the volume-averaged velocity \cite{AGG2012,Boyer2002,Ding2007}. For a comparative study of existing Navier--Stokes/Cahn--Hilliard models for incompressible two-phase flows with density-varied components, we refer to \cite{TKID2023}, where a unified framework based on continuum mixture theory and thermodynamical principles was developed. 

\subsection{Model description} 
In this work, we investigate a diffuse-interface model that describes the motion of two incompressible, immiscible Newtonian fluids of unmatched densities and viscosities \cite{TKID2023} (see also \cite{SRSvBvdZ2018} for an equivalent formulation under a slightly different form of the pressure). For simplicity, we consider the problem in $Q_T:=\mathbb{T}^3\times (0,T)$, where $T > 0$ is a given final time. Let $\rho_1$ and $\rho_2$ be the constant specific densities of the incompressible isothermal constituents. For unmatched densities, we assume that $0<\rho_1<\rho_2=1$ without loss of generality. Next, we introduce the phase-field variable based on the volume fractions $\phi:\mathbb{T}^3\times [0,T]\to [-1,1]$ such that the values $\phi=1$ and $\phi=-1$ represent the pure phases of the first and second constituents, respectively. Then the average density of the binary mixture is given by 
\begin{align}\label{density}
\rho(\phi)= \rho_1\frac{1+\phi}{2}+ \rho_2\frac{1-\phi}{2}
. 
\end{align}
The mixture velocity $\mathbf{u}: \mathbb{T}^3\times [0,T]\to \mathbb{R}^3$ is taken as a mass-averaged velocity (or barycentric velocity) and is identified through the relation
$$
\mathbf{u}=\frac{1}{\rho(\phi)}\left( \frac{1+\phi}{2}\rho_1\mathbf{u}_1+ \frac{1-\phi}{2}\rho_2\mathbf{u}_2\right),
$$
where $\mathbf{u}_i$ denotes the velocity of the $i$-th fluid component ($i=1,2$). Then the hydrodynamic system under investigation reads as follows 
\begin{subequations}
	\label{model2}
	\begin{align}
		\label{model2-1} \partial_t(\rho  \mathbf{u}) + \Div\left (\rho  \mathbf{u}\otimes
		\mathbf{u}\right) + \nabla p& =  \Div S(\phi,\mathbb{ D}\mathbf{u}) - 
		 \phi \nabla \mu, && \text{in} \ Q_T,\\
		\label{model2-2}
        \partial_t\rho + \Div(\rho\mathbf{u})&=0,
        && \text{in} \ Q_T,\\
		\label{model2-3}
		\partial_t\phi + \Div(\phi  \mathbf{u})& =\Delta(\mu+\alpha p),  && \text{in} \ Q_T,\\
		\label{model2-4}
		\mu& = - \Delta \phi + F'(\phi),  && \text{in} \ Q_T.
	\end{align}
    In addition, the system \eqref{model2-1}--\eqref{model2-4} is subject to the initial conditions
	\begin{align}
		\label{model2-5-0}
		\mathbf{u}|_{t=0} =\mathbf{u}_{0},\quad  \phi|_{t=0}=\phi_{0}, 
        \quad\text{in $\bbt^3$}.
	\end{align}
\end{subequations}

In the momentum equation \eqref{model2-1}, $p:\mathbb{T}^3\times [0,T]\to \mathbb{R}$ is the pressure. The matrix-valued function 
\begin{align*}
	S(\phi, \bbd \bu) \coloneqq 2 \nu(\phi) \bbd \bu
	+ \lambda(\phi) (\Div \bu) \,\bbi
	= 2 \nu(\phi) \Big(\bbd \bu- \frac{1}{3} (\Div \bu) \,\bbi\Big) + \Big(\lambda(\phi) + \frac{2 \nu(\phi)}{3}\Big) (\Div \bu)\, \bbi
\end{align*}
denotes the Newtonian stress tensor with  the shear viscosity $ \nu(\phi) > 0 $ and the bulk viscosity $\lambda(\phi)$ satisfying $ \lambda(\phi) + 2 \nu(\phi) / 3 \geq 0 $. Here, $\bbd \bu \coloneqq (\nabla \bu + (\nabla \bu)^\top)/2$ denotes the symmetric gradient of the velocity and $\mathbb{I}$ is the three-dimensional identity matrix. Since the average viscosity of the fluid mixture usually depends linearly on $\phi$ and thus linearly on $\rho$ (cf. \eqref{density}). For simplicity, we take   
\begin{align} 
	\label{eqs:viscosity-relation}
	\nu(\phi) = \frac{1}{2} \rho(\phi), \quad \lambda(\phi) = 0,\quad \text{and}\quad S(\phi,\bbd\mathbf{u})= \rho(\phi)\bbd\mathbf{u}.
\end{align}
The above choice of density-dependent viscosities in  \eqref{eqs:viscosity-relation} is also common in the context of shallow water fluids \cite{BD2003} and compressible fluids \cite{AS2022,BDL2003,MSZ2025,VY2016InventMath,VY2016}. 

In the convective Cahn--Hilliard equation \eqref{model2-3}--\eqref{model2-4}, the constant $\alpha$ is defined by 
\begin{align}
	\alpha \coloneqq \frac{\rho_2 - \rho_1}{\rho_2 + \rho_1}
    \in (0,1).
    \label{def-alpha}
\end{align}
For simplicity, the mobility coefficient is assumed to be one. The scalar function $\mu:\mathbb{T}^3\times [0,T]\to \mathbb{R}$ denotes the chemical potential, which is given by the variational derivative of the Ginzburg--Landau-type free energy: 
\begin{align}
\label{free-evergy}
\mathcal{E}_{\text{GL}}(\phi)= \int_{\mathbb{T}^{3}} \Big( \frac{1}{2}|\nabla \phi|^2+F(\phi)\Big)\, \mathrm dx.
\end{align}
In \eqref{free-evergy}, we have set the width of the diffuse interface to one, since we do not consider the sharp-interface limit in this work.
The nonlinear function $F$ denotes the homogeneous Helmholtz free energy density that usually has a double-well structure: 
\begin{align}
	F(r) = F_c(r) - \frac{\omega}{2} r^2,\quad  \forall\, r \in (-1,1),
    \label{F}
\end{align}
where $F_c$ is a strictly convex function, and the quadratic term gives a concave perturbation with $\omega \geq \inf_{r \in (-1,1)} F_c''(r)$. This condition on $\omega$ accounts for possible phase separation in the binary fluid mixture. 

Returning to \eqref{model2-1}, we note that the second term on its right-hand side, that is, 
\begin{align}
\label{ca-force}
\phi\nabla\mu&=\Div\left(\nabla\phi\otimes\nabla\phi+ \Big(\phi\mu-\frac{1}{2}|\nabla \phi|^2-F(\phi)\Big)\mathbb{I}\right),
\end{align}
models the capillary force due to surface tension acting on the free interface. Next, substituting \eqref{density} into \eqref{model2-2}, we get 
	\begin{align}\notag 
		\partial_{t}\phi +\Div(\phi\mathbf{u}) &=\frac{1}{\alpha}\Div\mathbf{u},\quad \text{in } Q_T,
	\end{align}
which combined with \eqref{model2-3}  yields 
  \begin{align}
  \label{original-weak-quasi}
\Div\mathbf{u}&=\alpha\Delta(\mu+\alpha p),\quad \text{in } Q_T.
  \end{align}
Conversely, from \eqref{density}, \eqref{model2-3} and \eqref{original-weak-quasi}, we can also recover the continuity equation \eqref{model2-2}. Because of  \eqref{original-weak-quasi}, the barycentric velocity $\mathbf{u}$ is in general not divergence-free and the fluid mixture is referred to as quasi-incompressible. In addition, the pressure $p$ can be viewed as a Lagrange multiplier for the constraint \eqref{original-weak-quasi}. Formally, letting $ \alpha \to 0 $ (equivalently $\varepsilon\to 0$ or $ \rho_1 -\rho_2 \to 0$), the system \eqref{model2-1}--\eqref{model2-4} with \eqref{density}, \eqref{original-weak-quasi} reduces to the classical ``Model H'' for two-phase flows of incompressible viscous fluids with matched densities.

The coupled system \eqref{model2-1}--\eqref{model2-4} has several distinct features. First, it provides a framework compatible with the metaphysical principles of mixture theory (see e.g., \cite{Tru1984,TKID2023}). Second, it is thermodynamically consistent and obeys an energy-dissipation law.  The total energy $\cE$ of the system is defined by
\begin{align}\label{Etot}
	\cE(\mathbf{u}, \phi)=
	\int_{\mathbb{T}^{3}}\frac{1}{2}\rho(\phi)|\mathbf{u}|^2\, \mathrm{d} x+ \mathcal{E}_{\text{GL}}(\phi),
\end{align}  
where the free energy $\mathcal{E}_{\text{GL}}$ is given as in  \eqref{free-evergy}. For sufficiently smooth solutions $(\bu,p,\phi,\mu)$, we can check that the total energy \eqref{Etot} satisfies the following energy equality: 
\begin{align}\label{ConEnL}
	\frac{\mathrm d}{\mathrm dt}
	\cE(\mathbf{u}, \phi) 
    &=
	-\int_{\mathbb{T}^{3}}  \rho(\phi)
	\abs{\bbd(\mathbf{u})}^2\,
	\dx
	-\int_{\mathbb{T}^{3}} |\nabla (\mu+\alpha p)|^2\,\dx,
    \quad \forall\, t\in 
    (0,T).
    \end{align}
Moreover, thanks to the periodic boundary conditions for state variables, this system conserves mass and momentum, that is, 
\begin{equation}
	\label{eqs:conservation-mass-momentum}
	\ddt \int_{\bbt^3} \rho \,\dx =\ddt \int_{\bbt^3} \phi \,\dx = 0, \quad
	\ddt \int_{\bbt^3} \rho \bu \,\dx = 0,\quad \quad \forall\, t\in (0,T).
\end{equation}
The balance of energy
and mass 
presented in \eqref{ConEnL} 
and \eqref{eqs:conservation-mass-momentum} 
will play a crucial role in the subsequent analysis of the initial-boundary value problem \eqref{model2}. 

\subsection{State of art}
The analysis of diffuse-interface models for incompressible two-phase flows with unmatched densities has drawn increasing attention in the literature. 

Concerning the thermodynamically consistent diffuse-interface model derived in Abels, Garcke and Gr\"{u}n \cite{AGG2012}, which is based on a volume-averaged velocity,  the existence of global weak solutions has been established in \cite {ADG20131,ADG2013}. In \cite{AGG2024}, the authors further proved global regularity and asymptotic stabilization of global weak solutions as $t \to \infty$ (see also \cite{GP2025} where a different approach for long-time behavior of global weak solutions was introduced). Extensive analysis of local/global strong solutions has been carried out in \cite{AW2021,GA2021,GA2022,GLW2026}. We also mention some recent results on various generalizations of this model, including models with dynamic boundary conditions \cite {GGW2019,GaLW2026,GK2023}, its nonlocal variant \cite {FS1,GGGP2023}, extensions accounting for the interaction with a chemical species \cite{AGG2012,LW2018,GHW2025}, the case with phase transition \cite{AGW2026}, as well as generalizations to incompressible two-phase viscoelastic fluids \cite{LT2025} and incompressible multi-phase flows \cite{AGP2024}, to name a few. 
 
Despite the rich literature on diffuse-interface models based on a volume-averaged velocity, analytical results for models employing a mass-averaged velocity, i.e., for quasi-incompressible two-phase flows, are rather limited. A first attempt was made by \cite{AbelsCMP2009}, where the author proved the existence of global weak solutions for the Navier--Stokes/Cahn--Hilliard system derived in Lowengrub and Truskinovski \cite{LT1998}, under suitable spatial regularization (i.e., via the $p$-Laplacian) and regular potentials. The existence and uniqueness of a local strong solution was subsequently established in \cite{Abels2012}, without the $p$-Laplacian regularization. 
Recently, the first three authors of this study and Han established the existence of global weak solutions to the system \eqref{model2-1}--\eqref{model2-4} in $\mathbb{T}^3\times(0,T)$, again under certain spatial regularization (i.e., via the fractional Laplacian) and regular potentials, see \cite{FFHL2025}, where they also investigated the incompressible limit as the density difference tends to zero. For the existence and uniqueness of a local strong solution, we refer to \cite{FFHL2024}. It is worth mentioning that spatial regularization and regular potential assumptions play an essential role in the aforementioned works \cite{AbelsCMP2009,FFHL2025} on the existence of global weak solutions.

\subsection{Main result} 
Our goal is to establish a first result on the existence of global weak solutions to problem \eqref{model2} for a class of singular potentials, without the need for any spatial regularization.

First, let us introduce the necessary assumptions.
\begin{assumption}\label{main assumption}
The constant $\alpha$ is defined as in \eqref{def-alpha}. The potential function $F$ takes the form \eqref{F} with $\omega\in \mathbb{R}$, where $F_c \in C^2(-1,1) $ is strictly convex such that 
	\begin{align*}
		\lim_{r\to \pm 1} F_c'(r)=\pm \infty, \quad \text{ and } \quad 
		F_c''(r) > 0, \quad \forall\, r \in (-1,1).
	\end{align*}
    There exists some $\sigma_0\in (0,1)$ such that $F_c''$ is non-decreasing in $[1-\sigma_0,1)$ and non-increasing in $(-1,-1+\sigma_0]$.
    In addition, we assume that   
	\begin{align}
		\liminf_{\abs{r} \to 1} \abs{F_c(r) (1-{r^2})^{\beta}} > 0, \quad \text{ for some } \beta\geq \frac{3}{2}.  \label{beta assum}
	\end{align}
\end{assumption}
\begin{remark}\rm
The requirement $\beta \geq 3/2$ in \eqref{beta assum} applies to the three-dimensional case. In two spatial dimensions, this condition can be relaxed to $\beta>1$, see Remark \ref{rem:2D-beta} for further details.
\end{remark}
Next, we present the notion of global weak solutions. 
\begin{definition}\label{def:original-vision}
	Let $T\in (0,\infty)$, $Q_T:=\mathbb{T}^3\times (0,T)$. We call the quadruple $(\bu,\phi,\mu,p)$ a weak solution to problem \eqref{model2} on the time interval $[0,T]$, provided that
	\begin{enumerate}
		\item The quadruple  $(\bu,\phi,\mu,p)$ satisfies
		\begin{align*}
			&\mathbf{u}\in L^\infty(0,T; L^{2}(\mathbb{T}^{3})) \cap L^{2}(0,T; H^{1}(\mathbb{T}^{3})),    
			\\ 
            &\phi \in BC_w([0,T]; H^1(\mathbb{T}^{3})) \cap L^{2}(0,T; H^{2}(\mathbb{T}^{3}))\cap W^{1,\frac43}(0,T;L^2(\mathbb{T}^3)),
            \\ 
            & \phi\in L^\infty(Q_T)\quad \text{with}\quad \absm{\phi(x,t)} < 1\  \text{ a.e. in } Q_T, 
			\\ 
			& \mu+\alpha p \in L^\infty(0,T;H^1(\mathbb{T}^3))\cap L^{2}(0,T; H^{2}(\mathbb{T}^{3})),
            \\
            & \mu \in L^1(Q_T), \quad p \in L^1(Q_T),\quad  F'(\phi) \in L^1(Q_T).
		\end{align*}
		\item For all $\boldsymbol{\varphi}\in C_{c}^{\infty}([0,T); C^{\infty}(\mathbb{T}^{3}))$, it holds 
		\begin{align}
			\nonumber
			&
            - \int_{\mathbb{T}^3} \rho_0\mathbf{u}_0\cdot \bphi(\cdot,0)\,\dx
            -\int_0^{T}\int_{\mathbb{T}^{3}} \rho\bu \cdot\partial_{t}\bphi\,\dx\dt \\
            & \quad \quad
			- \int_0^{T} \int_{\mathbb{T}^{3}}
			(\rho\mathbf{u}\otimes
			\mathbf{u}):\nabla \bphi\,\dx\dt
			+ \int_0^{T}\int_{\mathbb{T}^{3}} \rho \bbd \bu:\bbd \bphi\,\dx\dt
			\notag \\ 
            & 
			\quad = \int_0^{T}\int_{\mathbb{T}^{3}}p\Div\bphi\,\dx\dt +  \int_0^{T}\int_{\mathbb{T}^{3}} (\nabla \phi \otimes \nabla \phi) :\nabla \bphi\,\dx\dt
			+ \int_0^{T}\int_{\mathbb{T}^{3}} \phi\mu\Div\bphi\,\dx\dt
			\nonumber\\
			& \qquad - \int_0^{T}\int_{\mathbb{T}^{3}}\frac{\absm{\nabla \phi}^2}{2}\Div\bphi\,\dx\dt			-\int_0^{T}\int_{\mathbb{T}^{3}} F(\phi)\Div\bphi\,\dx\dt.
			\label{original-weak1}
		\end{align}
		\item  The relation \eqref{density} and equations \eqref{model2-2}--\eqref{model2-4} hold almost everywhere in $Q_T$.
		\item For almost all $t \in (0,T]$, the following energy inequality holds: 
		\begin{align}\label{ConEnL2}
			\cE(\mathbf{u}(t), \phi(t))
			+\int_0^{t}\int_{\mathbb{T}^{3}} \rho \abs{\bbd \bu}^2 \,\dx\d\tau
			+\int_0^{t}\int_{\mathbb{T}^{3}} |\nabla (\mu+\alpha p)|^{2} \,\dx\d\tau
			\leq \cE(\mathbf{u}_{0},\phi_{0}).
		\end{align}
		\item For almost all $t \in (0,T]$, the following Bresch--Desjardins entropy estimate holds:
		\begin{align}
			& E_{\mathrm{BD}}(\mathbf{u}(t), \rho(t)) 
			+\int_0^{t}\int_{\mathbb{T}^{3}} \rho \abs{\bbd \bu}^2 \,\dx\d\tau
			+\int_0^{t}\int_{\mathbb{T}^{3}} |\nabla (\mu+\alpha p)|^{2} \,\dx\d\tau
            \notag\\
            & \quad  + \left(\frac{\alpha+1}{\alpha}\right)^2 \int_0^t 
			\int_{\mathbb{T}^{3}} \abs{\Delta \rho}^{2} \mathrm{d}x\d\tau
			\leq C_T,
            \label{BD-entropy-original} 
		\end{align} 
		where 
        \begin{align*}
			E_{\mathrm{BD}}(\bu,\rho)
			\coloneqq \int_{\bbt^3} 
			\Big(
			\frac{\rho}{2} \abs{\bu + \nabla \ln \rho}^2 
            +  \frac{1}{2}\abs{\nabla \phi (\rho)}^2 + F(\phi(\rho))		
			\Big)\,\dx,
		\end{align*}
        and $C_T>0$ is a constant depending on $\cE(\bu_0,\phi_0)$, $\inf_{x\in \mathbb{T}^3}\rho(\phi_0(x))$ and $T$. 
		\item The initial condition \eqref{model2-5-0} is satisfied in the following sense  
		\begin{align}
			& \phi|_{t=0}=\phi_{0},\quad  \text{ a.e.~in }\ \mathbb{T}^{3},
            \notag 
		\end{align}	
        while $\mathbf{u}_0$ is attained in the weak formulation \eqref{original-weak1}.
	\end{enumerate}
\end{definition}
%
%

Now we are in a position to state the main result. 
\begin{theorem}[Existence of global weak solutions]
	\label{thm:main}
    Let $T\in (0,\infty)$ be arbitrarily fixed final time. Suppose that Assumption \ref{main assumption} is satisfied. 
    For any initial data $(\mathbf{u}_0,\phi_0)$ satisfying 
    $$
    \bu_{0}\in L^2(\bbt^3),\quad \phi_{0}\in H^1(\bbt^3) \quad \text{and}\quad F(\phi_{0})\in L^1(\bbt^3),
    $$
    problem \eqref{model2} admits a global weak solution $(\bu,\phi,\mu,p)$ on $[0,T]$ in the sense of Definition \ref{def:original-vision}. Moreover, we have 
	\begin{align*}
		\int_{\bbt^3} \phi(\cdot,t) \,\dx = \int_{\bbt^3} \phi_0 \,\dx, 
        \quad \forall\, t\in [0,T].
	\end{align*} 
\end{theorem}

In what follows, we give some comments on the problem setting and the main theorem.

\begin{remark}[The singular potential function $F$]
\rm 
A physically relevant example of $F$ is the logarithmic potential \cite{CH1958}:
\begin{align}
    F(r)=F_c(r)-\frac{\omega}{2}s^2
    =(1+r)\ln (1+r) + (1-r)\ln(1-r)-\frac{\omega}{2}r^2,\quad r\in (-1,1).\label{log}
\end{align}
Unfortunately, the logarithmic potential \eqref{log} does not satisfy Assumption \ref{main assumption} due to its weak singularity at the pure phases $\pm 1$. 
In a recent work \cite{AGW2026} on incompressible two-phase flows with phase transitions, the authors considered the stationary Stokes equations coupled to a Cahn--Hilliard equation with a logarithmic potential. This system serves as a simplified version of the hydrodynamic model derived in \cite{ADGK2014}, which is based on a mass-averaged velocity. However, the argument therein essentially depends on the stationary structure, which provides explicit regularity for the pressure directly, and therefore does not apply to our case. The existence of global weak solutions to problem \eqref{model2} with a logarithmic-type singular potential remains an interesting and challenging problem.

The singularity of $F$ (or its derivative) near the pure phases $\pm 1$ ensures that the phase-field variable $\phi$ remains within the physical interval $[-1,1]$ (see e.g., \cite{GGW2018,LP2018,MZ2004}), thereby preventing degeneracy of the density $\rho$ (see e.g., \cite{ADG20131,AGW2026,FS1,GA2021, GHW2025}). Here, we also require sufficiently strong singularity of $F$ to rule out possible concentrations of the pressure, see Section \ref{sec:equi-integrability} for further details. A similar consideration arises when studying the strict separation property of the phase-field variable from the pure states for the three-dimensional Cahn--Hilliard equation \cite{GP2024,LP2018,MZ2004}, see also \cite{FLM2016} for a hydrodynamic system for quasi-incompressible fluids. 
A singular potential with power-law growth near $\pm 1$, for example, $F_c(\phi) = (1-\phi^2)^{-\beta}$, is admissible under Assumption \ref{main assumption}. The range of the exponent $\beta$ is a technical assumption, and finding its sharp bound is an interesting open question. On the other hand, this type of singular potential is motivated by the van der Waals equation of state (i.e., the Lennard–Jones potential), which appears in the gas dynamics of hard-sphere systems via the Carnahan--Starling equation \cite{CS1969,FLM2016,FPRS2010,FZ2010}. 
Such a potential is referred to as cold pressure in the compressible fluid community, where it is used to characterize the phenomenon of plastification (cf.~\cite{BD2007,GLV2015}). Further applications can be found in the study of chemically reactive flows \cite{FLM2016,MPZ2013}. 
\end{remark}

\begin{remark}[Weak formulation of the momentum equation]
\rm 
Due to the poor regularity of the chemical potential $\mu$ (only $L^1(Q_T)$-integrability), we have used an alternative weak formulation \eqref{original-weak1} for the momentum equation \eqref{model2-1} through the relation \eqref{ca-force}. In the context of diffuse-interface models based on a volume-averaged velocity (see, e.g., \cite{ADG20131,AGG2012,GGW2019,LT2025}), the incompressibility condition $\Div \mathbf{u}=0$ holds. Under this property, we can replace the force term $\phi\nabla \mu$ by $\Div(\nabla\phi\otimes \nabla \phi)$, because the original pressure $p$ can be modified to $p - \phi \mu + \absm{\nabla \phi}^2/{2}+ F(\phi)$. We also note that the regularity property $p \in L^1(Q_T)$ is new in the context of diffuse-interface models based on a mass-averaged velocity (cf. \cite{AbelsCMP2009,FFHL2025}). In addition, it is a nontrivial task to obtain compactness of the pressure in $L^1(Q_T)$ during the limit passage.
\end{remark}

\begin{remark}[The continuity equation and the Cahn--Hilliard equation]\rm 
Due to the special structure \eqref{original-weak-quasi}, within the framework of weak solutions, the modified chemical potential 
$\mu_p=\mu+\alpha p$ enjoys much better regularity than the original chemical potential $\mu$ in the Navier–Stokes--Cahn--Hilliard system based on the volume-averaged velocity (see e.g., \cite{ADG20131,AGP2024,AGW2026,GLW2026}). This yields better regularity for $\phi$ (resp. $\rho$), so that the Cahn--Hilliard equation \eqref{model2-3} (resp. the continuity equation \eqref{model2-2}) is satisfied in the strong sense, and the initial condition of $\phi$ (resp. $\rho$) is attained almost everywhere in $\mathbb{T}^3$. 
\end{remark}

\begin{remark}[The Bresch--Desjardins entropy estimate]
\rm
Inspired by \cite{BD2003} on the viscous shallow-water equations (see also \cite{BDL2003}), we include a Bresch--Desjardins entropy (BD-entropy in short) estimate \eqref{BD-entropy-original} in Definition \ref{def:original-vision} for weak solutions. This is new in the context of diffuse-interface models for two-phase flows. Since the basic energy inequality \eqref{ConEnL2} does not provide sufficient estimates for the approximate solutions, we cannot obtain the compactness needed to pass to the limit. For example, no information on the pressure $p$ and the chemical potential $\mu$ is available. With the aid of the BD-entropy estimate, we can obtain supplementary regularity and compactness of the density function $\rho$, and thus possibly also for $\phi$ and $p$. The BD-entropy estimate reveals that the capillary force term $\phi\nabla\mu$ provides an additional damping effect for the density evolution, resulting in better regularity from a mathematical point of view, see Lemma \ref{lem:BD-entropy} for further details. This point was previously observed in \cite{FFHL2025} for the system \eqref{model2-1}–\eqref{model2-4} with spatial regularization and regular potentials, and it turns out to be a distinct feature compared to other Navier--Stokes/Cahn--Hilliard type systems. 
\end{remark}

\begin{remark}[Requirement on the initial data]
\rm
Our assumption on the initial data $(\mathbf{u}_0,\phi_0)$ implies that the initial total energy $\mathcal{E}(\mathbf{u}_0,\phi_0)$ is finite. Due to the singularity of $F$ at $\pm 1$, we infer from the condition $F(\phi_0) \in L^1(\bbt^3)$ that $\phi_0\in (-1,1)$ almost everywhere in $\mathbb{T}^3$. This property combined with the assumption 
$\phi_0 \in H^1(\bbt^3)$ and the relation $\rho_0 =\rho(\phi_0)$ (cf. \eqref{density}) yields that 
$$
\rho_0 \in H^1(\bbt^3)\quad \text{and}\quad 0 <\underline{\rho}:= 1-\frac{2}{\zeta} < \rho_0 < 1,\quad \text{with}\ \ \ \zeta:=\frac{\alpha+1}{\alpha}
>2,
$$ 
from which we find 
\begin{align*}
	0\leq \int_{\bbt^3} \rho_0|\nabla \ln \rho_0|^2\,\mathrm{d}x=  \int_{\bbt^3} \frac{\absm{\nabla\rho_0}^2}{\rho_0}\, \dx 
	\leq \frac{1}{\zeta(\zeta - 2)} \int_{\bbt^3} \absm{\nabla\phi_0}^2\, \dx.
\end{align*}
As a consequence, the initial BD-entropy 
$E_{\mathrm{BD}}(\bu_0,\rho_0)$ is also bounded. 
\end{remark}

\subsection{Strategy of proofs}
For diffuse-interface models of quasi-incompressible two-phase flows formulated in terms of a mass-averaged velocity, the available results in the literature on global weak solutions rely heavily on spatial regularization and regular potentials, as seen in \cite{AbelsCMP2009,FFHL2025}. 
In this work, our aim is to handle more general cases involving a class of singular potentials without any spatial regularization. The main challenge arises from the poor regularity of the chemical potential $\mu$ and the pressure $p$. The energy-dissipation law \eqref{ConEnL} provides information only on their linear combination $\mu + \alpha p$, but not on each of them separately. This further prevents us from obtaining useful estimates on the phase-field variable $\phi$ and the singular term $F'(\phi)$ from \eqref{model2-4}, in contrast to the classical Cahn--Hilliard equation and its variants  \cite{AGG2024,GP2024,GGW2018,LP2018,MZ2004}.

Our strategy is to exploit hidden information from the momentum equation \eqref{model2-1}. Inspired by \cite{FFHL2025}, we eliminate the pressure $p$ in \eqref{model2-1} using the quantity 
\begin{equation}
\mu_p := \mu + \alpha p = - \Delta \phi + F'(\phi) + \alpha p.
\label{muap}
\end{equation}
This allows us to reduce the original Navier--Stokes/Cahn--Hilliard system \eqref{model2-1}--\eqref{model2-4} to a Korteweg-like fluid model in terms of the velocity $\mathbf{u}$ and the density $\rho$, see~\eqref{eqs:reduced} below. 
Thanks to the third-order term $\nabla \Delta \rho$ in the capillary force, we can obtain regularity information without directly invoking \eqref{model2-4} for the chemical potential, cf.~\cite[Lemma 4.5]{FFHL2025}. This shares similar properties with the classical compressible Navier--Stokes--Korteweg system for liquid-vapor flows, see Section \ref{sec:reformulation} for further details. 


With the aforementioned reformation, it remains challenging to handle each term in the reduced momentum equation, for example, the newly defined pressure $\widetilde{P}(\rho)$ as in \eqref{eqs:new-entropy} and the higher-order force term $\rho \nabla \Delta \rho$. This difficulty stems from the singular potential under consideration, which exhibits algebraic growth near $\pm 1$ and for which we lack sufficient \textit{a priori} information. By choosing a suitable test function in terms of the density $\psi(t) \cB \big[\zeromean{\rho}\big]$ for $\psi(t) \in C_c^\infty((0,T))$, where $\mathcal{B}$ is the celebrated Bogovski\u{\i} operator (see \eqref{eqs:Bogovskii}), we can derive the $L^1$-integrability of $\widetilde{P}(\rho)$; see also \cite{FP2000,FZ2010} in the context of compressible fluids. Furthermore, inspired by \cite{BD2003,BDL2003}, we can exploit the third-order capillary force term $\rho\nabla \Delta \rho$ to obtain better regularity properties for the density function using the BD-entropy estimate (a formal argument can be found in Lemma \ref{lem:BD-entropy} below). We note that the BD-entropy estimate does not rely on a positive lower bound of $\rho$ and even allows for the vacuum state. Thanks to this estimate, we can obtain $H^2$-regularity for the density $\rho$, and consequently for the phase-field variable $\phi$ in view of \eqref{density}. This in turn yields an estimate for the pressure $p$ in $L^1(Q_T)$, cf. Lemma~\ref{lem:uniform-P-delta} below.

Let us now comment on the approximation scheme. As noted above, convergence from an approximate system of the problem \eqref{model2} to itself is a non-trivial task due to the lack of regularity and compactness for the pressure. To construct a suitable approximate system for \eqref{model2} that is convenient to solve and possesses favorable regularity properties for its solution, we introduce a family of regular potentials $F_\sigma$ with $0<\sigma\ll 1$ that appropriately approximates the singular potential $F$, following the approach for the standard Cahn--Hilliard equation with singular potentials (cf.~\cite{GGW2018}). For problem \eqref{model2} with regular potentials, the existence of global weak solutions was established in \cite{FFHL2025}, under a spatial regularization involving a fractional Laplacian (which is essential in the proof therein and cannot be removed). 
Hence, a natural idea is to add a fractional Laplacian term $\delta \Lambda^{2s} \phi$ to the chemical potential equation \eqref{model2-4} and then take the limit as $\delta \to 0$.  Here, $0 < \delta \ll 1$, and $\Lambda^{2s}$ (for some $s > 3/2$) denotes the fractional Laplacian defined via the Fourier transform on the torus $\mathbb{T}^3$.

However, the limiting procedure is non-trivial. First, the result in \cite{FFHL2025} did not provide sufficient information on the pressure, which is the Lagrange multiplier corresponding to \eqref{original-weak-quasi}. Instead, a weaker notion of solutions involving a modified pressure term was adopted. The argument therein does not apply to the current setting because it crucially relies on the phase-field variable being confined to a finite range, namely $\phi \in (-1-\theta, 1+\theta)$ \textit{a priori} for some small constant $\theta > 0$. This confinement ensures that the density $\rho$ admits a uniformly positive lower bound. Unfortunately, the finite range of $\phi$ depends on the regularization term $\delta \Lambda^{2s} \phi$ (see \cite[Lemma A.5]{FFHL2025}) and is therefore not uniform with respect to the regularization parameter $\delta$. Second, as noted above, one may exploit the compressible Korteweg structure to derive the BD-entropy estimate. With the regularization term $\delta \Lambda^{2s} \phi$, one can obtain higher-order regularity, such as $\phi \in L^{2}(0,T; H^{s+1}(\mathbb{T}^{3}))$, from the BD-entropy estimate, which, however, is not uniform at the level of weak solutions. Consequently, one cannot employ the BD-entropy estimate to derive the integrability of the singular term  $F'(\phi)$ and hence the integrability of the pressure $p$.

To fully exploit the capillary structure, we approximate the original problem \eqref{model2} and then examine a reformulated equivalent system of Navier--Stokes--Korteweg type. This allows us to derive improved regularity and compactness properties for the equivalent approximate problem and to establish convergence in a proper sense. Our approximate scheme is inspired by \cite{AS2022,MSZ2025,VY2016} on the reduction of a Navier--Stokes--Korteweg system for compressible fluids. More precisely, it contains a two-layer approximation characterized by parameters $\sigma, \delta>0$, which reads as follows
\begin{subequations}
	\begin{align*}
		\begin{split}
			& \partial_t (\rho \bu) 
			+ \Div(\rho \bu \otimes \bu) 
			+ \nabla \widetilde{P}_\sigma(\rho) 
			+ \delta \rho \nabla W_\delta'(\rho)
			- \Div (\rho \bbd \bu)\\
			& \qquad \qquad = 
			-\delta \rho \nabla \Delta^2 \rho
			- \delta \nabla \rho \cdot \nabla \bu
			- \delta \Delta^2 \bu
			+ \zeta^2 \rho \nabla \Delta \rho
			- \frac{1}{\alpha^2} \nabla \Delta^{-1} \Div \bu,
		\end{split}
		\\
		& \pt \rho + \Div (\rho \bu) = \delta \Delta \rho,\\
		& \mathbf{u}|_{t=0} =\mathbf{u}_{0},\quad  \rho|_{t=0}=\rho_{0,\delta}.
	\end{align*}
\end{subequations}
Here, we have introduced several approximation terms that play different roles in the subsequent analysis: 
\begin{itemize}
	\item The newly defined singular pressure $\widetilde{P}$ is approximated by $\widetilde{P}_\sigma$, such that $\widetilde{P}_\sigma$ becomes regular near $\underline{\rho}$ and $1$, i.e., the endpoints of $\rho$. This approximation follows the standard expansion near the endpoints used in phase-field models. We note that the construction of $\widetilde{P}_\sigma$ is independent of $\delta$.
	\item The artificial regular potential $W_\delta$ leads to an $L^\infty$-estimate for $\rho$. It is designed to be sufficiently steep outside the interval $(\underline{\rho}, 1)$. This enables us to derive strictly positive upper and lower bounds for $\rho$ that extend slightly beyond the original interval $(\underline{\rho}, 1)$. In particular, these bounds are uniform with respect to $\sigma$. The construction of $W_\delta$ is independent of $\sigma$. 
	\item The terms $\delta \rho \nabla \Delta^2 \rho $, $\delta \Delta^2 \bu $ provide additional regularity properties of $(\bu, \rho) $ so that the testing procedure for the BD-entropy estimate makes sense.
	\item The artificial viscous term $\delta \Delta \rho$ in the continuity equation for $\rho$ enables us to apply the classical Faedo–Galerkin approximation scheme to construct solutions of the regularized system. The additional term $\delta\nabla \rho \cdot \nabla \bu$ in the momentum equation for $\mathbf{u}$ is introduced to complement this artificial viscous term, thus facilitating the derivation of necessary energy estimates.
\end{itemize}
In principle, these additional terms should be equipped with different parameters, however, in our case, we can let them vanish simultaneously, except for $\sigma > 0$. Compared to the constructions in \cite{AS2022,MSZ2025,VY2016}, we do not need a regularization like $ \delta \rho \nabla \left(\frac{\Delta \sqrt{\rho}}{\sqrt{\rho}}\right) $ due to the Bohm potential of quantum fluids. In fact, we are working with a singular potential that guarantees the density remains bounded and kept away from vacuum, so a term like $\nabla \rho^{-k}$ ($k \in \mathbb{Z}^+$) is not necessary here either. It is worth mentioning that, in the limit, we indeed establish the existence of weak solutions to a compressible Navier--Stokes--Korteweg system with singular pressure and density-dependent viscosity, see Theorem \ref{thm:NSK} below.

The limiting procedure for the aforementioned approximate system consists of two steps: first, we let $\sigma \to 0$ with $\delta > 0$ fixed (see Section \ref{sec:limit-sigma}), and then we send $\delta \to 0$ (see Section \ref{sec:limit-delta}). In this way, we obtain a limit system that is equivalent (in a weak sense) to the weak formulation \eqref{original-weak1} of the original problem \eqref{model2}, thus establishing the existence of global weak solutions. We note that the $L^1$-integrability of $\widetilde{P}_{\sigma}$ generally does not imply weak convergence. Inspired by \cite{FLM2016,FZ2010}, we establish the weak compactness in $L^1(Q_T)$ with the aid of the Dunford--Pettis theorem, by proving the equi-integrability. This goal is achieved by controlling the tail part of the integral to exclude possible concentrations, thanks to the BD-entropy estimate, see Lemmas \ref{lem:uniform-P-chi-sigma}, \ref{lem:uniform-P-chi} below. Then the weak $L^1$-convergence of $\widetilde{P}_{\sigma}$ gives rise to the corresponding weak $L^1$-convergence of the chemical potential $\mu$ and the pressure $p$, which ultimately justifies the weak formulation \eqref{original-weak1}.

We believe that our arguments will prove useful for the study of related diffuse-interface models for quasi-incompressible two-phase flows. In particular, they may provide a compatible framework for the analysis of the Lowengrub--Truskinovsky model \cite{LT1998} and the model proposed by Aki et al.~\cite{ADGK2014} for quasi-incompressible binary fluids with phase transition, for which the existence of global weak solutions remains open.

\subsection{Plan of the paper}
In Section \ref{sec:preliminaries}, we describe the necessary notation, present the formal reduction, and derive the \textit{a priori} estimates. Section \ref{sec:approximate-system} addresses the regularized problem, its finite-dimensional approximation, and the associated energy estimates. Then we prove the existence of a global weak solution to the regularized system. In Section \ref{sec:limit-sigma}, we pass to the limit in the approximated potential and pressure, thereby obtaining a global weak solution to an intermediate limit system with an artificial viscous term and other regularization terms. The compactness of the newly defined pressure is established via the equi-integrability. Section \ref{sec:limit-delta} deals with the limit procedure as $\delta \to 0$. We derive the BD-entropy estimate and obtain improved uniform estimates in Section \ref{sec:uniform-BD-app}. Then we establish the integrability and equi-integrability in Sections  \ref{sec:integrability-pressure-delta}, 
\ref{sec:equi-integrability}, respectively, using the additional regularity properties  provided by the BD-entropy estimate. Finally, the existence of global weak solutions to the original problem \eqref{model2}, along with the weak equivalence to the reformulated problem, is demonstrated in Section \ref{sec:limit-final}. In Appendix \ref{sec:technical-lemma}, we report some properties of the mollifiers on $\mathbb{T}^3$.

\section{Preliminaries}
\label{sec:preliminaries}

In this section, we first introduce the functional settings and the notation. Next, we reformulate the problem \eqref{model2} and derive some \textit{a priori} estimates for sufficiently smooth solutions.

\subsection{Notation and function spaces}
We denote $\mathbf{a}\otimes \mathbf{b}=(a_i b_j)^3_{i,j=1}$ for $\mathbf{a}, \mathbf{b}\in \mathbb{R}^3$. For any $q\in [1,\infty]$, $L^q(\mathbb{T}^3)$ denotes the Lebesgue space with norm $\|\cdot\|_{L^q(\mathbb{T}^3)}$. 
For $k\in\mathbb{Z}^+$, $q\in[1,\infty]$, $W^{k,q}(\mathbb{T}^3)$ denotes the Sobolev space with norm $\|\cdot\|_{W^{k,q}(\mathbb{T}^3)}$, and the spaces with negative order are defined by duality $W^{-k,q}(\mathbb{T}^3) = (W^{k,q'}(\mathbb{T}^3))'$, with $1/q +1/q'=1$. For $q=2$, we set $H^k(\mathbb{T}^3)=W^{k,2}(\mathbb{T}^3)$. For any $f\in L^1(\mathbb{T}^3)$, $\mean{f}$ denotes the mean value of $f$ in $\mathbb{T}^3$. Then we have the Poinca\'{e}--Wirtinger inequality
$$
\Big\|f-\mean{f}\Big\|_{L^2(\mathbb{T}^3)}\leq C_P\|\nabla f\|_{L^2(\mathbb{T}^3)},\quad \forall\, f\in H^1(\mathbb{T}^3), 
$$
where $C_P>0$ is independent of $f$. 

Given a Banach space $X$ with norm $|\cdot|_{X}$, we denote its dual space by $X'$ and the duality product by $\langle \cdot,\cdot\rangle_{X',X}$. For a Hilbert space $H$, we denote its inner product by $(\cdot,\cdot)_{H}$. When no ambiguity arises, we do not distinguish notationally between scalar-, vector-, or matrix-valued function spaces. 
Given an interval $J\subseteq [0,\infty)$, $L^q(J;X)$ with $q\in[1,\infty]$ denotes the space consisting of Bochner measurable $q$-integrable functions
with values in $X$. 
For $q\in[1,\infty]$, $W^{1,q}(J;X)$ is the space of all $f\in L^q(J;X)$ with $\partial_t f\in L^q(J;X)$.
For $q=2$, we set $H^1(J;X)=W^{1,2}(J;X)$. The set of continuous functions $f : J \to X$ is denoted by $C(J;X)$. In addition, $BC(J; X)$ is the Banach space of all bounded and continuous functions $f : J \to X$ equipped with the supremum norm, while $BC_w(J; X)$ denotes the topological vector space of all bounded and weakly continuous functions $f : J \to X$. Finally, for $J=(0,T)$, $T\in (0,\infty)$, we denote by $C^\infty_c(J; X)$ the space of all smooth functions $f : J \to X$ with $\mathrm{supp} f \subset\subset J$.


Throughout the manuscript, $C$ denotes a generic positive constant whose value may vary from line to line. The specific dependence of $C$ will be explained if necessary. For $A,B\geq 0$, the notation $A\lesssim B$ (resp. $A \gtrsim B$) represents inequalities of the form $A \leq CB$ (resp. $A \geq CB$) for some constant $C>0$, and $A\approx B$ means $B\lesssim A\lesssim B$. 

\subsection{Model reduction}
\label{sec:reformulation}

We rewrite the original system \eqref{model2-1}--\eqref{model2-4} in a suitable form that reveals a structure distinct from the classical Navier--Stokes/Cahn--Hilliard system. The resulting Korteweg-type system plays a crucial role in the subsequent analysis.

Using the linear relation \eqref{density} and the continuity equation \eqref{model2-2}, we rewrite the momentum equation \eqref{model2-1} as
\begin{align*}
	\rho (\pt \bu + \bu \cdot \nabla \bu)  + \frac{1}{\alpha} \nabla \mu_p = \Div(\rho \bbd \bu) + \zeta \rho \nabla \mu,
\end{align*}
where $\zeta=\frac{\alpha+1}{\alpha}$ and $\mu_p=\mu+\alpha p$ (see \eqref{muap}). Using \eqref{model2-4}, we further get a Navier--Stokes--Korteweg type equation
\begin{align}
	\label{eqs:momentum-reduced}
	\rho (\pt \bu + \bu \cdot \nabla \bu) + \frac{1}{\alpha} \nabla \mu_p = \Div(\rho \bbd \bu) - \zeta \rho \nabla \Delta \phi + \zeta \rho \nabla F'(\phi).
\end{align}
Expressing the potential $F$ in terms of the density instead of the phase-field variable, that is, 
\begin{align} 
\widetilde{F}(\rho) = F(\phi(\rho))\quad \text{and}\quad \widetilde{F}_c(\rho) = F_c(\phi(\rho)), \quad \text{with}\quad  \phi(\rho) = - \zeta \rho + \zeta -1,
\label{eqs:resuced-F}
\end{align} 
we have $\widetilde{F}'(\rho)=-\zeta F'(\phi)$ and can write \eqref{eqs:momentum-reduced} as
\begin{equation}
	\label{eqs:reduced}
	\rho (\partial_t \mathbf{u} + \bu \cdot \nabla \bu) +  \rho \nabla \widetilde{F}'(\rho)  
	=  \Div (\rho \bbd \bu)
	+ \zeta^2 \rho \nabla \Delta \rho 
	- \frac{1}{\alpha} \nabla \mu_p.
\end{equation}
Next, let us define the ``entropy'' functions 
\begin{align}
	\label{eqs:new-entropy}
	\widetilde{P}(\rho) = \rho\widetilde{F}'(\rho) - \widetilde{F}(\rho)\quad \text{and}\quad \widetilde{P}_c(\rho) = \rho\widetilde{F}_c'(\rho) - \widetilde{F}_c(\rho)
\end{align}
through Legendre's transform. Then we arrive at 
\begin{align}
	\label{eqs:reformulated-momentum}
	\rho (\partial_t \mathbf{u} + \bu \cdot \nabla \bu) + \nabla \widetilde{P}(\rho)  
	=  \Div (\rho \bbd \bu)
	+ \zeta^2 \rho \nabla \Delta \rho 
	- \frac{1}{\alpha} \nabla \mu_p,
\end{align}
which combined with the continuity equation \eqref{model2-2}, yields a Korteweg-type fluid system with an additional forcing term $- {\alpha}^{-1} \nabla \mu_p$. In the subsequent analysis, we shall see that this structure greatly facilitates in exploiting the improved regularity properties. Similar reductions can be found in \cite[Section 3, Section 4]{FK2017} for Naiver--Stokes/Cahn--Hilliard models for non-isothermal and isothermal compressible fluid mixtures. See also \cite{FK2019} for related discussions.

\begin{remark}\rm 
	The additional forcing term $- {\alpha}^{-1} \nabla \mu_p$ on the right-hand side of \eqref{eqs:reformulated-momentum} is ``good'' compared to the classical Navier--Stokes--Korteweg type system, in view of the energy dissipation. Indeed, from the divergence equation $\Div \bu = \alpha \Delta \mu_p$ (see \eqref{original-weak-quasi}), we formally find that
	\begin{align*}
		\frac{1}{\alpha} \nabla \mu_p = \frac{1}{\alpha^2} \nabla \Delta^{-1} \Div \bu \sim \bu,
	\end{align*}
	which serves as a damping term in the momentum equation. Here, $\nabla \Delta^{-1}$ denotes the pseudo-differential operator of the Fourier symbol $\mathrm{i} \xi/{\absm{\xi}^2}$. For some other hydrodynamic models, this kind of damping is usually assumed for the purpose of analysis, for example, a term like $- r_0 \bu$ in \cite{BD2003,BDL2003,VY2016}. Here, we get it for free from the Cahn--Hilliard coupling structure.
\end{remark}

Below we give some comments on the newly defined pressure $\widetilde{P}$.
\begin{remark}\label{rem:properties-entropy} \rm 
In view of \eqref{density}, for any given $\phi \in (-1,1)$, we have 
$$
\underline{\rho} < \rho(\phi) < 1,\quad \text{where}\quad \underline{\rho}:= 1-\frac{2}{\zeta}\in (0,1).
$$ 
Consider an admissible singular potential, for example, $F_c(\phi) = (1-\phi^2)^{-\beta}$. Near the corresponding singular points of the transformed potential $\widetilde{F}_c$, namely $\underline{\rho}$ and $1$, we have 
	\begin{align}
		\widetilde{F}(r) \sim \widetilde{F}_c(r) \sim \frac{1}{(1-r)^{\beta}} + \frac{1}{(r-\underline{\rho})^{\beta}},\label{singular behavior}
	\end{align}
	which implies 
	\begin{align*}
		\widetilde{F}'(r) \sim \widetilde{F}_c'(r) \sim \frac{\beta}{(1-r)^{\beta+1}} - \frac{\beta}{(r-\underline{\rho})^{\beta+1}}.
	\end{align*}
	Hence, 
	\begin{align*}
		\widetilde{P}(r) \sim \widetilde{P}_c(r) \sim \frac{\beta r}{(1-r)^{\beta+1}} - \frac{\beta r}{(r-\underline{\rho})^{\beta+1}}
		- \frac{1}{(1-r)^{\beta}} - \frac{1}{(r-\underline{\rho})^{\beta}}.
	\end{align*}
	From the fact $\widetilde{P}_c'(r)= r\widetilde{F}_c''(r) > 0$ for $r\in (\underline{\rho},1)$, we infer that $\widetilde{P}_c(r)$ has only one zero point on $(\underline{\rho},1)$, which we denote by $\rho_*$. Finally, we observe that the singular behavior of $\widetilde{P}$ and $\widetilde{F}'$ near $\underline{\rho}$ and $1$ are the same by construction, because $\widetilde{F}(r)$ is less singular than $r\widetilde{F}'(r)$.
	\begin{figure}[ht]\label{Fig-1}
		\def\zeta{3}
		
		\pgfmathsetmacro{\rmin}{1 - 2/\zeta}
		\pgfmathsetmacro{\rmax}{1}
		
		\pgfmathdeclarefunction{phi}{1}{\pgfmathparse{-\zeta*#1+\zeta-1}}
		
		\pgfmathdeclarefunction{Wphi}{1}{\pgfmathparse{1/(1 - (phi(#1))^2)}}
		
		\pgfmathdeclarefunction{dWdphi}{1}{\pgfmathparse{2*(phi(#1))/(1 - (phi(#1))^2)^2}}
		
		\pgfmathdeclarefunction{Fprime}{1}{\pgfmathparse{dWdphi(#1)*(-\zeta)}}
		
		\pgfmathdeclarefunction{Ptilde}{1}{\pgfmathparse{#1*Fprime(#1)-Wphi(#1)}}
		
		\begin{subfigure}[c]{0.48\textwidth}
			\centering
			\begin{tikzpicture}[scale=0.8]
				\begin{axis}[
					axis lines=middle, 
					width=8cm, height=6cm,
					xlabel={$r$}, 
					xlabel style={at={(1.075,-0.01)}, anchor=north east}, 
					ylabel={$\widetilde{F}_c(r)$},
					ylabel style={at={(0,1)}, anchor=east}, 
					xtick=\empty,
					ytick=\empty,
					domain=\rmin+0.0175:\rmax-0.0175, samples=300,
					ymin=0, ymax=10,
					xmin=0.3, xmax=1.05,
					clip=false %
					]
					\addplot[blue,thick] {Wphi(x)};
					\draw[dashed,black] (axis cs:\rmin, 0) -- (axis cs:\rmin, 10);
					\draw[dashed,black] (axis cs:\rmax, 0) -- (axis cs:\rmax, 10);
					\node[below] at (axis cs:\rmin,0) {$\underline{\rho}$};
					\node[below] at (axis cs:\rmax,0) {$1$};
				\end{axis}
			\end{tikzpicture}
			\caption{}
			\label{fig:Frho}
		\end{subfigure}
		\begin{subfigure}[c]{0.48\textwidth}
			\centering
			\begin{tikzpicture}[scale=0.8]
				\begin{axis}[
					axis lines=middle, 
					width=8cm, height=6cm,
					xlabel={$r$}, 
					xlabel style={at={(1.075,0.49)}, anchor=north east}, 
					ylabel={$\widetilde{P}_c(r)$},
					ylabel style={at={(0,1)}, anchor=east}, %
					xtick=\empty,
					ytick=\empty,
					domain=\rmin+0.01:\rmax-0.01, samples=300,
					ymin=-50, ymax=50,
					xmin=0.3, xmax=1.05
					]
					\addplot[red,thick] {Ptilde(x)};
					\draw[dashed,black] (axis cs:\rmin, -50) -- (axis cs:\rmin, 50);
					\draw[dashed,black] (axis cs:\rmax, -50) -- (axis cs:\rmax, 50);
					\node[below] at (axis cs:\rmax-0.25,0) {$\rho_*$};
					\node[below right] at (axis cs:\rmin,0) {$\underline{\rho}$};
					\node[below left] at (axis cs:\rmax,0) {$1$};
				\end{axis}
			\end{tikzpicture}
			\caption{}
			\label{fig:Prho}
		\end{subfigure}%
		\caption{Illustration of the newly defined potential $\widetilde{F}_c$ and entropy $\widetilde{P}_c$}
		\label{fig:entropy}
	\end{figure}
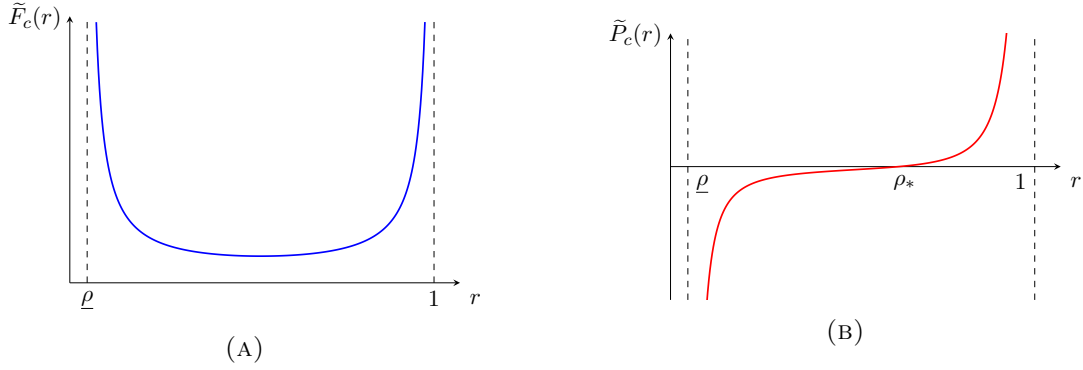
\end{remark}

\subsection{\textit{A priori} estimates}
\label{sec:a-priori-estimates}

We derive formal \textit{a priori} estimates for smooth solutions to problem \eqref{model2}, which will be rigorously justified by suitable approximations later. 
\begin{lemma}[Conservation of mass and momentum]
	For smooth solutions of problem \eqref{model2}, it holds
	\begin{gather}
		\label{eqs:conservation-mass-momentum-rho}
		  \ddt \int_{\bbt^3} \rho \,\dx 
         = \ddt \int_{\bbt^3} \phi \,\dx 
         =  0, \quad
		\ddt \int_{\bbt^3} \rho \bu \,\dx = 0,\quad \forall\, t\in (0,T).
	\end{gather}
\end{lemma}
\begin{proof}
	The first equality in \eqref{eqs:conservation-mass-momentum-rho} can be derived by integrating \eqref{model2-2} and \eqref{model2-3} over $\bbt^3$, while the second can be derived by integrating \eqref{model2-1} over $\bbt^3$. We note that the integral involving the capillary force term vanishes due to \eqref{ca-force}, that is, 
	\begin{align*}
		\int_{\bbt^3} \phi \nabla \mu \,\dx
		& =  \int_{\bbt^3}\Div\Big(\nabla\phi\otimes\nabla\phi+ \Big(\phi\mu-\frac{1}{2}|\nabla \phi|^2-F(\phi)\Big)\mathbb{I}\Big)\,\dx
		= 0.
	\end{align*}
	The proof is complete. 
\end{proof}
\begin{lemma}[Energy identity]
	For smooth solutions of problem \eqref{model2}, it holds
	\begin{align}
		& \frac{\mathrm d}{\mathrm dt} 
		\int_{\mathbb{T}^{3}}
        \Big(\frac{\rho}{2}|\mathbf{u}|^2 
        + \frac{1}{2}|\nabla \phi|^2 
		+ F(\phi)\Big)\, \mathrm dx
		+ \int_{\mathbb{T}^{3}}  \rho
		|\mathbb{D} \mathbf{u}|^2 \, \mathrm{d} x 
		+ \int_{\mathbb{T}^{3}} |\nabla (\mu+\alpha p)|^2\, \mathrm{d} x
		= 0,  
		\label{eqs:Energy-indentity}
	\end{align}
    for any $t\in (0,T)$.
\end{lemma}
\begin{proof}
	Multiplying \eqref{model2-1}, \eqref{model2-3} and \eqref{model2-4} by $ \mathbf{u} $, $ \mu $ and $ \partial_{t} \phi $, respectively, integrating over $\mathbb{T}^3$ and adding the resultants together, using integration by parts and \eqref{model2-2}, \eqref{original-weak-quasi}, we arrive at the conclusion \eqref{eqs:Energy-indentity}. The proof is complete.
\end{proof}

Next, we present an auxiliary identity related to the continuity equation \eqref{model2-2}.
\begin{lemma}\label{lem:BD-entropy-logrho}
	For smooth solutions of problem \eqref{model2}, it holds
	\begin{align*}
		\frac{1}{2}\frac{\mathrm d}{\mathrm dt}\int_{\mathbb{T}^{3}} \rho|\nabla \ln \rho|^2\, \mathrm{d} x+\int_{\mathbb{T}^{3}} \nabla\Div\mathbf{u}\cdot\nabla\rho\, \mathrm{d} x+\int_{\mathbb{T}^{3}} \rho\bbd\mathbf{u}:\nabla \ln \rho\otimes \nabla \ln \rho\, \mathrm{d} x=0,
	\end{align*}
    for any $t\in (0,T)$. 
\end{lemma}
\begin{proof}
	The proof is the same as that for  \cite[Lemma 2]{BDL2003}. Here, we just recall that if the phase-field variable satisfies $\phi \in (-1,1)$, then the mass density is strictly positive and bounded, i.e.,  $0<\underline{\rho} < \rho(\phi) < 1$. 
\end{proof}

Finally, we derive a BD-entropy law for the reformulated equation \eqref{eqs:reduced} using the idea in \cite[Lemma 4]{BDL2003}.
%
\begin{lemma}[BD-entropy law]
	\label{lem:BD-entropy}
	For smooth solutions of problem \eqref{model2}, it holds 
	\begin{align}
		& \ddt \int_{\bbt^3} 
		\Big(
		\frac{\rho}{2} \abs{\bu + \nabla \ln \rho}^2
		+ \frac{1}{2}|\nabla \phi|^2
        + F(\phi) 
		- \frac{1}{\alpha^{2}} \ln \rho
		\Big) \, \dx 
        + \int_{\bbt^3} \absm{\nabla \mu_p}^2 \, \dx  
        \nonumber \\
		& \quad 
        +\int_{\mathbb{T}^{3}} \widetilde{F}''(\rho)|\nabla\rho|^{2} \mathrm{d} x
		+ \zeta^2\int_{\mathbb{T}^{3}} |\Delta\rho|^{2} \,\mathrm{d} x  
		+\frac{1}{4} \int_{\mathbb{T}^{3}} \rho  \abs{\nabla \bu-(\nabla\bu)^\top}^2\,\mathrm{d}x 
		=  0,
        \label{eqs:BD-u-nabla-log-rho}
	\end{align}
    for any $t\in (0,T)$. 
\end{lemma}
\begin{proof}
	Multiplying  \eqref{eqs:reduced} by $ \nabla \ln \rho$ and integrating over $\mathbb{T}^{3}$, we have
	\begin{align}
		& \int_{\mathbb{T}^{3}} (\partial_{t}\mathbf{u} + \mathbf{u} \cdot \nabla \mathbf{u}) \cdot \nabla\rho\, \mathrm{d} x
		+ \int_{\mathbb{T}^{3}} \bbd \mathbf{u}: \Big(\nabla^{2}\rho-\frac{\nabla\rho \otimes \nabla\rho}{\rho} \Big)\, \mathrm{d} x 
        \nonumber
		\\ 
        & \qquad + \int_{\mathbb{T}^{3}} \widetilde{F}''(\rho) \abs{\nabla \rho}^2\,\mathrm{d} x
		+\zeta^2 \int_{\mathbb{T}^{3}}|\Delta\rho|^{2}\, \mathrm{d} x
		+ \frac{1}{\alpha} \int_{\mathbb{T}^{3}} \nabla \mu_{p} \cdot \nabla \ln \rho \,\mathrm{d} x = 0. \label{BD-entropy 122}
	\end{align}
    From \eqref{BD-entropy 122} and Lemma \ref{lem:BD-entropy-logrho} we infer that 
	\begin{align} 
		&\int_{\mathbb{T}^{3}} \partial_{t}\mathbf{u}\cdot \nabla\rho \, \mathrm{d} x 
        +\int_{\mathbb{T}^{3}}(\mathbf{u}\cdot\nabla \mathbf{u})\cdot \nabla\rho \, \mathrm{d} x
		+ \int_{\mathbb{T}^{3}} \bbd \mathbf{u}: \nabla^{2} \rho \, \mathrm{d} x \nonumber
		\\
        &\qquad +
		\int_{\mathbb{T}^{3}} \widetilde{F}''(\rho)  \abs{\nabla \rho}^2 \, \mathrm{d} x
		+ \zeta^2 \int_{\mathbb{T}^{3}} |\Delta\rho|^{2} \, \mathrm{d} x
		+ \frac{1}{\alpha} \int_{\mathbb{T}^{3}}\nabla\mu_{p}\cdot \nabla \ln \rho \, \mathrm{d} x \nonumber
		\\
        &\qquad +\frac{1}{2}\frac{\mathrm d}{\mathrm dt}\int_{\mathbb{T}^{3}} \rho|\nabla \ln \rho|^2\, \mathrm{d} x +\int_{\mathbb{T}^{3}} \nabla\Div\mathbf{u}\cdot\nabla\rho\, \mathrm{d} x = 0.
        \label{BD-entropy 123}
	\end{align}
	Using \eqref{model2-2} and integration by parts, we find   
	\begin{align*}
		\int_{\mathbb{T}^{3}} \partial_{t} \mathbf{u} \cdot \nabla \rho \,\mathrm{d} x 
        & 
		= \ddt \int_{\mathbb{T}^{3}} \mathbf{u} \cdot \nabla \rho \, \mathrm{d} x 
		-\int_{\mathbb{T}^{3}} \mathbf{u} \cdot \nabla \partial_{t} \rho \, \mathrm{d} x
		\\ 
        & = \ddt \int_{\mathbb{T}^{3}} \mathbf{u} \cdot \nabla \rho \, \mathrm{d} x 
		+\int_{\mathbb{T}^{3}} \mathbf{u} \cdot \nabla \Div(\rho\mathbf{u}) \, \mathrm{d} x
		\\ 
        & =\ddt \int_{\mathbb{T}^{3}} \mathbf{u} \cdot \nabla \rho \, \mathrm{d} x 
		-\int_{\mathbb{T}^{3}} (\mathbf{u}\cdot\nabla \mathbf{u}) \cdot \nabla \rho \, \mathrm{d} x 
        -
		\int_{\mathbb{T}^{3}} \rho\nabla\mathbf{u} : (\nabla \mathbf{u})^\top \,\mathrm{d} x,
	\end{align*}
	which together with \eqref{BD-entropy 123} yields 
	\begin{align}
		&  \ddt \int_{\mathbb{T}^{3}} \mathbf{u} \cdot \nabla \rho \, \mathrm{d} x 
        + \frac{1}{2} \ddt \int_{\mathbb{T}^{3}} \rho|\nabla \ln \rho|^2\, \mathrm{d} x 
        +\int_{\mathbb{T}^{3}} \widetilde{F}''(\rho)|\nabla\rho|^{2} \, \mathrm{d} x 
		+ \zeta^2\int_{\mathbb{T}^{3}} |\Delta\rho|^{2} \, \mathrm{d} x  
        \nonumber
		\\
		& 
		\quad = \int_{\mathbb{T}^{3}} \rho \nabla\mathbf{u}:(\nabla \bu)^\top 
        \, \mathrm{d} x
		- \frac{1}{\alpha}  \int_{\mathbb{T}^{3}} \nabla \mu_p \cdot \nabla \ln \rho \, \dx 
        \nonumber
		\\
		& 
		\quad = \int_{\mathbb{T}^{3}} \rho \nabla\mathbf{u}:(\nabla \bu)^\top \, \mathrm{d} x
		+ \frac{1}{\alpha^2}  \int_{\mathbb{T}^{3}} \Div \bu \ln \rho \,\dx, 
        \label{eqs:BD-nabla-log-rho}
	\end{align}
	where we have used the fact $\Div \bu = \alpha \Delta \mu_p$ and
	\begin{align*}
		\int_{\mathbb{T}^{3}}  \bbd \mathbf{u}: \nabla^{2} \rho \mathrm{d}x + \int_{\mathbb{T}^{3}} \nabla\Div\mathbf{u}\cdot\nabla\rho\, \mathrm{d} x=0.
	\end{align*}
	From \eqref{model2-2} we also get
	\begin{align} \label{eqs:BD-ddt-log-rho}
		\int_{\mathbb{T}^{3}} \Div \bu \ln \rho  \,\dx 
        =-\int_{\mathbb{T}^{3}} \mathbf{u} \cdot \nabla \ln \rho \, \mathrm{d} x
        =\ddt \int_{\mathbb{T}^{3}} \ln \rho \,\mathrm{d} x.
	\end{align}
 Using \eqref{eqs:BD-nabla-log-rho}, \eqref{eqs:BD-ddt-log-rho} and the identity
	\begin{align*}
		\nabla \mathbf{u} : (\nabla \mathbf{u})^\top
		=  |\mathbb{D} \mathbf{u}|^2 -
		\frac{1}{4} \abs{\nabla \bu-(\nabla\bu)^\top}^2,
	\end{align*} 
    we can deduce that 
	\begin{align}
		& \frac{1}{2} \ddt \int_{\mathbb{T}^{3}} \rho|\nabla \ln \rho|^2\, \mathrm{d} x
		+ \ddt \int_{\mathbb{T}^{3}} \mathbf{u} \cdot \nabla \rho \,\mathrm{d} x 
		- \frac{1}{\alpha^{2}} \ddt \int_{\mathbb{T}^{3}} \ln \rho \, \mathrm{d} x 
        \nonumber \\ 
		& \qquad 
        +\int_{\mathbb{T}^{3}} \widetilde{F}''(\rho)|\nabla\rho|^{2} \, \mathrm{d} x
		+ \zeta^2\int_{\mathbb{T}^{3}} |\Delta\rho|^{2} \,\mathrm{d} x  
		+\frac{1}{4} \int_{\mathbb{T}^{3}} \rho  \abs{\nabla \bu-(\nabla\bu)^\top}^2\,\mathrm{d}x 
        \nonumber \\
		&\quad = \int_{\mathbb{T}^{3}} \rho |\mathbb{D} \mathbf{u}|^2\,  \mathrm{d} x. \label{eqs:BD-nabla-log-rho-2}
	\end{align}
Finally, combining \eqref{eqs:BD-nabla-log-rho-2} with the energy identity \eqref{eqs:Energy-indentity}, we arrive at the conclusion \eqref{eqs:BD-u-nabla-log-rho}.

The proof is complete.
\end{proof}

\section{The Regularized System}
\label{sec:approximate-system}
In this section, we study a properly regularized system for the original problem \eqref{model2}. 

\subsection{Regularizations}
Substituting \eqref{model2-2}, \eqref{original-weak-quasi} and \eqref{muap} into the reformulated momentum equation \eqref{eqs:reformulated-momentum}, we obtain 
\begin{align}
	\label{eqs:reformulated-momentum-b}
	\partial_t (\rho \bu) 
			+ \Div(\rho \bu \otimes \bu)  + \nabla \widetilde{P}(\rho)  
	=  \Div (\rho \bbd \bu)
	+ \zeta^2 \rho \nabla \Delta \rho 
	- \frac{1}{\alpha^2} \nabla \Delta^{-1} \Div \bu.
\end{align}
Let $(\bu_{0},\rho_{0}) \in L^2(\bbt^3) \times H^1(\bbt^3)$, with $\rho_0\in (\underline{\rho},1)$ almost everywhere in $\mathbb{T}^3$.
For any given (small) constants $\sigma,\delta >0$, we consider the following regularized problem for equation \eqref{eqs:reformulated-momentum-b} and the continuity equation \eqref{model2-2}:  
\begin{subequations}
	\label{eqs:app-P-sigma}
	\begin{align}
			& \partial_t (\rho \bu) 
			+ \Div(\rho \bu \otimes \bu) 
			+ \nabla \widetilde{P}_\sigma(\rho) 
			+ \delta \rho \nabla W_\delta'(\rho)
			- \Div (\rho \bbd \bu) 
            \notag \\
			& \qquad \qquad = 
			-\delta \rho \nabla \Delta^2 \rho
			- \delta \nabla \rho \cdot \nabla \bu
			- \delta \Delta^2 \bu
			+ \zeta^2 \rho \nabla \Delta \rho
			- \frac{1}{\alpha^2} \nabla \Delta^{-1} \Div \bu,
		\label{eqs:app-momentum-P-sigma}
		\\
		& \pt \rho + \Div (\rho \bu) = \delta \Delta \rho, 
        \label{eqs:app-mass-P-sigma}
        \\ 
		& \mathbf{u}|_{t=0} =\mathbf{u}_{0},\quad  \rho|_{t=0}=\rho_{0,\delta}.
        \label{eqs:initial data}
	\end{align}
\end{subequations}
\begin{remark} 
\label{rem:ini}
\rm
Related to the original problem \eqref{model2}, the initial datum $\rho_0$ of interest is given by 
$$\rho_0 = - \frac{\phi_0}{\zeta} - \frac{1}{\zeta} + 1,\quad \text{for}\ \ \phi_0\in H^1(\mathbb{T}^3)\ \ \text{with} \ \ F(\phi_0)\in L^1(\mathbb{T}^3).
$$ 
We note that the two equations \eqref{eqs:app-momentum-P-sigma} and \eqref{eqs:app-mass-P-sigma} have already formed a closed system for $(\mathbf{u}, \rho)$. Once $\rho$ is constructed, we can determine the ``approximate'' phase-field variable $\phi$ by \eqref{density}, and then the corresponding ``approximate'' chemical potential through $\mu=-\Delta \phi + F_\sigma'(\phi)$, cf. \eqref{app-F}.
\end{remark}

\textbf{The regularized density $\rho_{0,\delta}$.} 
In \eqref{eqs:initial data}, for $0<\delta \ll 1$, we take the approximate initial density by
$$ 
\rho_{0,\delta} = \rho_0 * \widetilde{\eta}^{\mathrm{per}}_\delta\quad \text{with}\ \ \ \widetilde{\eta}^{\mathrm{per}}_\delta(x) = \sum_{\mathbf{k} \in \bb{Z}^3} \eta_{\delta^{1/4}}(x+\mathbf{k}),
$$ 
where $\eta_\delta$ (with  $\delta>0$) denote the classical mollifiers in $\bbr^3$, see Appendix \ref{sec:technical-lemma}. Then we have 
$$
\rho_{0,\delta}\in C^\infty(\mathbb{T}^3),\quad  \text{with}\ \ \underline{\rho}< \rho_{0,\delta}< 1\ \ \text{in}\ \ \mathbb{T}^3,
$$
satisfying 
\begin{align}
	\label{eqs:rho0-bounds}
	\norm{\rho_{0,\delta}}_{H^2(\bbt^3)} \leq C \delta^{-\frac14} \norm{\rho_0}_{H^1(\bbt^3)},
\end{align}
and  
\begin{align}
	\label{eqs:rho0-convergence}
	\norm{\rho_{0,\delta} - \rho_0}_{H^1(\bbt^3)} \to 0,\quad \text{as}\ \delta \to 0.
\end{align}

\textbf{The regularized potentials $F_{\sigma}$ and $\widetilde{F}_{\sigma}$.} 
The singular potential $ F$ is approximated by 
\begin{align}
F_{\sigma}(r) = F_{c,\sigma}(r) - \frac{\omega}{2} r^2 \in C^2(\mathbb{R}),
\label{app-F}
\end{align} 
for $0<\sigma \ll 1$, where $F_{c,\sigma}$ denotes the approximation of $F_c$ such that 
\begin{align*}
	F_{c,\sigma}(r)
	= \left\{
	\begin{aligned}
		&  \sum_{j=0}^2 \frac{1}{j!} F_c^{(j)}(1-\sigma) [r - (1-\sigma)]^j, && \quad r \geq 1-\sigma, \\
		& F_c(r), && \quad r \in (-1+\sigma, 1-\sigma), \\
		&  \sum_{j=0}^2 \frac{1}{j!} F_c^{(j)}(-1+\sigma) [r - (-1+\sigma)]^j, && \quad r \leq -1+\sigma,
	\end{aligned}
	\right.
\end{align*}
Then we define 
$$
\widetilde{F}_\sigma(\rho) = F_\sigma(\phi(\rho))\quad \text{and}\quad \widetilde{F}_{c,\sigma}(\rho) = F_{c,\sigma}(\phi(\rho)),
$$ 
through the linear relation \eqref{density}. 
By Assumption \ref{main assumption} and \cite[Lemma 3.1]{GGW2018}, there exists some constant $C_*>0$ that may depend on $\omega$, $\sigma_0$, $\zeta$, but is independent of $0<\sigma\ll 1$ such that 
\begin{align}
\widetilde{F}(r)\geq -C_*,\quad 
\widetilde{F}_\sigma(r)\geq -C_*,\quad \forall\, r\in \mathbb{R}.
\label{low-bd-F1}
\end{align}
In addition, similar to \cite[Lemma 3.2]{GGW2018}, we have 
\begin{align}
\widetilde{F}_{c,\sigma}(r)\leq \widetilde{F}_c(r),\quad \forall\, r\in (\underline{\rho},1).
\label{Fc-up-sigma}
\end{align}

Using Legendre's transform, we introduce the corresponding ``entropy'' functions 
\begin{align*}
	\widetilde{P}_\sigma(r) \coloneqq r \widetilde{F}_\sigma'(r) - \widetilde{F}_\sigma(r) \quad
    \text{and} \quad 
	\widetilde{P}_{c,\sigma}(r) \coloneqq r \widetilde{F}_{c,\sigma}'(r) - \widetilde{F}_{c,\sigma}(r).
\end{align*}
It follows that  $\widetilde{P}_{c,\sigma}'(r)=\widetilde{F}_{c,\sigma}''(r)>0$ for all $r>0$. Let $\rho_*$ be the only zero of $\widetilde{P}_{c}$ in $(\underline{\rho},1)$. We observe that $\widetilde{P}_{c}(\rho_*)=\widetilde{P}_{c,\sigma}(\rho_*)=0 $ whenever $0<\sigma\ll 1$. Thus, $\rho_*$ is the only zero of $\widetilde{P}_{c,\sigma}$ in $(\underline{\rho},1)$ as well. This also leads to the following formulation  
\begin{align*}
	\widetilde{F}_{c,\sigma}(r) 
    = r \int_{\rho_*}^r \frac{\widetilde{P}_{c,\sigma}(\tau)}{\tau^2} \,\mathrm{d} \tau + \frac{\widetilde{F}_{c,\sigma}(\rho_*)}{\rho_*} r  
    = r \int_{\rho_*}^r \frac{\widetilde{P}_{c,\sigma}(\tau)}{\tau^2} \,\mathrm{d} \tau + \widetilde{F}_{c,\sigma}'(\rho_*)r.
\end{align*}

\textbf{The additional potential $W_{\delta}$.} 
Let $ W \in C^2([\underline{\rho},1])$ be a given convex function that fulfills 
\begin{align}
   W(r) \geq 0 \quad \text{and} \quad W''(r) \geq 0, 
\quad \forall\, r\in [\underline{\rho},1]. 
   \label{W}
\end{align}
For convenience, we also introduce 
\begin{align}
	H(r) \coloneqq r W'(r) - W(r),\quad \forall\, r\in [\underline{\rho},1]. 
    \label{H def}
\end{align} 
Then the nonlinear function $W_\delta$ in 
\eqref{eqs:app-momentum-P-sigma} can be chosen 
by applying the following lemma, which is inspired by \cite{AbelsCMP2009} (see also \cite{FFHL2025}).
\begin{lemma}
\label{lower-and-up-bounds-for-rho-sigma}
	Let $W$ satisfy \eqref{W}, and let $\delta\in (0,1)$, $R>0$, $\theta \in (0,\underline{\rho})$, $c_0\in [\underline{\rho},1]$. There exists an extension $W_\delta\in C^2(\mathbb{R})$ with $W_\delta(r)=W(r)$ for $r\in [\underline{\rho},1]$, and $W_\delta(r)\geq 0$, $W_\delta''(r)\geq 0$ for all $r\in \mathbb{R}$, such that for any $\rho\in H^2(\mathbb{T}^3)$ with $\int_{\mathbb{T}^3} \rho\,\mathrm{d}x=c_0$, 
	\begin{align*}
		\text{if}\ \  \int_{\mathbb{T}^3}\Big(\frac{\delta}{2}|\Delta \rho|^2+ \delta W_\delta(\rho)\Big)\,\mathrm{d}x \leq R,\ \ \text{then} \quad  \rho(x) 
        \in (- \theta+\underline{\rho}, 1 + \theta)
		\quad 
		\text{for all}\ \ x\in \mathbb{T}^3.
    \end{align*}
\end{lemma}
\begin{proof} 
For reader's convenience, we sketch the proof by adapting the argument for \cite[Lemma 2.3]{AbelsCMP2009}. Due to the Sobolev embedding $H^2(\mathbb{T}^3)\hookrightarrow C^\frac12(\mathbb{T}^3)$, it holds 
$$
|\rho(x)-\rho(x_0)|\leq C_1\|\rho\|_{H^2(\mathbb{T}^3)}|x-x_0|^\frac12,\quad \forall\, x, x_0\in \mathbb{T}^3. 
$$
Since  
$\|\rho\|_{H^2(\mathbb{T}^3)}\leq R'=R'(R,\delta,c_0)$, 
we have 
$$
|\rho(x)-\rho(x_0)|\leq C_1R'|x-x_0|^\frac12\leq \frac{\theta}{2},\quad \forall\, x\in B_l(x_0)\cap \mathbb{T}^3 \ \ \text{with}\ \ l= \left(\frac{\theta}{2C_1R'}\right)^2.
$$
In addition, $\kappa :=\inf_{x_0\in \mathbb{T}^3}|B_l(x_0)\cap \mathbb{T}^3|>0$. Then we can choose a function $W_\delta\in C^2(\mathbb{R})$ such that $W_\delta(r)=W(r)$ for $r\in [\underline{\rho},1]$, and
$$
W_\delta(r)\geq \frac{R+1}{\delta \kappa} + \sup_{r\in [\underline{\rho},1]} W(r), \quad W_\delta''(r)\geq 0, \quad \text{for}\ r\leq - \frac{\theta}{2} +\underline{\rho}, \ \text{or}\ r\geq 1+ \frac{\theta}{2}.
$$
Assume, in contradiction, that there exists a point $x_0\in \mathbb{T}^3$ such that $\rho(x_0)\geq 1+\theta$. Then for all $x\in B_l(x_0)\cap \mathbb{T}^3$, we have $\rho(x)\geq 1-\theta/2$. As a result, 
$$
\int_{\mathbb{T}^3}\Big(\frac{\delta}{2}|\Delta \rho|^2+ \delta W_\delta(\rho)\Big)\,\mathrm{d}x \geq 
\int_{B_l(x_0)\cap \mathbb{T}^3} \delta W_\delta(\rho)\,\mathrm{d}x \geq R+1,
$$
which leads to a contradiction. Hence, $\rho(x)<1+\theta$ for all $x\in \mathbb{T}^3$. In a similar manner, we can conclude $\rho(x)>-\theta+\underline{\rho}$ in $\mathbb{T}^3$. The proof is complete. 
\end{proof}

For fixed $0<\sigma,\delta\ll 1$, we define the approximate energy  
\begin{align}
	E_{\sigma,\delta}(\bu,\rho) \coloneqq
	\int_{\bbt^3} \Big(
	\onehalf \rho \absm{\bu}^2
	+ \frac{\zeta^2}{2} \absm{\nabla {\rho}}^2
	+ \widetilde{F}_\sigma(\rho)	
    + \frac{\delta}{2} \absm{\Delta \rho}^2
	+ \delta W_\delta(\rho)
	\Big)\,\dx.
    \label{appro-energy}
\end{align}
The corresponding initial energy is given by 
\begin{align}
	E_{\sigma,\delta}(\bu_0,\rho_{0,\delta}) \coloneqq
	\int_{\bbt^3} \Big(
	\onehalf \rho_{0,\delta} \absm{\bu_0}^2
	+ \frac{\zeta^2}{2} \absm{\nabla {\rho_{0,\delta}}}^2
	+ \widetilde{F}_\sigma(\rho_{0,\delta})	
    + \frac{\delta}{2} \absm{\Delta \rho_{0,\delta}}^2
	+ \delta W(\rho_{0,\delta})
	\Big)\,\dx,
    \label{appro-energy-0}
\end{align}
thanks to the fact $\underline{\rho}\leq \rho_{0,\delta}\leq 1$.
Take 
$$
R= \big(E_{\sigma,\delta}(\bu_{0},\rho_{0,\delta})+C_*\big)\big(1+2\delta \omega T e^{2\delta \omega T}\big)+1,\quad \theta = \frac{1}{2}\underline{\rho},\quad c_0=\int_{\mathbb{T}^3} \rho_{0,\delta}\,\mathrm{d}x,
$$
where the constant $C_*>0$ is the same as in \eqref{low-bd-F1}. Then we can choose the admissible extension $W_\delta$ by Lemma \ref{lower-and-up-bounds-for-rho-sigma}. Once $W_\delta$ is constructed, we further set 
\begin{align}
	H_\delta(r) \coloneqq r W_\delta'(r) - W_\delta(r),\quad \forall\, r\in \mathbb{R}. 
    \label{H def-d}
\end{align}

 The main result of this section is the existence of global weak solutions to the regularized problem \eqref{eqs:app-P-sigma}. More precisely, we establish the following:
\begin{proposition}
	\label{prop:weak-solution-sigma}
	Let $0<\delta,\sigma \ll 1$ be arbitrarily small but fixed constants, and let $T\in (0,\infty)$. Given initial data $(\bu_{0},\rho_{0}) \in L^2(\bbt^3) \times H^1(\bbt^3)$ with $\rho_0\in (\underline{\rho},1)$ almost everywhere in $\mathbb{T}^3$, the associated regularized problem \eqref{eqs:app-P-sigma} with the aforementioned functions $\widetilde{P}_\sigma$ and $W_\delta$ admits a global weak solution $(\mathbf{u},\rho)$ on $[0,T]$ such that 
    \begin{align}
     & \mathbf{u} \in BC_w([0,T];L^2(\mathbb{T}^3)) \cap L^2(0,T;H^2(\mathbb{T}^3)),\notag \\
     & \rho\in BC([0,T];H^2(\mathbb{T}^3))\cap L^2(0,T;H^3(\mathbb{T}^3)) \cap H^1(0,T;H^1(\mathbb{T}^3)),
     \notag \\ 
     & \frac{1}{2}\underline{\rho}<  \rho(x,t)< 1+\frac{1}{2}\underline{\rho} \quad  \text{for all}\ \ (x,t)\in \mathbb{T}^3\times [0,T].
     \label{UL-rho-sigma} 
    \end{align} 
    The solution $(\bu,\rho)$ satisfies 
 \begin{align}
 \label{eqs:app-mass-P-sigma-delta}
		& \pt \rho + \Div (\rho \bu) = \delta \Delta \rho,\quad \text{a.e.~in } Q_T,
 \end{align}
 with $\rho|_{t=0}=\rho_{0,\delta}$ almost everywhere in $\mathbb{T}^3$, and the following weak formulation  
	\begin{align}
		\nonumber
		& 
		- \int_{\mathbb{T}^3} \rho_{0,\delta}\mathbf{u}_0\cdot\bphi(\cdot,0)\,\dx - \int_0^T \int_{\bbt^3} \rho \bu \cdot \pt \bphi \,\dxdt \\ 
    	\nonumber
    	& \qquad - \int_0^T \int_{\bbt^3} 
    	\rho \bu \otimes \bu : \nabla \bphi \,\dxdt
        + \int_0^T \int_{\bbt^3} \rho \bbd \bu : \bbd \bphi \,\dxdt \\
		\nonumber
		& \quad = 
		  \int_0^T \int_{\bbt^3} \Big(\widetilde{P}_\sigma(\rho)
		+ \delta H_\delta(\rho)
		\Big) \Div \bphi \,\dxdt 
        -\delta  \int_0^T \int_{\bbt^3} \rho \nabla \Delta^2\rho  \cdot \bphi \,\dxdt \\
		\nonumber
		& \qquad 
        - \delta\int_0^T \int_{\bbt^3}  (\nabla \rho \cdot \nabla \bu) \cdot \bphi \,\dxdt 
		-\delta  \int_0^T \int_{\bbt^3} \Delta \bu \cdot \Delta \bphi \,\dxdt 
		 \\
		\nonumber
		& \qquad - \zeta^2 \int_0^T \int_{\bbt^3} (\rho  \Delta \rho) \Div \bphi \,\dxdt
		- \zeta^2 \int_0^T \int_{\bbt^3}  \Delta \rho (\nabla \rho \cdot \bphi) \,\dxdt
		 \\ 
         & \qquad + \frac{1}{\alpha^2}\int_0^T \int_{\bbt^3}  (\Delta^{-1} \Div \bu) \Div \bphi \,\dxdt,\quad 
         \forall\, \boldsymbol{\varphi}\in C_{c}^{\infty}([0,T); C^{\infty}(\mathbb{T}^{3})). 
         \label{eqs:weak-formulation-ddt-sigma} 
	\end{align}
    Here, the nonlinear function $H_\delta$ is defined as in \eqref{H def-d} and the term $\int_0^T \int_{\bbt^3} \rho \nabla \Delta^2 \rho  \cdot \bphi \,\dxdt$ should be understood as 
	\begin{align}
		\nonumber
        &\int_0^T \int_{\bbt^3} \rho \nabla \Delta^2 \rho  \cdot \bphi \,\dxdt \\ 
        \nonumber
        &\quad \coloneqq \int_0^{T} \int_{\mathbb{T}^{3}}   \nabla^2\rho:(\nabla\Delta\rho \otimes \bphi)\,\dxdt  
        + \int_0^{T} \int_{\mathbb{T}^{3}}   (\nabla\Delta\rho\otimes \nabla\rho) : \nabla\bphi\,\dxdt\\
        &\qquad\ +\int_0^{T} \int_{\mathbb{T}^{3}}   \nabla\rho\cdot\nabla\Delta\rho  \Div\bphi\,\dxdt + \int_0^{T} \int_{\mathbb{T}^{3}}   \rho \nabla\Delta \rho\cdot \nabla\Div\bphi\,\dxdt.
        \label{higher-terms}
	\end{align}
Moreover, we have the conservation of mass 
\begin{align}
		\int_{\bbt^3} \rho(\cdot,t) \,\dx
		= \int_{\bbt^3} \rho_{0,\delta} \,\dx,\quad \forall\, t\in [0,T],
        \label{conservation law-mass-sd}
\end{align}
and the following energy estimate  
\begin{align}
	\nonumber
	& \sup_{0 \leq t \leq T} E_{\sigma,\delta}(\bu(t),\rho(t))
	+ \int_0^T \int_{\bbt^3} \rho \absm{\bbd \bu}^2 \,\dxdt
	+ \frac{1}{\alpha^2} \int_0^T \int_{\bbt^3} \absm{\nabla \Delta^{-1} \Div \bu}^2 \,\dxdt \\
	\nonumber
	&\qquad  + \delta \int_0^T \int_{\bbt^3} \absm{\Delta \bu}^2 \,\dxdt
	+ \delta^2 \int_0^T \int_{\bbt^3} \absm{\nabla\Delta \rho}^2 \,\dxdt
	+ \delta \zeta^2 \int_0^T \int_{\bbt^3} \absm{\Delta \rho}^2 \,\dxdt 
    \\
	& \quad \leq  E_{\sigma,\delta}(\bu_{0},\rho_{0,\delta}) 
    + \big(E_{\sigma,\delta}(\bu_{0},\rho_{0,\delta}) +C_*\big)
    2\delta \omega T e^{2\delta \omega T}.
	\label{eqs:energy-inequality-sigma}
    \end{align}
In \eqref{eqs:energy-inequality-sigma}, the approximate energy $E_{\sigma,\delta}$ and its initial value are given by \eqref{appro-energy}, \eqref{appro-energy-0}, respectively, while the constant $C_*>0$ is the same as in \eqref{low-bd-F1}.
\end{proposition}
 
\subsection{The Faedo--Galerkin approximation}
Motivated by  \cite{AS2022,FNP2001,MSZ2025,VY2016}, we apply a suitable Faedo--Galerkin approximation scheme. For $N \in \mathbb{Z}^+$, we introduce the finite-dimensional space $X_N = \mathrm{span} \left\{\be_1, \be_2, \dots, \be_N\right\}$, where  $\{\be_i\}\subset C^\infty(\mathbb{T}^3)$ forms an orthonormal basis of $L^2(\bbt^3)$ as well as an orthogonal basis of $H^2(\bbt^3)$. A typical choice could be eigenfunctions of the (minus) Laplacian on $\mathbb{T}^3$.  
Define 
$$\bu_N = \sum_{i=1}^N \lambda_i(t) \be_i(x)\quad \text{and}\quad \bu_{0,N} = \sum_{i=1}^N (\bu_0,\mathbf{e}_i)_{L^2(\mathbb{T}^3)} \be_i(x)
$$ 
for some functions $\lambda_i(t)$ ($i=1,...,N$) that are continuous in $[0,T]$. Then $\bu_N \in C([0,T];X_N)$, which also implies $\bu_N \in C([0,T];C^k(\bbt^3))$ for any $k\in \mathbb{N}$. The Faedo--Galerkin approximation is given by 
\begin{align}
	\nonumber
	&  \int_{\bbt^3} \partial_t(\rho_N \bu_N) \cdot \bphi \,\dx -\int_{\bbt^3} \rho_N \left(
	 \bu_N \otimes \bu_N
	- \bbd \bu_N \right) : \nabla \bphi \,\dx \\
	\nonumber
	&\quad  = 
	 \int_{\bbt^3} \Big(\widetilde{P}_\sigma(\rho_N)
	+ \delta H_\delta(\rho_N)
	\Big) \Div \bphi \,\dx 
    - \delta \int_{\bbt^3}  \rho_N \nabla \Delta^2 \rho_N  \cdot \bphi \,\dx \\
	\nonumber
	& \qquad 
	- \delta \int_{\bbt^3}  (\nabla \rho_N \cdot \nabla \bu_N )\cdot \bphi \,\dx 
    - \delta \int_{\bbt^3}  \Delta \bu_N \cdot \Delta \bphi \,\dx 
	  \\
	& \qquad - \zeta^2 \int_{\bbt^3}  (\rho_N \Delta \rho_N) \Div \bphi \,\dx
	- \zeta^2 \int_{\bbt^3} \Delta \rho_N(\nabla \rho_N  \cdot \bphi) \,\dx
	\notag \\
    &\qquad + \frac{1}{\alpha^2}  \int_{\bbt^3} (\Delta^{-1} \Div \bu_N) \Div \bphi \,\dx,\quad \forall\, t\in (0,T),\ \bphi\in X_N,
	\label{eqs:weak-formulation-ddt-N}
\end{align}
where $\rho_N$ satisfies 
\begin{align}
 \label{eqs:app-mass-P-sigma-delta-N}
		& \pt \rho_N + \Div (\rho_N \bu_N) = \delta \Delta \rho_N,\quad \text{a.e.~in }\ Q_T,
 \end{align}
subject to the initial datum $\rho_{0,\delta}$. 

First, for any given $\bu \in C([0,T];X_N)$, thanks to the classical theory of parabolic equations (cf. e.g., \cite[Section 2]{FNP2001}), equation \eqref{eqs:app-mass-P-sigma} subject to the initial datum $\rho_{0,\delta}$ (see \eqref{eqs:initial data}) admits a unique classical solution $\rho \in C^1([0,T];C^{k}(\bbt^3)) \cap C([0,T];C^{k+2}(\bbt^3))$ for any $k\in \mathbb{N}$ such that 
\begin{align*}
\norm{\rho}_{C([0,T];C^{k+2}(\bbt^3))}
	+ \norm{\pt \rho}_{C([0,T];C^{k}(\bbt^3))}
	\leq C(\delta,k) \big(\norm{\rho_{0,\delta}}_{C^{2+k}(\bbt^3)} + \norm{\Div(\rho \bu)}_{C([0,T];C^{k}(\bbt^3))}\big),
\end{align*}
and for all $(x,t) \in \mathbb{T}^3\times [0,T]$, 
\begin{align}
	\inf_{y \in \bbt^3} \rho_{0,\delta}(y) e^{- \int_0^T\norm{\Div \bu(\tau)}_{L^\infty(\bbt^3)}\,\mathrm{d}\tau} \leq \rho(x,t) \leq \sup_{y \in \bbt^3} \rho_{0,\delta}(y) e^{\int_0^T\norm{\Div \bu(\tau)}_{L^\infty(\bbt^3)}\,\mathrm{d}\tau},
    \label{rho-UL}
\end{align}
Thus, $0 < \vartheta \leq \rho(x,t) \leq \vartheta^{-1} < \infty $ holds in $\mathbb{T}^3\times [0,T]$ for some constant $\vartheta=\vartheta(\underline{\rho},N,\delta,\sigma) > 0$.  
The above well-posedness result allows us to introduce a linear continuous operator
\begin{align*}
	\cS: C([0,T];X_N) & \to C([0,T];C^{k+2}(\bbt^3)),\\
    \mathbf{u} & \mapsto \rho, 
\end{align*}
where $\rho=\cS(\bu)$ is the unique classical solution of \eqref{eqs:app-mass-P-sigma} subject to the given velocity field $\mathbf{u}$ and the initial condition  $\rho_N|_{t=0}=\rho_{0,\delta}$. 
Moreover, it holds   
\begin{align*}
	\norm{\cS(\bu_1) - \cS(\bu_2)}_{C([0,T];H^1(\bbt^3))}
	\leq C \norm{\bu_1 - \bu_2}_{C([0,T];H^1(\bbt^3))},
\end{align*}
for any $\mathbf{u}_1,\mathbf{u}_2\in C([0,T];X_N)$, where the constant $C>0$ depends on $T$, $\|\mathbf{u}_i\|_{L^\infty(0,T;W^{2,\infty}(\mathbb{T}^3))}$, $i=1,2$ (see \cite[Lemma 2.2]{FNP2001}). Continuous dependence estimates with higher-order (resp. lower-order) norms on the left-hand (resp. right-hand) side can be easily obtained using the fact that $X_N$ is finite dimensional. 

Next, we define the set 
\begin{align*}
	\cB = \left\{
	\bu \in C([0,T^*];X_N): \norm{\bu_N - \bu_{0,N}}_{C([0,T^*];L^{2}(\mathbb{T}^3))} \leq 1
	\right\}
\end{align*}
for some $T^* \in (0, T]$ that will be chosen later. Consider the operator $\cT: \cB \to C([0,T^*];X_N)$ given by 
\begin{align*}
	\cT(\bu_N) \coloneqq \cM[\cS(\bu_N)(t)]^{-1} \Big(
	\cM[\rho_{0,\delta}](\bu_{0,N}) + \int_0^t \cN(\cS(\bu_N), \bu_N)(\tau) \,\d \tau
	\Big). 
\end{align*}
Here, $\cM[\rho] : X_N \to X_N^*$ is defined as
\begin{align*}
	\inner{\cM[\rho](\bu)}{\bphi}_{X_N^*,X_N} = \int_{\bbt^3} \rho \bu \cdot \bphi \,\dx, \quad \forall\, \bu, \bphi \in X_N,
\end{align*}
and $\mathcal{N}(\rho,\mathbf{u})$ satisfies  
\begin{align*}
	&\inner{\cN(\rho, \bu)}{\bphi}_{X_N^*,X_N}\\
	&\quad  \coloneqq \int_{\bbt^3} \Big(
	- \Div(\rho \bu \otimes \bu)
	+ \Div (\rho \bbd \bu)
	- \nabla \widetilde{P}_\sigma(\rho)
	- \delta \rho \nabla W_\delta'(\rho)
	\Big) \cdot \bphi \,\dx 
    \\
	& \qquad\, + \int_{\bbt^3} \Big(
	- \delta \rho \nabla \Delta^2 \rho
	- \delta \nabla \rho \cdot \nabla \bu
	-\delta \Delta^2 \bu
	+ \zeta^2 \rho \nabla \Delta \rho
	- \frac{1}{\alpha^2} \nabla \Delta^{-1} \Div \bu
	\Big) \cdot \bphi \,\dx
\end{align*}
for any $\mathbf{u},\bphi \in X_N$.
For $\rho$ (resp. $\rho_1$, $\rho_2$) in the class $\big\{\rho\in L^1(\mathbb{T}^3)\ |\ \inf_{x\in \mathbb{T}^3} \rho(x) \geq \vartheta_N\big\}$
with some $\vartheta_N>0$, the operator  $\cM[\rho]$ is invertible such that  
\begin{align*}
	\norm{\cM[\rho]^{-1}}_{\cL(X_N^*, X_N)} \leq  \frac{1}{\vartheta_N},
\end{align*}
and
\begin{align}
	\label{eqs:M-Lipschitz}
	\norm{\cM[\rho_1]^{-1} - \cM[\rho_2]^{-1}}_{\cL(X_N^*, X_N)} \leq C(N,\vartheta_N) \norm{\rho_1 - \rho_2}_{L^1(\bbt^3)}.
\end{align}

For $T^*=T^*(N)>0$ sufficiently small, it is straightforward to check that $\cT(\cB) \subset \cB$. Thanks to the estimate on $\partial_t \mathcal{S}(\mathbf{u}_N)$, the continuity of $\mathcal{S}$ and \eqref{eqs:M-Lipschitz}, we find that $\cT(\cB)$ is compact due to the Arzel\`{a}--Ascoli theorem. These observations, together with the continuity of $\mathcal{T}$ on $\mathcal{B}$, allow us to apply Schauder's fixed point theorem. This yields a local-in-time solution $\mathbf{u}_N \in \mathcal{B}$ to the following integral equation, provided that $T^*$ is sufficiently small (cf. \cite{AS2022,MSZ2025,VY2016}):  
\begin{align}
	\bu_N(t) = \cM[\cS(\bu_N)(t)]^{-1} \Big(
	\cM[\rho_{0,\delta}](\bu_{0,N}) + \int_0^t \cN(\cS(\bu_N), \bu_N)(\tau) \,\d \tau
	\Big),\quad \forall\, t\in [0,T^*]. 
    \label{F-G}
\end{align}
Noticing that the operator $\cN$ is continuous in time and thus $t \mapsto \int_0^t \cN(\rho_N, \bu_N)(\tau) \,\d \tau$ is differentiable,  
we also have $\partial_t \bu_N \in C([0,T^*];X_N)$. Since $\mathbf{u}_N$ is sufficiently regular, its uniqueness can easily be verified using the energy method.    

In summary, for every $N\in \mathbb{Z}^+$, we have shown the existence and uniqueness of a local-in-time classical solution 
$$
(\mathbf{u}_N,\rho_N)=(\mathbf{u}_N,\mathcal{S}(\mathbf{u}_N))\quad \text{on}\quad [0,T^*]
$$ 
to the Faedo--Galerkin approximation scheme for problem \eqref{eqs:app-P-sigma} that is equivalently given by \eqref{F-G}. Applying $\cM[\rho_N(t)]$ to both sides of \eqref{F-G}, differentiating the resultant in time and then testing it by $\bphi\in X_N$, we recover the weak formulation \eqref{eqs:weak-formulation-ddt-N} on $(0,T^*)$ with the aid of \eqref{eqs:app-mass-P-sigma-delta-N}. 

\subsection{Proof of Proposition \ref{prop:weak-solution-sigma}}
The first step is to show that $T^*(N)=T$ for every $N\in \mathbb{Z}^+$. For this purpose, we derive uniform estimates for $\mathbf{u}_N$ on $[0,T^*]$ that may depend on $T$ but not on $T^*$. Taking $\bphi = \bu_N$ in \eqref{eqs:weak-formulation-ddt-N} and using \eqref{eqs:app-mass-P-sigma-delta-N}, we obtain the following energy equality
\begin{align}
	\nonumber
	& \frac{\d}{\dt} E_{\sigma,\delta}(\bu_N,\rho_N)
	  + \int_{\bbt^3} \rho_N \absm{\bbd \bu_N}^2 \,\dx 
    + \delta \int_{\bbt^3} \absm{\Delta \bu_N}^2 \,\dx
	+ \frac{1}{\alpha^2} \int_{\bbt^3} \absm{\nabla \Delta^{-1} \Div \bu_N}^2 \,\dx \\
	\nonumber
	&\qquad  
	+ \delta^2 \int_{\bbt^3} \absm{\nabla \Delta \rho_N}^2 \,\dx 
     + \delta \zeta^2 \int_{\bbt^3} \absm{\Delta \rho_N}^2 \,\dx
	\\
	&\quad 
	=- \delta \int_{\bbt^3} \big(\widetilde{F}_\sigma''(\rho_N) + \delta W_\delta''(\rho_N)\big) \absm{\nabla \rho_N}^2 \,\dx
    \notag 
    \\
    &\quad \leq \delta \omega \zeta^2 \int_{\bbt^3} \absm{\nabla \rho_N}^2 \,\dx,
    \quad \forall\, t\in (0,T^*),
	\label{eqs:energy-N}
\end{align}
where
\begin{align*}
	E_{\sigma,\delta}(\bu_N,\rho_N) \coloneqq
	\int_{\bbt^3} \Big(
	\onehalf \rho_N \absm{\bu_N}^2
	+ \frac{\zeta^2}{2} \absm{\nabla {\rho_N}}^2
	+ \widetilde{F}_\sigma(\rho_N)
    + \frac{\delta}{2} \absm{\Delta \rho_N}^2 
	+ \delta W_\delta(\rho_N)
	\Big)\,\dx.
\end{align*}
In the derivation of \eqref{eqs:energy-N}, we have used the facts 
$$
\widetilde{F}_\sigma''(\rho_N) = \widetilde{F}_{c,\sigma}''(\rho_N) - \omega \zeta^2  > - \omega \zeta^2 \quad \text{and}\quad W_\delta''(\rho_N) \geq 0.
$$
Hence, an application of Gronwall's lemma to \eqref{eqs:energy-N} yields 
\begin{align}
\sup_{0 \leq t \leq T^*} \big(E_{\sigma,\delta}(\bu_N(t),\rho_N(t)) + C_*\big) 
\leq \big(E_{\sigma,\delta}(\bu_{0,N},\rho_{0,\delta})+C_*\big) e^{2 \delta \omega T^*}
\label{eqs:energy-inequality-N-a}
\end{align}
and then 
\begin{align}
	\nonumber
	  & \sup_{0 \leq t \leq T^*} \big(E_{\sigma,\delta}(\bu_N(t),\rho_N(t)) + C_*\big)  + \int_0^{T^*} \int_{\bbt^3} \rho_N \absm{\bbd \bu_N}^2 \,\dxdt  
    \\
	\nonumber
	& \qquad + \delta \int_0^{T^*} \int_{\bbt^3} \absm{\Delta \bu_N}^2 \,\dxdt + \frac{1}{\alpha^2} \int_0^{T^*} \int_{\bbt^3} \absm{\nabla \Delta^{-1} \Div\bu_N}^2 \,\dxdt 
    \\
    \nonumber 
	&\qquad + \delta^2 \int_0^{T^*} \int_{\bbt^3} \absm{\nabla \Delta \rho_N}^2 \,\dxdt
	+ \delta \zeta^2 \int_0^{T^*} \int_{\bbt^3} \absm{\Delta \rho_N}^2 \,\dxdt \\
	&\quad  \leq E_{\sigma,\delta}(\bu_{0,N},\rho_{0,\delta}) + \big(E_{\sigma,\delta}(\bu_{0,N},\rho_{0,\delta})+C_*\big) 2\delta \omega T^* e^{2 \delta \omega T^*}.
	\label{eqs:energy-inequality-N}
\end{align}
The right-hand side of \eqref{eqs:energy-inequality-N-a}, \eqref{eqs:energy-inequality-N} can be controlled by a positive constant depending on $E_{\sigma,\delta}(\bu_{0},\rho_{0,\delta})$, $C_*$ and $T$, which is independent of $T^*$. Recalling that $X_N$ is finite dimensional, from \eqref{rho-UL} and \eqref{eqs:energy-inequality-N}, we can deduce that $0 < \vartheta \leq \rho_N(x,t) \leq \vartheta^{-1}$ in $\mathbb{T}^3\times [0,T^*]$ for some  $\vartheta=\vartheta(\underline{\rho},N,\delta,\sigma,T,C_*) > 0$. Combining these bounds of $\rho_N$ and uniform-in-time estimates of \eqref{eqs:energy-inequality-N}, for each fixed $N\in \mathbb{Z}^+$, we can extend the local solution $(\mathbf{u}_N,\rho_N)$ to the entire interval $[0, T]$. 

In the following, we study the limit as $N\to \infty$. At this stage, we cannot take advantage of the finite dimensionality of $X_N$. Nevertheless, the estimate \eqref{eqs:energy-inequality-N} is valid with $T^*$ replaced by $T$, and the right-hand side is bounded from above by $K:=\big(E_{\sigma,\delta}(\bu_{0},\rho_{0,\delta})+C_*\big)\left(1+2\delta \omega T e^{2\delta\omega T}\right)+1$ whenever $N$ is sufficiently large, since $\lim_{N\to \infty}\|\mathbf{u}_{0,N}-\mathbf{u}_0\|_{L^2(\mathbb{T}^3)}=0$ and $\underline{\rho}< \rho_{0,\delta}< 1$ in $\mathbb{T}^3$. In particular, we have 
$$
\sup_{0\leq t\leq T} \int_{\bbt^3} \Big(
	  \frac{\delta}{2} \absm{\Delta \rho_N}^2 
	+ \delta W_\delta(\rho_N)
	\Big)\,\dx \leq K,\quad \text{for}\ \ N\gg 1.
$$
Keeping in mind the mass conservation
$$
\int_{\mathbb{T}^3} \rho_N(t)\,\mathrm{d}x = \int_{\mathbb{T}^3} \rho_{0,\delta}\,\mathrm{d}x,
$$
and the aforementioned construction of $W_\delta$, we can apply Lemma \ref{lower-and-up-bounds-for-rho-sigma} to conclude the uniform upper and lower bounds of $\rho_N$: 
\begin{align}
\rho_N(x,t) \in \Big(\frac{1}{2}\underline{\rho},\ 1+\frac{1}{2}\underline{\rho}\Big),\quad \text{for all}\ \ (x,t)\in \mathbb{T}^3\times [0,T].
\label{L-inf-rhoN}
\end{align}
From \eqref{eqs:energy-inequality-N} and \eqref{L-inf-rhoN}, we obtain the following regularity properties of $(\mathbf{u}_N,\rho_N)$ with uniform estimates with respect to $N$ in the corresponding spaces:
\begin{align*}
& \sqrt{\rho_N}\mathbf{u}_N, \  \mathbf{u}_N\in L^\infty(0,T;L^2(\mathbb{T}^3)),  
\qquad 
\sqrt{\rho_N}\mathbb{D} \mathbf{u}_N,\  \mathbb{D} \mathbf{u}_N,\  \delta^\frac12 \Delta \mathbf{u}_N\in L^2(0,T;L^2(\mathbb{T}^3)),
\\
& \zeta\nabla \rho_N,\,\delta^\frac12\Delta \rho_N\in L^\infty(0,T;L^2(\mathbb{T}^3)),
\qquad \delta \nabla \Delta \rho_N,\  \delta^\frac12\zeta \Delta \rho_N\in L^2(0,T;L^2(\mathbb{T}^3)),
\\
& \widetilde{F}_\sigma(\rho_N)+C_*,\ \delta W_\delta(\rho_N)\in L^\infty(0,T;L^1(\mathbb{T}^3)),
\qquad 
\widetilde{P}_\sigma(\rho_N),\ 
	\delta H_\delta(\rho_N) \in L^\infty(Q_T).
\end{align*}

Next, we derive estimates for time derivatives. From \eqref{eqs:app-mass-P-sigma-delta-N}, we find 
\begin{align}
\|\partial_t\rho_N\|_{H^1(\mathbb{T}^3)}
& \lesssim \|\rho_N\|_{H^2(\mathbb{T}^3)}\|\mathbf{u}_N\|_{H^2(\mathbb{T}^3)}+\delta\|\nabla \Delta \rho_N\|_{L^2(\mathbb{T}^3)},
\label{es-rt-N}
\end{align}
which yields that 
$$
\partial_t\rho_N\ \ \text{is uniformly bounded in} \ \ L^2(0,T;H^1(\mathbb{T}^3)). 
$$
Next, from \eqref{eqs:weak-formulation-ddt-N}, we obtain  
\begin{align}
& \left|\big\langle \partial_t(\rho_N\mathbf{u}_N),\bphi\big\rangle_{(H^2(\mathbb{T}^3))^*,H^2(\mathbb{T}^3)}\right| 
\notag 
\\
   &\quad \lesssim \|\rho_N\|_{L^\infty(\mathbb{T}^3)} \|\mathbf{u}_{N}\|_{L^\infty(\mathbb{T}^3)} \|\mathbf{u}_{N}\|_{L^2(\mathbb{T}^3)} \|\nabla \bphi\|_{L^2(\mathbb{T}^3)} 
+  \|\rho_N\|_{L^\infty(\mathbb{T}^3)} \|\mathbb{D} \mathbf{u}_{N}\|_{L^2(\mathbb{T}^3)}\|\nabla \bphi\|_{L^2(\mathbb{T}^3)}
\notag \\
    &\qquad + \|\widetilde{P}_\sigma(\rho_N)
	+ \delta H_\delta(\rho_N)\|_{L^2(\mathbb{T}^3)}\|\Div \bphi\|_{L^2(\mathbb{T}^3)}
    +  \delta \| \nabla^2\rho_N\|_{L^2(\mathbb{T}^3)} \| \nabla\Delta\rho_N\|_{L^2(\mathbb{T}^3)} \| \bphi\|_{L^\infty(\mathbb{T}^3)}
    \notag \\
	&\qquad  +  \delta \| \nabla\Delta\rho_N\|_{L^2(\mathbb{T}^3)} 
    \| \nabla\rho_N\|_{L^6(\mathbb{T}^3)}  \| \bphi\|_{W^{1,3}(\mathbb{T}^3)} 
     + \delta \|\rho_N\|_{L^\infty(\mathbb{T}^3)}  \| \nabla\Delta \rho_N\|_{L^2(\mathbb{T}^3)} \|\bphi\|_{H^2(\mathbb{T}^3)}
      \notag   \\  
     & \qquad 
     + \delta \|\nabla \rho_N \|_{L^2(\mathbb{T}^3)} \|\nabla \bu_N \|_{L^2(\mathbb{T}^3)} \| \bphi\|_{L^\infty(\mathbb{T}^3)} 
     + \delta \|\Delta \bu_N\|_{L^2(\mathbb{T}^3)} \|\Delta \bphi \|_{L^2(\mathbb{T}^3)} 
	   \notag \\
	& \qquad + \zeta^2 \|\rho_N\|_{L^\infty(\mathbb{T}^3)} \|\Delta \rho_N\|_{L^2(\mathbb{T}^3)} 
    \|\Div \bphi\|_{L^2(\mathbb{T}^3)}
    + \zeta^2 \|\Delta \rho_N\|_{L^2(\mathbb{T}^3)}\|\nabla \rho_N\|_{L^2(\mathbb{T}^3)}\|\bphi\|_{L^\infty(\mathbb{T}^3)}  
	\notag \\
    &\qquad + \alpha^{-2} \|\bu_N\|_{L^2(\mathbb{T}^3)} \|\Div \bphi\|_{L^2(\mathbb{T}^3)},\quad \forall\, \bphi\in H^2(\mathbb{T}^3),
    \label{es-rut-N}
\end{align}
which together with the estimates on $(\mathbf{u}_N,\rho_N)$ implies 
$$
\partial_t(\rho_N \mathbf{u}_N)\ \ \text{is uniformly bounded in} \ \ L^2(0,T;(H^2(\mathbb{T}^3))^*). 
$$

Based on the above uniform estimates, we can extract a convergent subsequence ${(\mathbf{u}_N,\rho_N)}$ (which we do not relabel for simplicity) such that, as $N \to \infty$, its limit $(\mathbf{u},\rho)$ is a weak solution to the problem \eqref{eqs:app-P-sigma} in $[0,T]$ with the required properties. The compactness argument is similar to that in, e.g., \cite{FNP2001,VY2016}, and thus is omitted here. Moreover, by the weak lower semicontinuity of convex functions and norms, the limit functions $(\bu,\rho)$ satisfy the energy estimate \eqref{eqs:energy-inequality-sigma}. Finally, the mass relation \eqref{conservation law-mass-sd} easily follows from \eqref{eqs:app-mass-P-sigma-delta}.

The proof of Proposition \ref{prop:weak-solution-sigma} is complete. 
\hfill $\square$

\begin{remark}\rm 
    In the above argument, we do not need a Mellet--Vasseur-type inequality for the renormalization of the momentum equation because of the force term $- {\alpha}^{-2} \nabla \Delta^{-1} \Div \bu $, which serves as a linear damping term, cf.~\cite{MSZ2025,VY2016InventMath}.
\end{remark}

\section{Passage to the limit as $\sigma \to 0$} 
\label{sec:limit-sigma}

In this section, we investigate the limit as $\sigma \to 0$, while keeping $0<\delta\ll 1$ fixed. 
Formally, the corresponding limit system of \eqref{eqs:app-P-sigma} reads as follows 
\begin{subequations}
	\label{eqs:app-delta}
	\begin{align}
			& \partial_t (\rho \bu) 
			+ \Div(\rho \bu \otimes \bu) 
			+ \nabla \widetilde{P}(\rho) 
			+ \delta \rho\nabla W'(\rho)
			- \Div (\rho \bbd \bu) 
            \notag \\
			& 
			\quad = 
			- \delta \rho \nabla \Delta^2\rho
			- \delta \nabla \rho \cdot \nabla \bu
			- \delta \Delta^2 \bu
			+ \zeta^2 \rho \nabla \Delta \rho
			- \frac{1}{\alpha^2} \nabla \Delta^{-1} \Div \bu, \quad \text{in}\ \ Q_T,
            \label{eqs:app-momentum-delta}
		\\
		\label{eqs:app-mass-delta}
		 &\pt \rho + \Div (\rho \bu) = \delta \Delta \rho,\quad \text{in}\ \ Q_T,
      \\
      & \mathbf{u}|_{t=0} =\mathbf{u}_{0},\quad  \rho|_{t=0}=\rho_{0,\delta},\quad \text{in}\ \ \mathbb{T}^3.
        \label{eqs:initial data-delta}
	\end{align}
\end{subequations}
\begin{remark} \rm 
In \eqref{eqs:app-momentum-delta}, since the density function $\rho$ will be shown to satisfy $\rho(x,t) \in \big(\underline{\rho}, 1\big)$ almost everywhere in $Q_T$ (see Proposition \ref{prop:weak-solution-delta} below), it follows that $W_\delta(\rho)=W(\rho)$. Next, due to the singular behavior of $\widetilde{F}_c$ near $\underline{\rho}$ and $1$, we require the initial density $\rho_0$ to satisfy $  \widetilde{F}_c(\rho_0)\in L^1(\mathbb{T}^3)$ (cf. also Remark \ref{rem:ini}), which is equivalent to $\int_{\mathbb{T}^3} \widetilde{F}_c(\rho_0)\,\mathrm{d}x<\infty$ in view of \eqref{low-bd-F1}. On the other hand, this condition implies that $\rho_0 \in(\underline{\rho},1)$ almost everywhere in $\mathbb{T}^3$, thereby ensuring the validity of Proposition \ref{prop:weak-solution-sigma}. For the approximate initial density $\rho_{0,\delta}$ in \eqref{eqs:initial data-delta}, the convexity of $\widetilde{F}_c$ together with Jensen's inequality yields
$$
\widetilde{F}_c(\rho_{0,\delta}(x))\leq 
\int_{\mathbb{T}^3} \widetilde{F}_c(\rho_0(y))\widetilde{\eta}_\delta^{\mathrm{per}}(x-y)\,\mathrm{d}y,\quad \forall\,x\in \mathbb{T}^3.
$$
Without loss of generality, we can assume $\widetilde{F}_c(r)\geq 0$ for $r\in \big(\underline{\rho}, 1\big)$. Then it follows from Fubini's theorem that 
\begin{align}
\int_{\mathbb{T}^3} \widetilde{F}_c(\rho_{0,\delta}(x))\,\mathrm{d}x 
& \leq \int_{\mathbb{T}^3}   \int_{\mathbb{T}^3} \widetilde{F}_c(\rho_0(y))\widetilde{\eta}_\delta^{\mathrm{per}}(x-y)\,\mathrm{d}y\mathrm{d}x
\notag \\
&=\int_{\mathbb{T}^3} \widetilde{F}_c(\rho_0(y))  \Big(\int_{\mathbb{T}^3} \widetilde{\eta}_\delta^{\mathrm{per}}(x-y)\,\mathrm{d}x \Big)\mathrm{d}y
= \int_{\mathbb{T}^3} \widetilde{F}_c(\rho_{0}(x))\,\mathrm{d}x.
\label{Fc-delta}
\end{align}
From \eqref{Fc-up-sigma} and \eqref{Fc-delta}, we further get
\begin{align}
\int_{\mathbb{T}^3} \widetilde{F}_{c,\sigma}(\rho_{0,\delta}(x))\,\mathrm{d}x
\leq \int_{\mathbb{T}^3} \widetilde{F}_c(\rho_{0}(x))\,\mathrm{d}x,
\quad \text{for all}\ \ 0<\sigma, \delta\ll 1.\label{Fc-delta-b}
\end{align}
Because the concave part of $\widetilde{F}$ consists only of a quadratic polynomial, for $\rho_0\in H^1(\mathbb{T}^3)$, we can alternatively assume that $\widetilde{F}(\rho_0)\in L^1(\mathbb{T}^3)$.
\end{remark} 

The following proposition establishes the existence of global weak solutions to problem \eqref{eqs:app-delta}.
\begin{proposition}
\label{prop:weak-solution-delta} 
Suppose that Assumption \ref{main assumption} is satisfied and $\widetilde{F}$, $\widetilde{P}$ are defined as in \eqref{eqs:resuced-F}, \eqref{eqs:new-entropy}, respectively.
Let $0<\delta\ll 1$ be arbitrarily small but fixed, and let $T\in (0,\infty)$. Then for any initial data satisfying $(\bu_{0},\rho_{0}) \in L^2(\bbt^3) \times H^1(\bbt^3)$ with $\widetilde{F}(\rho_0)\in L^1(\mathbb{T}^3)$, problem \eqref{eqs:app-delta} admits a global weak solution $(\mathbf{u},\rho)$ on $[0,T]$ such that 
\begin{align}
& \bu\in L^\infty(0,T;L^2(\bbt^3)) \cap L^2(0,T;H^2(\bbt^3)), \notag \\
& \rho\in BC([0,T];H^2(\bbt^3))\cap  L^2(0,T;H^3(\bbt^3))\cap H^1(0,T;H^1(\mathbb{T}^3)), \notag \\
& \rho\in L^\infty(Q_T) \ \ \text{with}\ \ \rho \in \big(\underline{\rho}, 1\big) \  \text{ a.e. in } Q_T. 
\label{rho-UpLo-delta}
\end{align} 
The solution $(\mathbf{u},\rho)$ satisfies 
\begin{align}
 \label{eqs:app-mass-sigma1}
		 &\pt \rho + \Div (\rho \bu) = \delta \Delta \rho, \quad \text{a.e.~in }\ Q_T, 
\end{align} 
with $\rho|_{t=0}=\rho_{0,\delta}$ almost everywhere in $\mathbb{T}^3$, and the following weak formulation%
\begin{align}
	\nonumber
	& -\int_{\mathbb{T}^3} \rho_{0,\delta}\mathbf{u}_0\cdot\bphi(\cdot,0)\,\dx
	- \int_0^T \int_{\bbt^3} \rho \bu \cdot \pt \bphi \,\dxdt \\
	\nonumber
	& \qquad - \int_0^T \int_{\bbt^3} 
	\rho \bu \otimes \bu : \nabla \bphi \,\dxdt
    + \int_0^T \int_{\bbt^3} \rho \bbd \bu : \bbd \bphi \,\dxdt \\
	\nonumber
	& \quad = 
	 \int_0^T \int_{\bbt^3} \Big(\widetilde{P}(\rho)
	+ \delta H(\rho)
	\Big) \Div \bphi \,\dxdt - \delta\int_0^T \int_{\bbt^3}  \rho \nabla \Delta^2 \rho  \cdot \bphi \,\dxdt 
    \\
	\nonumber
	& \qquad 
    - \delta\int_0^T \int_{\bbt^3}  (\nabla \rho \cdot \nabla \bu) \cdot \bphi \,\dxdt 
	-\delta  \int_0^T \int_{\bbt^3} \Delta \bu \cdot \Delta \bphi \,\dxdt 
	  \\
	\nonumber
	& \qquad - \zeta^2\int_0^T \int_{\bbt^3}  (\rho \Delta \rho ) \Div \bphi \,\dxdt
	-\zeta^2 \int_0^T \int_{\bbt^3}  \Delta \rho  (\nabla \rho \cdot \bphi) \,\dxdt
	\\
	& \qquad + \frac{1}{\alpha^2}\int_0^T \int_{\bbt^3}  (\Delta^{-1} \Div \bu) \Div \bphi \,\dxdt,
    \quad \forall\, \boldsymbol{\varphi}\in C_{c}^{\infty}([0,T); C^{\infty}(\mathbb{T}^{3})).
    \label{eqs:weak-formulation-delta}
\end{align}
Here, the nonlinear function $H$ is defined as in \eqref{H def} and the term $\int_0^T \int_{\bbt^3} \rho \nabla \Delta^2 \rho  \cdot \bphi \,\dxdt$ should be understood as in \eqref{higher-terms}. Moreover, the following energy estimate holds
\begin{align}
	\nonumber
	& \sup_{0 \leq t \leq T} E_\delta(\bu(t),\rho(t))
	  + \int_0^T \int_{\bbt^3} \rho \absm{\bbd \bu}^2 \,\dxdt
	+ \frac{1}{\alpha^2} \int_0^T \int_{\bbt^3} \absm{\nabla \Delta^{-1}\Div \bu}^2 \,\dxdt 
    \\
    \nonumber 
    &\qquad + \delta \int_0^T \int_{\bbt^3} \absm{\Delta \bu}^2 \,\dxdt
    + \delta^2 \int_0^T \int_{\bbt^3} \absm{\nabla \Delta \rho}^2 \,\dxdt
    + \delta \zeta^2 \int_0^T \int_{\bbt^3} \absm{\Delta \rho}^2 \,\dxdt
    \\
	&  
	\quad \leq  E_\delta(\bu_{0},\rho_{0,\delta}) 
    + \big(E_\delta(\bu_{0},\rho_{0,\delta})+ C_*\big) 
     2\delta \omega T e^{2\omega\delta T},
	\label{eqs:energy-inequality-delta}
\end{align}
where $E_\delta$ is the approximate energy defined by
\begin{align}
	E_\delta(\bu,\rho) \coloneqq
	\int_{\bbt^3} \Big(
	\onehalf \rho \absm{\bu}^2
	+ \frac{\zeta^2}{2} \absm{\nabla \rho}^2
	+ \widetilde{F}(\rho)
    + \frac{\delta}{2} \absm{ \Delta \rho}^2
	+ \delta W(\rho)
	\Big) \,\dx.\label{energy-delta}
\end{align}
In addition, we have the conservation of mass 
	\begin{align}
		\int_{\bbt^3} \rho(\cdot,t) \,\dx
		= \int_{\bbt^3} \rho_{0,\delta} \,\dx,
        \quad \forall\, t\in[0,T].
        \label{conservation law-000}
	\end{align}
\end{proposition}

To prove Proposition \ref{prop:weak-solution-delta}, we construct approximate solutions using Proposition \ref{prop:weak-solution-sigma}. For simplicity, we denote these approximate solutions by $(\bu_\sigma,\rho_\sigma)$, omitting the fixed parameter $\delta$. After deriving the necessary estimates for $(\bu_\sigma,\rho_\sigma)$, which are uniform in $\sigma$, we pass to the limit as $\sigma \to 0$.

\subsection{Energy estimates and mass conservation}
From \eqref{appro-energy}, \eqref{appro-energy-0}, \eqref{H def-d}, \eqref{UL-rho-sigma}, \eqref{eqs:energy-inequality-sigma} and \eqref{Fc-delta-b}, 
we obtain the following regularity properties of $(\mathbf{u}_\sigma, \rho_\sigma)$, which hold with uniform bounds in the corresponding spaces for $0<\sigma\ll 1$:
\begin{align}
& \sqrt{\rho_\sigma}\mathbf{u}_\sigma, \  \mathbf{u}_\sigma\in L^\infty(0,T;L^2(\mathbb{T}^3)),  
\qquad 
\sqrt{\rho_\sigma}\mathbb{D} \mathbf{u}_\sigma,\  \mathbb{D} \mathbf{u}_\sigma,\  \delta^\frac12 \Delta \mathbf{u}_\sigma\in L^2(0,T;L^2(\mathbb{T}^3)),
\label{eqs:uniform-u-sigma}
\\
& \zeta\nabla \rho_\sigma,\,\delta^\frac12\Delta \rho_\sigma\in L^\infty(0,T;L^2(\mathbb{T}^3)),
\qquad \delta \nabla \Delta \rho_\sigma,\  \delta^\frac12\zeta \Delta \rho_\sigma\in L^2(0,T;L^2(\mathbb{T}^3)),
\label{eqs:uniform-rho-sigma}
\\
& \widetilde{F}_\sigma(\rho_\sigma),\ \delta W_\delta(\rho_\sigma)\in L^\infty(0,T;L^1(\mathbb{T}^3)),\quad \delta H_\delta(\rho_\sigma)\in L^\infty(0,T;L^\infty(\Omega)).
\label{eqs:uniform-FWH-sigma}
\end{align}
In addition, we have the conservation of mass (see \eqref{conservation law-mass-sd})
	\begin{align}
		\label{eqs:conservation-mass-delta}
		\mean{\rho_\sigma(\cdot,t)} = \mean{\rho_{0,\delta}} \in \big(\underline{\rho}, 1\big),\quad \forall\, t\in[0,T].
	\end{align}
Using the above estimates, similarly to \eqref{es-rt-N}, we can show that  
\begin{align}
\partial_t\rho_\sigma\ \ \text{is uniformly bounded in} \ \ L^2(0,T;H^1(\mathbb{T}^3)).
\label{eqs:uniform-rhot-sigma}
\end{align}
The estimate for $\partial_t(\rho_\sigma\mathbf{u}_\sigma)$ is not available at this stage, since the pressure  $\widetilde{P}_\sigma(\rho_\sigma)$ must be treated in a different manner (cf. \eqref{es-rut-N}).    
  
\subsection{Integrability of the pressure}
We prove the uniform integrability of $\widetilde{P}_\sigma(\rho_\sigma)$ using an argument inspired by \cite{FP2000, FNP2001}. To this end, let us introduce the so-called Bogovski\u{\i} operator 
$$ \cB =\nabla \Delta^{-1}.$$ 
For any $g \in L^p(\bbt^3)$ with $\mean{g} = 0$, $1 < p <\infty$, the vector $\mathbf{v}=\cB[g]\in W^{1,p}(\bbt^3)$ is a solution of the equation $\Div \mathbf{v} = g$ in $\bbt^3$. More precisely, in $\mathbb{T}^3$, for any given function 
	\begin{align*}
		g(x) = \sum_{k \in \bb{Z}^3 \setminus \{0\}} \hat{g}(k) e^{2 \pi \mathrm{i} k \cdot x},
	\end{align*}
we have
	\begin{align} \label{eqs:Bogovskii}
		\cB[g](x) =-\sum_{k \in \bb{Z}^3 \setminus \{0\}} \frac{\mathrm{i}k}{2 \pi \abs{k}^2} \hat{g}(k) e^{2 \pi \mathrm{i} k \cdot x}.
	\end{align}
	According to \cite{GHH2006}, for $ 1 < p < \infty$, the following estimates hold 
	\begin{align}
	\label{eqs:Bogovskii-Lp}
	   & \norm{ \cB[g]}_{W^{k+1,p}(\bbt^3)}
		 \leq C_p \norm{g}_{W^{k,p}(\bbt^3)}, \quad
		 k=-1,0,1,...
	   \\ 
      \label{eqs:Bogovskii-Lp-div} 
	   & \norm{ \cB[\Div \mathbf{v}]}_{L^p(\bbt^3)}
		 \leq C_p \norm{\mathbf{v}}_{L^p(\bbt^3)}.
	\end{align}

\begin{lemma}
	\label{lem:uniform-P}
    Suppose that the assumptions of Proposition \ref{prop:weak-solution-delta} are satisfied. Let $(\bu_{\sigma},\rho_{\sigma})$ be an approximate solution given by Proposition \ref{prop:weak-solution-sigma}. 
    Then we have 
	\begin{align}
		\label{eqs:uniform-P-sigma}
		\int_0^{T} \int_{\bbt^{3}} \big|\widetilde{P}_\sigma(\rho_{\sigma})\big|\,\dxdt \leq C,
	\end{align}
	where $C > 0$ depends on $\delta$ but is independent of $\sigma$.
\end{lemma}
\begin{proof}
As in \cite{FNP2001}, we consider the test function 
$\psi(t) \cB \big[\rho_\sigma^{(0)}\big]$ with 
$$
\psi \in C_c^\infty((0,T)) \quad \text{and}  
\quad \rho_\sigma^{(0)}=\zeromean{\rho_\sigma}.
$$
Then, taking $\bphi=\psi(t) \cB \big[\rho_\sigma^{(0)}\big]$ in \eqref{eqs:weak-formulation-ddt-sigma} and \eqref{higher-terms}, we obtain
	\begin{align*}
		\int_0^T \psi(t) \int_{\mathbb{T}^{3}} \widetilde{P}_\sigma(\rho_{\sigma})  \rho_\sigma^{(0)} \,\dxdt = \sum_{i = 1}^{3} I_i,
	\end{align*}
	where
	\begin{align*}
		I_1 &=- \int_0^{T} \psi'(t) \int_{\mathbb{T}^{3}} \rho_{\sigma} \bu_{\sigma} \cdot \cB \big[\rho_\sigma^{(0)}\big]\,\dxdt, \\        
		I_2 & = - \int_0^{T} \psi(t) \int_{\mathbb{T}^{3}} \rho_{\sigma} \bu_{\sigma} \cdot \cB \big[\pt\rho_\sigma^{(0)}\big]\,\dxdt, 
		\\
		I_3 & = -\int_0^{T} \psi(t) \int_{\mathbb{T}^{3}} \rho_\sigma (\bu_{\sigma} \otimes \bu_{\sigma}): \nabla \cB \big[\rho_\sigma^{(0)}\big]\,\dxdt 
        +\int_0^{T} \psi(t) \int_{\mathbb{T}^{3}} \rho_{\sigma} \bbd \bu_{\sigma} : \nabla \cB \big[\rho_\sigma^{(0)}\big]\,\dxdt 
        \\
		&\quad -\delta\int_0^T \psi(t) \int_{\mathbb{T}^{3}} H_\delta(\rho_{\sigma})  \rho_\sigma^{(0)} \,\dxdt
        + \delta\int_0^{T} \psi(t) \int_{\mathbb{T}^{3}}   \nabla^2\rho_{\sigma}:\nabla\Delta\rho_{\sigma} \otimes\cB \big[\rho_\sigma^{(0)}\big]\,\dxdt
         \\
		&\quad +\delta \int_0^{T} \psi(t) \int_{\mathbb{T}^{3}}   \nabla\Delta\rho_{\sigma} \otimes\nabla\rho_{\sigma} :\nabla \cB \big[\rho_\sigma^{(0)}\big]\,\dxdt
        + \delta\int_0^{T} \psi(t) \int_{\mathbb{T}^{3}}   (\nabla\rho_{\sigma} \cdot\nabla\Delta\rho_{\sigma})  \rho_\sigma^{(0)}\,\dxdt
		\\
        &\quad + \delta \int_0^{T} \psi(t) \int_{\mathbb{T}^{3}}   \rho_{\sigma} \nabla\Delta\rho_{\sigma} \cdot \nabla\rho_\sigma^{(0)}\,\dxdt 
        + \delta\int_0^{T} \psi(t) \int_{\mathbb{T}^{3}}   (\nabla\rho_{\sigma} \cdot\nabla \bu_\sigma) \cdot \cB \big[\rho_\sigma^{(0)}\big]\,\dxdt
        \\
        &\quad + \delta\int_0^{T} \psi(t) \int_{\mathbb{T}^{3}}   \Delta\bu_{\sigma} \cdot\Delta \cB \big[\rho_\sigma^{(0)}\big]\,\dxdt
        + \zeta^2\int_0^{T} \psi(t) \int_{\mathbb{T}^{3}}  \rho_{\sigma} \Delta \rho_{\sigma}  \rho_\sigma^{(0)} \,\dxdt         
		 \\
		 & \quad         
         + \zeta^2\int_0^{T} \psi(t) \int_{\mathbb{T}^{3}} \Delta \rho_{\sigma} \nabla \rho_{\sigma} \cdot \cB \big[\rho_\sigma^{(0)}\big]\,\dxdt 
         - \frac{1}{\alpha^2} \int_0^{T} \psi(t) \int_{\mathbb{T}^{3}} \Delta^{-1} \Div \bu_\sigma  \rho_\sigma^{(0)} \,\dxdt.
	\end{align*}
Using H\"{o}lder's inequality,  \eqref{eqs:uniform-u-sigma}, \eqref{eqs:uniform-rho-sigma},
\eqref{eqs:conservation-mass-delta},  \eqref{eqs:Bogovskii-Lp} and the Poincar\'{e}--Wirtinger inequality, we get 
	\begin{align}
		|I_{1}|
		&\leq \|\psi'\|_{L^{1}(0,T)}
		\| \bu_{\sigma}\|_{L^{\infty}(0,T;L^{2}(\mathbb{T}^{3}))}\|\rho_{\sigma}\|_{L^{\infty}(0,T;L^{6}(\mathbb{T}^{3}))}
		\norm{\cB\big[ \rho^{(0)}_{\sigma}\big]}_{L^{\infty}(0,T;L^{3} (\mathbb{T}^{3}))}
		\nonumber
		\\
		&\leq C\|\psi'\|_{L^{1}(0,T)}\norm{\cB\big[\rho^{(0)}_{\sigma}\big]}_{L^{\infty}(0,T;H^{1}(\mathbb{T}^{3}))}
        \nonumber
        \\
        &\leq C\|\psi'(t)\|_{L^{1}(0,T)}.\label{pressure-I1}
	\end{align}
	Exploiting the continuity equation \eqref{eqs:app-mass-P-sigma-delta} and \eqref{eqs:Bogovskii-Lp-div}, we have
	\begin{align}
		|I_2| & 
		\leq \abs{\int_0^{T} \psi \int_{\mathbb{T}^{3}} \rho_{\sigma} \bu_{\sigma} \cdot \cB \big[\Div(\rho_{\sigma} \bu_{\sigma})-\delta\Delta\rho_\sigma\big]\,\dxdt} \nonumber \\
		& \leq C\|\psi(t)\|_{L^{\infty}(0,T)} \norm{\rho_\sigma \bu_\sigma}_{L^2(0,T;L^2(\bbt^3))} \big(\norm{\rho_\sigma \bu_\sigma}_{L^2(0,T;L^2(\bbt^3))} 
        +\delta \norm{\nabla \rho_\sigma}_{L^2(0,T;L^2(\bbt^3))} \big)
		\nonumber
		\\
		&\leq C \|\psi(t)\|_{L^{\infty}(0,T)}.
        \label{pressure-I2}
	\end{align}
	Similarly, using H\"{o}lder's inequality, the estimates \eqref{eqs:uniform-u-sigma}--\eqref{eqs:uniform-FWH-sigma}, \eqref{eqs:Bogovskii-Lp}, \eqref{eqs:Bogovskii-Lp-div}, after some straightforward computations, we obtain 
	\begin{align}\label{pressure-I3}
		|I_3|&\leq C\|\psi(t)\|_{L^{\infty}(0,T)}.
	\end{align}
	The positive constant $C$ in \eqref{pressure-I1}, \eqref{pressure-I2} and \eqref{pressure-I3} depends on $\delta$ but is independent of $\sigma$ and $\psi$. Combining \eqref{pressure-I1}--\eqref{pressure-I3} yields   
	\begin{align}
		\label{eqs:P-rho-sigma}
		\abs{\int_0^T \psi(t)\int_{\mathbb{T}^{3}} \widetilde{P}_\sigma(\rho_{\sigma}) \rho_\sigma^{(0)}\,\dxdt} 
        \leq C\big(\|\psi\|_{L^{\infty}(0,T)}+\|\psi'\|_{L^{1}(0,T)}\big).
	\end{align} 
In \eqref{eqs:P-rho-sigma}, choosing  $\psi=\psi_{m}$ ($m\in \mathbb{Z}^+$) such that 
\begin{equation}
	\label{cut-off function}
    \notag 
	\begin{aligned}
		& \psi_{m}\in C_{c}^{\infty}((0,T)),\quad  \psi_{m}\in [0,1],
		\quad \psi_{m}(t)=1 \ \ \text{for}\ \ t\in\left[\frac{1}{m},T-\frac{1}{m}\right], \quad \text{and} \ \ |\psi'|\leq 2m,
	\end{aligned}
\end{equation}
and then sending $ m \to \infty $,  we get  
\begin{align}
		\label{eqs:P-rho-sigma-0}
		\abs{\int_0^T \int_{\mathbb{T}^{3}} \widetilde{P}_\sigma(\rho_{\sigma}) \rho_\sigma^{(0)}\,\dxdt} \leq C.
	\end{align}
Define $\tm= \mean{\rho_{0,\delta}}$.  
Recalling that $\rho_*$, the only zero point of $\widetilde{P}_{c,\sigma}$ in $(\underline{\rho},1)$, is independent of $\sigma$ for $0<\sigma\ll 1$, we can find some $0<a\ll 1$ independent of $\sigma$ such that (cf. also \eqref{eqs:conservation-mass-delta})
    $$
    (1-a)\underline{\rho} + a\tm< \rho_* < a \tm+(1-a).
    $$  
According to different ranges of $\rho_\sigma$, we make the following decomposition
	\begin{align*}
		\int_0^T \int_{\mathbb{T}^{3}} \widetilde{P}_{c,\sigma}(\rho_{\sigma}) \rho_\sigma^{(0)}\,\dxdt
		= J_1 + J_2 + J_3,
	\end{align*}
	with
	\begin{align*}
		J_1 & \coloneqq \int_{\{(1-a)\underline{\rho} + a\tm < \rho_{\sigma}(x,t) < a \tm+(1-a)\}} \widetilde{P}_{c,\sigma}(\rho_{\sigma}) \rho_\sigma^{(0)}\,\dxdt, 
        \\
		J_2 & \coloneqq \int_{\{\rho_{\sigma}(x,t) \geq a \tm+(1-a)\}} \widetilde{P}_{c,\sigma}(\rho_{\sigma})\rho_\sigma^{(0)}\,\dxdt, 
        \\
		J_3 & \coloneqq \int_{\{\rho_{\sigma}(x,t) \leq (1-a)\underline{\rho} + a\tm\}} \widetilde{P}_{c,\sigma}(\rho_{\sigma}) \rho_\sigma^{(0)}\,\dxdt,
	\end{align*}
 Thanks to Remark \ref{rem:properties-entropy}, we obtain 
	\begin{align*}
		J_2 \geq (1-a)(1-\tm) \int_{\{\rho_{\sigma}(x,t) \geq a \tm+(1-a)\}} \abs{\widetilde{P}_{c,\sigma}(\rho_\sigma)}\,\dxdt\geq 0,
	\end{align*}
     where we have used the fact that $\widetilde{P}_{c,\sigma} (\rho_\sigma)>0$ on the set where $\rho_{\sigma}(x,t) \geq a \tm+(1-a)$.  This implies 
\begin{align}
		\int_{\{\rho_{\sigma}(x,t) \geq a \tm+(1-a)\}} \abs{\widetilde{P}_{c,\sigma}(\rho_\sigma)}\,\dxdt
        \leq \frac{1}{(1-a)(1-\tm)}J_2.
        \label{j2}
	\end{align}
In a similar manner, we can conclude
	\begin{align}
		\int_{\{\rho_{\sigma}(x,t) \leq (1-a)\underline{\rho} + a\tm\}}  \abs{\widetilde{P}_{c,\sigma}(\rho_\sigma)}\,\dxdt \leq \frac{1}{(1-a)(\tm - \underline{\rho})} J_3.\label{j3} 
	\end{align}
Since in the set $ \big\{(x,t) \in Q_T: (1-a)\underline{\rho} + a\tm < \rho_{\sigma}(x,t) < a \tm+(1-a)\big\}$, $ \widetilde{P}_{c,\sigma}(\rho_{\sigma}) $ is bounded and its bound is independent of $\sigma$ for $0<\sigma\ll 1$, we have  
    \begin{align}
    |J_1|\leq 
		\int_{\{(1-a)\underline{\rho} + a\tm < \rho_{\sigma}(x,t) < a \tm+(1-a)\}} \abs{\widetilde{P}_{c,\sigma}(\rho_{\sigma})}\,\dxdt\leq C.
        \label{J1}
        \end{align}
This estimate, together with \eqref{eqs:P-rho-sigma-0} and the uniform boundedness of $\rho_\sigma$ (see \eqref{UL-rho-sigma}) yields 
    \begin{align*}
        0\leq J_2+J_3=\int_0^T \int_{\mathbb{T}^{3}} \widetilde{P}_{c,\sigma}(\rho_{\sigma}) \rho_\sigma^{(0)}\,\dxdt
		-J_1\leq C.
    \end{align*}
Combining \eqref{j2}, \eqref{j3}, \eqref{J1}, and using \eqref{UL-rho-sigma} again, we arrive at the conclusion \eqref{eqs:uniform-P-sigma}. 
\end{proof}

\subsection{Equi-integrability of the pressure}

Next, we justify the equi-integrability of $\{\widetilde{P}_\sigma(\rho_{\sigma})\}$ for $0< \sigma \ll 1$. The singular behavior of $\widetilde{F}_c$ near $\underline{\rho}$ and $1$ plays a crucial role in the subsequent proof.     

Invoked by \cite{FLM2016,FZ2010}, we introduce the following test function
\begin{align*}
	\bw_\sigma(x,t) 
	\coloneqq \psi(t) \cB\left[\zeromean{\chi_\sigma(\rho_\sigma)}\right], \quad \psi \in C_c^\infty((0,T)),\ \ 0<\sigma\ll 1,
\end{align*}
where $\chi_\sigma$ is a truncated function defined by
\begin{align}\label{cut function-1}
	\chi_\sigma(r) 
=
	\left\{
	\begin{aligned}
		& \ln(1 - \sigma - \underline{\rho}) - \ln \sigma, && \quad r > 1 - \sigma, \\
		& \ln(r - \underline{\rho}) - \ln(1 - r), && \quad \underline{\rho} + \sigma \leq r \leq 1 - \sigma, \\
		& \ln\sigma - \ln(1 - \sigma - \underline{\rho}), && \quad r < \underline{\rho} + \sigma.
	\end{aligned}
	\right.
\end{align}
The above truncation is due to the bound \eqref{UL-rho-sigma} for $\rho_\sigma$.  
From the definition of $\widetilde{F}_{\sigma}$, $\widetilde{P}_{\sigma}$, Assumption \ref{main assumption}, and \eqref{UL-rho-sigma}, we can check that $\chi_\sigma(\rho_\sigma)$ satisfies
\begin{align*}
	\abs{\chi_\sigma(\rho_\sigma(x,t))}^{p} & \leq C \abs{\widetilde{F}_{\sigma} (\rho_\sigma(x,t))} +C,\quad  \forall\, 1 \leq p < \infty, 
    \\
	\abs{\chi_\sigma' (\rho_\sigma(x,t))}^{\beta} 
    & \leq C\abs{\widetilde{F}_{\sigma} (\rho_\sigma(x,t))} +C,  
    \\
	\abs{\chi_\sigma' (\rho_\sigma(x,t))}^{\beta+1} & \leq C\abs{\widetilde{P}_{\sigma} (\rho_\sigma(x,t))} + C.
\end{align*}
These estimates together with \eqref{eqs:uniform-FWH-sigma} and \eqref{eqs:uniform-P-sigma} ensure that 
\begin{align}
	\label{sigular1}\chi_\sigma(\rho_\sigma) & \text{ is uniformly bounded in }  L^\infty(0,T;L^p(\mathbb{T}^{3})), \text{ for all } 1 \leq p < \infty, \\
	\label{sigular2}   \chi_\sigma'(\rho_\sigma) & \text{ is uniformly bounded in } L^\infty(0,T;L^\beta(\mathbb{T}^{3})) \cap  L^{\beta+1}(Q_T),
\end{align}
with respect to $\sigma$ whenever $\sigma>0$ is sufficiently small. 

The following lemma establishes a uniform bound of $\widetilde{P}_\sigma(\rho_{\sigma}) \chi_\sigma(\rho_{\sigma})$ in $L^1(Q_T)$, which implies the equi-integrability of $\{\widetilde{P}_\sigma(\rho_{\sigma})\}$ for $0<\sigma\ll 1$, that is, 
\begin{align*}
 \forall\,\varepsilon>0, \quad \exists\, M>0, \quad \text{s.t.}\quad  \sup_{0<\sigma\ll 1}
\int_{\{|\widetilde{P}_\sigma(\rho_{\sigma}) |\geq M\}} |\widetilde{P}_\sigma(\rho_{\sigma}) | \,\dxdt \leq \varepsilon. 
\end{align*}

\begin{lemma}
	\label{lem:uniform-P-chi-sigma}
	Suppose that the assumptions of Proposition \ref{prop:weak-solution-delta} are satisfied. Let $(\bu_{\sigma},\rho_{\sigma})$ be an approximate solution given by Proposition \ref{prop:weak-solution-sigma}. 
    Then we have 
	\begin{align}
		\label{eqs:uniform-P-sigma-chi}
		\int_0^{T} \int_{\bbt^{3}} \abs{\widetilde{P}_\sigma(\rho_{\sigma}) \chi_\sigma(\rho_{\sigma})}\,\dxdt \leq C,
	\end{align}
	where $C > 0$ depends on $\delta$ but is independent of $\sigma$ (for $0<\sigma\ll 1$).
\end{lemma}
\begin{proof}
Choosing the test function $\bphi=\bw_\sigma$ in \eqref{eqs:weak-formulation-ddt-sigma} and \eqref{higher-terms}, denoting for simplicity  $$\chi_\sigma^{(0)}= \zeromean{\chi_\sigma(\rho_\sigma)},$$ 
then we get
\begin{align*}
  & \int_0^{T} \psi(t) \int_{\mathbb{T}^{3}} \widetilde{P}_\sigma(\rho_\sigma)\chi_\sigma^{(0)}\,\dxdt\\
  &
    = -\int_0^{T} \psi'(t) \int_{\mathbb{T}^{3}}\rho_\sigma \bu_\sigma \cdot \cB\big[\chi_\sigma^{(0)}\big]\,\dxdt
  - \int_0^{T} \psi(t) \int_{\mathbb{T}^{3}} \rho_\sigma \bu_\sigma \cdot \cB\big[\pt\chi_\sigma^{(0)}\big] \,\dxdt 
		\\
        &\quad - \int_0^{T} \psi(t) \int_{\mathbb{T}^{3}} \rho_\sigma (\bu_\sigma \otimes \bu_\sigma) : \nabla \cB\big[\chi_\sigma^{(0)}\big] \,\dxdt
		  +\int_0^{T} \psi(t) \int_{\mathbb{T}^{3}} \rho_\sigma \bbd \bu_\sigma : \nabla\cB\big[\chi_\sigma^{(0)}\big]\,\dxdt	 
		\\
        &\nonumber\quad 
        - \delta \int_0^{T} \psi(t) \int_{\mathbb{T}^{3}}   H_\delta(\rho_\sigma) \chi_\sigma^{(0)} \,\dxdt 
        + \delta\int_0^{T} \psi(t) \int_{\mathbb{T}^{3}}   \nabla^2\rho_{\sigma}:\nabla\Delta\rho_{\sigma} \otimes \cB\big[\chi_\sigma^{(0)}\big] \,\dxdt
        \\
        &\quad 
		   + \delta\int_0^{T} \psi(t) \int_{\mathbb{T}^{3}}   \nabla\Delta\rho_{\sigma} \otimes\nabla\rho_{\sigma} : \nabla\cB\big[\chi_\sigma^{(0)}\big]  \,\dxdt
        + \delta\int_0^{T} \psi(t) \int_{\mathbb{T}^{3}}   \nabla\rho_{\sigma} \cdot\nabla\Delta\rho_{\sigma} \chi_\sigma^{(0)}\,\dxdt
        \\
        &\quad +\delta \int_0^{T} \psi(t) \int_{\mathbb{T}^{3}} \rho_{\sigma} \nabla\Delta\rho_{\sigma} \cdot \nabla\chi_\sigma^{(0)}\,\dxdt
        + \delta\int_0^{T} \psi(t) \int_{\mathbb{T}^{3}} (\nabla \rho_\sigma \cdot\nabla\bu_\sigma) \cdot \cB\big[\chi_\sigma^{(0)}\big] \,\dxdt
		\\
        &\nonumber\quad	
        + \delta\int_0^{T} \psi(t) \int_{\mathbb{T}^{3}} \Delta \bu_\sigma  \cdot \Delta\cB\big[\chi_\sigma^{(0)}\big] \,\dxdt
        + \zeta^2\int_0^{T} \psi(t) \int_{\mathbb{T}^{3}}\Delta \rho_\sigma \nabla\rho_\sigma   \cdot\cB\big[\chi_\sigma^{(0)}\big]\, \dxdt
        \\
        &\nonumber\quad
		  + \zeta^2\int_0^{T} \psi(t) \int_{\mathbb{T}^{3}}\rho_\sigma \Delta \rho_\sigma \chi_\sigma^{(0)}\, \dxdt
        - \frac{1}{\alpha^2}\int_0^{T} \psi(t) \int_{\mathbb{T}^{3}} \Delta^{-1} \Div \bu_\sigma \chi_\sigma^{(0)}\,\dxdt.
	\end{align*}
Using \eqref{UL-rho-sigma}, \eqref{eqs:uniform-u-sigma}--\eqref{eqs:conservation-mass-delta},  \eqref{eqs:Bogovskii-Lp}, \eqref{eqs:Bogovskii-Lp-div} and the estimates \eqref{sigular1}, \eqref{sigular2} for $\chi_\sigma(\rho_\sigma)$, by an argument similar to that for Lemma \ref{lem:uniform-P}, we can deduce that  
\begin{align}
		\label{eqs:uniform-P-sigma-chi-b}
		\left|\int_0^{T} \int_{\bbt^{3}}  \widetilde{P}_\sigma(\rho_{\sigma}) \chi_\sigma^{(0)}\,\dxdt\right|
        \leq C,
\end{align}
where $C>0$ depends on $\delta$ but is independent of $\sigma$. Combining \eqref{UL-rho-sigma}, \eqref{sigular1} and \eqref{eqs:uniform-P-sigma-chi-b}, we further get 
    \begin{align}
		\label{eqs:uniform-P-sigma-chi-c}
		\left|\int_0^{T} \int_{\bbt^{3}}  \widetilde{P}_{c,\sigma}(\rho_{\sigma}) \chi_\sigma^{(0)}\,\dxdt\right|
        \leq C.
\end{align}
Applying \eqref{sigular1} with $p=1$, we find 
$\sup_{0\leq t\leq T}\big|\mean{\chi_\sigma(\rho_\sigma)}\big|$ is bounded by a positive constant $\widehat{m}$, which is independent of $\sigma$. By the definition of $P_{c,\sigma}$ and $\chi_\sigma$, there exists a sufficiently small $\widehat{\sigma}\in (0,1)$ such that for all $\sigma\in (0,\widehat{\sigma})$, 
\begin{align*}
& 
P_{c,\sigma}(r)>0,\quad \chi_\sigma(r) \geq 2\widehat{m},\quad \forall\,  r\geq 1-\widehat{\sigma},\\
&
P_{c,\sigma}(r)<0,\quad \chi_\sigma(r)\leq -2\widehat{m},\quad \forall\,  r\leq \underline{\rho}+\widehat{\sigma}. 
\end{align*}
Make the decomposition
	\begin{align*}
		\int_0^T \int_{\mathbb{T}^{3}} \widetilde{P}_{c,\sigma}(\rho_{\sigma}) \chi_\sigma^{(0)}\,\dxdt
		= \widehat{J}_1 + \widehat{J}_2 + \widehat{J}_3,
	\end{align*}
	with
	\begin{align*}
		\widehat{J}_1 & \coloneqq \int_{\{\underline{\rho}+\widehat{\sigma}< \rho_{\sigma}(x,t) < 1-\widehat{\sigma}\}} \widetilde{P}_{c,\sigma}(\rho_{\sigma}) \chi_\sigma^{(0)}\,\dxdt, 
        \\
		\widehat{J}_2 & \coloneqq \int_{\{\rho_{\sigma}(x,t) \geq 1-\widehat{\sigma}\}} \widetilde{P}_{c,\sigma}(\rho_{\sigma})\chi_\sigma^{(0)}\,\dxdt, 
        \\
		\widehat{J}_3 & \coloneqq \int_{\{\rho_{\sigma}(x,t) \leq \underline{\rho}+\widehat{\sigma}\}} \widetilde{P}_{c,\sigma}(\rho_{\sigma}) \chi_\sigma^{(0)}\,\dxdt.
	\end{align*}
Using the same argument as for Lemma \ref{lem:uniform-P}, we can conclude \eqref{eqs:uniform-P-sigma-chi} from \eqref{eqs:uniform-P-sigma-chi-c}. This completes the proof. 
\end{proof}
%

\subsection{The limit procedure as $\sigma\to 0$}
We are ready to pass to the limit as $\sigma\to 0$, while keeping $0<\delta\ll1 $ fixed. In what follows, convergence is always understood in the sense of a subsequence (not relabeled for simplicity).

Thanks to the uniform estimates \eqref{eqs:uniform-rho-sigma}, \eqref{eqs:uniform-u-sigma}, \eqref{eqs:uniform-rhot-sigma} and the weak (resp. weak star) compactness theorem, there is a subsequence $\{(\mathbf{u}_\sigma,\rho_\sigma)\}$ and a pair $(\mathbf{u},\rho)$ satisfying  
\begin{align*}
& \bu\in L^\infty(0,T;L^2(\bbt^3)) \cap L^2(0,T;H^2(\bbt^3)),  
\\
&\rho\in L^{\infty}(0,T;H^2(\bbt^3))\cap  L^2(0,T;H^3(\bbt^3))\cap H^1(0,T;H^1(\mathbb{T}^3)),
\end{align*} 
such that $(\mathbf{u}_\sigma,\rho_\sigma)$ weakly (resp. weakly star) converges to $(\mathbf{u},\rho)$
in the corresponding spaces as $\sigma\to 0$. Moreover, the Aubin--Lions--Simon lemma yields 
$\rho_\sigma\to \rho$ strongly in $C([0,T];H^2(\mathbb{T}^3))$, which also implies the almost everywhere convergence of $\rho_\sigma$ in $Q_T$. 
Then we have  
\begin{align}
\rho\in BC([0,T];H^2(\bbt^3))\quad \text{and}\quad 
		\frac{1}{2}\underline{\rho} 
        \leq \rho(x,t)\leq 
        1+\frac{1}{2}\underline{\rho} \quad  \text{for all}\ \ (x,t)\in \mathbb{T}^3\times [0,T]. \label{range-1}
	\end{align}

With the help of the singular behavior of $\widetilde{F}_c$ near $\underline{\rho}$ and $1$  (cf. \eqref{singular behavior}), we can  improve the upper and lower bounds in \eqref{range-1}. The proof is based on a classical method for the Cahn--Hilliard equation with a singular potential (cf. \cite{Mi2019}).  
\begin{lemma}
	\label{lem:density-lower-upper-bound}
	Let $\rho$ be the limit function obtained above. We have $\rho \in(\underline{\rho }, 1)$ almost everywhere in $Q_T$.
\end{lemma}
\begin{proof} It follows from \eqref{eqs:uniform-FWH-sigma} and \eqref{UL-rho-sigma} that $\widetilde{F}_{c,\sigma}(\rho_\sigma)$ is uniformly bounded in $L^1(Q_T)$.  
	For any $0< \sigma \ll 1 $ and $\epsilon\in (\sigma, (1-\underline{\rho})/4)$, we introduce the sets
	\begin{align*}
		E^\sigma_\epsilon & \coloneqq \left\{ (x,t) \in Q_T : \rho_\sigma(x,t) < \underline{\rho} + \epsilon \text{ or } \rho_\sigma(x,t) > 1 - \epsilon \right\}, \\
		E_\epsilon & \coloneqq \left\{ (x,t) \in Q_T : \rho(x,t) < \underline{\rho} + \epsilon \text{ or } \rho(x,t) > 1 - \epsilon \right\}.
	\end{align*}
	Due to Assumption \ref{main assumption}, the definition of $\widetilde{F}_c$, $\widetilde{F}_{c,\sigma}$, and the fact $0<\sigma<\epsilon$, we have 
	\begin{align*}
		\min \left\{|\widetilde{F}(\underline{\rho} + \epsilon)|,\  |\widetilde{F}(1 - \epsilon)|\right\} \meas(E^\sigma_\epsilon) \leq \int_{E^\sigma_\epsilon} |\widetilde{F}_{c,\sigma}(\rho_\sigma)|\,\dxdt
		\leq \int_{Q_T} |\widetilde{F}_{c,\sigma}(\rho_\sigma)|\,\dxdt
		\leq C,
	\end{align*}
	where the constant $ C > 0 $ is independent of $ \sigma $ and $ \epsilon $. Since $ \rho_\sigma \to \rho$ almost everywhere in $Q_T$, it follows from  Fatou’s lemma that
	\[
	\meas(E_\epsilon) 
    \leq \liminf_{\sigma \to 0^+} \meas(E^\sigma_\epsilon)
    \leq \frac{C}{\min \left\{ |\widetilde{F}(1 - \epsilon)|, |\widetilde{F}(\underline{\rho} + \epsilon)| \right\}}.
	\]
   	Next, passing to the limit as \( \epsilon \to 0 \), 
	we find 
	\begin{align*}
		\meas\left( \left\{ (x,t) \in Q_T : \rho(x,t) \leq \underline{\rho} \text{ or } \rho(x,t) \geq 1 \right\} \right) = 0,
	\end{align*}
	which implies 
    \begin{align}
		\rho \in \big(\underline{\rho}, 1\big) \quad \text{ a.e. in } Q_T.
        \label{bound-1}
	\end{align}   
	This completes the proof.
\end{proof}

From the almost everywhere convergence of $\rho_\sigma$ to $\rho$, Lemma \ref{lem:density-lower-upper-bound}, and the definition of $\widetilde{F}_{\sigma}$, $\widetilde{P}_\sigma$, we infer that as $\sigma\to 0$,  
\begin{align}
& \widetilde{F}_\sigma(\rho_\sigma) \to \widetilde{F}(\rho),\quad \widetilde{F}'_\sigma(\rho_\sigma) \to \widetilde{F}'(\rho),\quad \text{a.e. in}\ \ Q_T,\notag
\end{align}
which yield 
\begin{align}
&\widetilde{P}_\sigma(\rho_\sigma) \to \widetilde{P}(\rho)= \rho \widetilde{F}'(\rho) - \widetilde{F}(\rho), \quad \text{a.e. in}\ \ Q_T. 
\label{pressure pointwise}
\end{align}
The point-wise convergence \eqref{pressure pointwise} together with the uniform estimates \eqref{eqs:uniform-P-sigma} and \eqref{eqs:uniform-P-sigma-chi} enables us to apply the Dunford--Pettis theorem and conclude that
\begin{align*}
	\widetilde{P}_\sigma(\rho_\sigma) \rightharpoonup \widetilde{P}(\rho) \quad \text{ weakly in } L^1(Q_T).
\end{align*}
Lemma \ref{lem:density-lower-upper-bound} also implies  
$$
W_\delta(\rho)=W(\rho)\quad \text{and}\quad H_\delta(\rho)=H(\rho).
$$ 

With the estimate \eqref{eqs:uniform-P-sigma}, similarly to \eqref{es-rut-N}, we can obtain 
\begin{align*}
& \left|\big\langle \partial_t(\rho_\sigma\mathbf{u}_\sigma),\bphi\big\rangle_{L^1(0,T;(W^{2,4}(\mathbb{T}^3))^*),L^\infty(0,T;W^{2,4}(\mathbb{T}^3))}\right| 
\leq C \|\bphi\|_{L^\infty(0,T;W^{2,4}(\mathbb{T}^3))},
\end{align*}
for any $\bphi\in L^\infty(0,T;W^{2,4}(\mathbb{T}^3))$, which implies 
$$
\partial_t(\rho_\sigma \mathbf{u}_\sigma)\ \ \text{is uniformly bounded in} \ \ L^1(0,T;(W^{2,4}(\mathbb{T}^3))^*). 
$$
Then applying an argument similar to that in \cite[Section 4.1]{FS2021}, we can conclude that 
$$
\bu_\sigma \to \bu \quad \text{strongly}\quad L^{4}(0,T; L^{4}(\mathbb{T}^{3}))\quad \text{as}\ \ \sigma\to 0.
$$ 

The convergence results established above enable us to pass to the limit in the weak formulation \eqref{eqs:weak-formulation-ddt-sigma} and thereby obtain \eqref{eqs:weak-formulation-delta}. In addition, the continuity equation \eqref{eqs:app-mass-sigma1} can be achieved by taking $\sigma \to 0$ in \eqref{eqs:app-mass-P-sigma-delta}. Like in Proposition \ref{prop:weak-solution-sigma}, we can recover the energy estimate \eqref{eqs:energy-inequality-delta} and the conservation of mass \eqref{conservation law-000}. 
The remaining details are standard, and hence omitted.

The proof of Proposition \ref{prop:weak-solution-delta}  is complete. \qed




\section{Passage to the limits as $\delta \to 0$}
\label{sec:limit-delta}
In this section, we study the limit as $ \delta \to 0 $. This establishes the existence of a global weak solution $ (\bu,\rho,\mu_p) $ to the following Navier--Stokes--Korteweg type system  
\begin{subequations}
	\label{eqs:app-limit}
	\begin{align}
		\label{eqs:app-momentum-limit}
		& \partial_t (\rho \bu) 
		+ \Div(\rho\bu\otimes \bu) 
		+ \nabla \widetilde{P}(\rho)= \Div (\rho \bbd \bu)
		+ \zeta^2 \rho \nabla \Delta \rho
		- \frac{1}{\alpha} \nabla \mu_p, 
        \quad \text{in}\ \ Q_T,
        \\
		\label{eqs:app-mass-limit}
		& \pt \rho + \Div (\rho \bu) = 0, \quad \text{in}\ \ Q_T,
        \\
		& \Div \bu = \alpha \Delta \mu_{p},\quad \text{in}\ \ Q_T,
        \label{eqs:app-mup-limit}
	    \\
        &
		\label{eqs:app-initial-limit}
		\bu|_{t=0} =\bu_{0},\quad  \rho|_{t=0} =\rho_{0}, \quad\text{in}\ \  \bbt^3.
	\end{align}
\end{subequations}

Let us first introduce the definition of weak solutions to the problem \eqref{eqs:app-limit}.
\begin{definition}
	\label{def:limit}
	Let $T\in (0,\infty)$. A triple $(\bu,\rho,\mu_p)$ is called a weak solution to problem \eqref{eqs:app-limit} in $[0,T]$, if it satisfies the following properties. 
	\begin{enumerate}
		\item Regularity: $(\bu,\rho,\mu_{p})$ satisfies \begin{align*}
        &\mathbf{u}\in L^\infty(0,T; L^{2}(\mathbb{T}^{3})) \cap L^{2}(0,T; H^{1}(\mathbb{T}^{3})),    
			\\
            & \rho \in BC_w([0,T]; H^1(\mathbb{T}^{3})) \cap L^{2}(0,T; H^{2}(\mathbb{T}^{3}))\cap W^{1,\frac43}(0,T;L^2(\Omega)), 
			\\
            & \rho\in  L^\infty(Q_T) \quad \text{with} \quad \underline{\rho} < \rho(x,t) < 1\ \text{ a.e. in } \ Q_T,
			\\
            & \mu_{p} \in L^\infty(0,T;H^1(\mathbb{T}^3))\cap L^{2}(0,T; H^{2}(\mathbb{T}^{3})).
		\end{align*}
		\item Weak formulation of the momentum equation:
		\begin{align}
			\nonumber
		      &-\int_{\mathbb{T}^3} \rho_0\mathbf{u}_0\cdot\bphi(\cdot,0)\,\dx 
              - \int_0^{T}\int_{\mathbb{T}^{3}} \rho \bu\cdot\partial_{t}\bphi\,\dx\dt
             \\
             \notag
             & \qquad -\int_0^{T}\int_{\mathbb{T}^{3}} (\rho \bu \otimes \bu) :\nabla \bphi\,\dx\dt 
             + \int_0^{T}\int_{\mathbb{T}^{3}} \rho \bbd \bu:\bbd \bphi\,\dx\dt 
			\\ 
            \nonumber 
            &\quad =
			\int_0^{T}\int_{\mathbb{T}^{3}} \widetilde{P}(\rho)\Div \bphi\,\dx\dt 
            -\zeta^2\int_0^{T}\int_{\mathbb{T}^{3}} \rho \Delta \rho\Div \bphi\,\dx\dt  
			\\
            &\qquad 
            -\zeta^2\int_0^{T} \int_{\mathbb{T}^{3}} \Delta \rho (\nabla \rho \cdot\bphi) \,\dx\dt
			- \frac{1}{\alpha} \int_0^{T}\int_{\mathbb{T}^{3}} \nabla \mu_p\cdot\bphi\,\dx\dt,
			\label{weak1-sigma}
		\end{align}
        for all $\boldsymbol{\varphi}\in C_{c}^{\infty}([0,T); C^{\infty}(\mathbb{T}^{3}))$. 
		\item It holds
		\begin{align}
			&\label{weak2-sigma}
			\partial_{t}\rho +\Div(\rho\mathbf{u})=0,\quad \ \text{a.e.~in } Q_T,
			\\
            &\label{weak3-sigma}	\Div\mathbf{u}=\alpha \Delta \mu_{p},\qquad\quad  \text{a.e.~in } Q_T.
		\end{align}
		\item The triple $(\bu, \rho, \mu_{p})$ satisfies the energy inequality
		\begin{align}
			\label{eqs:energy-inequality}
			E(\mathbf{u}(t),\rho(t))
			+\int_0^{t}\int_{\mathbb{T}^{3}} \left(\rho \abs{\bbd \bu}^2 +|\nabla \mu_{p}|^{2}\right)\, \dx\dtau
			\leq E(\mathbf{u}_{0},\rho_0),
		\end{align}
        for almost all $t \in (0,T]$, 	where the total energy $E$ is defined as
		\begin{align*}
			E(\mathbf{u}, \rho)
			\coloneqq
			\int_{\mathbb{T}^{3}} \Big( \frac{1}{2}\rho|\mathbf{u}|^2 
            + \frac{\zeta^2}{2}\absm{\nabla \rho}^2 
            +\widetilde{F}(\rho)
			\Big)\,\dx.
		\end{align*}
		\item The triple $(\bu, \rho, \mu_{p})$ satisfies the BD-entropy estimate
		\begin{align}
			E_{\mathrm{BD}}(\bu(t),\rho(t))
			+ \int_0^t \int_{\bbt^3}\Big(\rho  \absm{\nabla \bu - (\nabla \bu)^\top}^2+   \absm{\nabla \mu_{p}}^2 + \zeta^2 \abs{\Delta \rho}^{2}  
			\Big)\, \dx\dtau
			\leq C,
            \label{eqs:BD-entropy}
		\end{align}
        for almost all $t \in (0,T]$, 
		where the BD-entropy is defined as 
		\begin{align*}
			E_{\mathrm{BD}}(\bu,\rho)
			= \int_{\bbt^3} 
			\Big(
			\frac{\rho}{2} \abs{\bu + \nabla \ln \rho}^2
            + \frac{\zeta^2}{2} \absm{\nabla \rho}^2 
			+ \widetilde{F}(\rho)
			\Big)\,\dx,
		\end{align*}
        and the constant $C>0$ depends on $E(\mathbf{u}_{0},\rho_0)$, $\underline{\rho}$ and $T$. 
        \item Conservation of mass:         
	\begin{align}
		\int_{\bbt^3} \rho(\cdot,t) \,\dx = \int_{\bbt^3} \rho_0 \,\dx, 
        \quad \forall\, t\in [0,T].
        \label{conservation law-1111}
	\end{align} 
		\item The initial condition is satisfied in the following sense: 
		\begin{align}
			& \rho|_{t=0}=\rho_{0},\ \text{ a.e.~in }\ \mathbb{T}^{3},\label{attain-initial-1-sigma}
		\end{align}	
        while $\mathbf{u}_0$ is attained in the weak formulation \eqref{weak1-sigma}.
	\end{enumerate}
\end{definition}
Now we are in a position to state the main result of this section.
\begin{theorem}\label{thm:NSK}
Suppose that Assumption \ref{main assumption} is satisfied and $\widetilde{F}$, $\widetilde{P}$ are defined as in \eqref{eqs:resuced-F}, \eqref{eqs:new-entropy}, respectively. Let $T\in (0,\infty)$. Then for any initial data satisfying $(\bu_{0},\rho_{0}) \in L^2(\bbt^3) \times H^1(\bbt^3)$ with $ \widetilde{F}(\rho_0)\in L^1(\mathbb{T}^3)$,  
problem \eqref{eqs:app-limit} admits a global weak solution $(\mathbf{u}, \rho, \mu_p)$ on $[0,T]$ in the sense of Definition \ref{def:limit}. 
\end{theorem}
\begin{remark}
Theorem \ref{thm:NSK} establishes the existence theory for a class of Navier--Stokes--Korteweg system with singular potential and density-dependent viscosity, which is of independent interest. In this regard, we also refer to \cite{FLM2016} for a similar system but without the higher-order capillary force term $\rho \nabla \Delta \rho$ in the momentum equation. The main difference arises from the different form of the free energy and, consequently, the chemical potential.
\end{remark}

To prove Theorem \ref{thm:NSK}, we construct approximate solutions denoted by  $(\bu_\delta,\rho_\delta)$ using Proposition \ref{prop:weak-solution-delta}. After deriving the necessary estimates for $(\bu_\delta,\rho_\delta)$, which are uniform in $\delta$, we pass to the limit as $\delta \to 0$. 

\subsection{Energy estimates and mass conservation}
From \eqref{Fc-delta}, \eqref{rho-UpLo-delta}, \eqref{eqs:energy-inequality-delta}, \eqref{energy-delta} and the definition of $\rho_{0,\delta}$ (cf. \eqref{eqs:rho0-bounds}, \eqref{eqs:rho0-convergence}), we obtain the following regularity properties of $(\mathbf{u}_\delta, \rho_\delta)$, which hold with uniform bounds in the corresponding spaces for $0<\delta\ll 1$:
\begin{align}
& \sqrt{\rho_\delta}\mathbf{u}_\delta, \  \mathbf{u}_\delta\in L^\infty(0,T;L^2(\mathbb{T}^3)),  
\qquad 
\sqrt{\rho_\delta}\mathbb{D} \mathbf{u}_\delta,\  \mathbb{D} \mathbf{u}_\delta,\  \delta^\frac12 \Delta \mathbf{u}_\delta\in L^2(0,T;L^2(\mathbb{T}^3)),
\label{eqs:uniform-u-del}
\\
& \zeta\nabla \rho_\delta,\,\delta^\frac12\Delta \rho_\delta\in L^\infty(0,T;L^2(\mathbb{T}^3)),
\qquad \delta \nabla \Delta \rho_\delta,\  \delta^\frac12\zeta \Delta \rho_\delta\in L^2(0,T;L^2(\mathbb{T}^3)),
\label{eqs:uniform-rho-del}
\\
& \widetilde{F}(\rho_\delta) \in L^\infty(0,T;L^1(\mathbb{T}^3)).
\label{eqs:uniform-FWH-del}
\end{align}
In addition, we have the conservation of mass (see \eqref{conservation law-000})
	\begin{align}
		\label{eqs:conservation-mass-delta-b}
		\mean{\rho_\delta(\cdot,t)} = \mean{\rho_{0,\delta}} \in \big(\underline{\rho}, 1\big),\quad \forall\, t\in[0,T].
	\end{align}
Let $\mu_{p,\delta}$ be a solution to the equation 
\begin{align}
\label{eqs:app-mup-delta} 
\Div \mathbf{u}_\delta = \alpha\Delta\mu_{p,\delta}\quad \text{in}\ \ Q_T,
\end{align}
which satisfies $\mean{\mu_{p,\delta}(t)}=0$ for almost all $t\in [0,T]$. It follows from \eqref{eqs:uniform-u-del}, the elliptic estimate and the Poincar\'{e}--Wirtinger inequality that 
\begin{align}
	\label{eqs:uniform-mup-H2}
	\mu_{p,\delta} & \text{ is uniformly bounded in } L^\infty(0,T;H^1(\mathbb{T}^3))\cap L^2(0,T;H^2(\bbt^3)).
\end{align}

\subsection{BD-entropy estimate} 
\label{sec:uniform-BD-app} 
To obtain better estimates, we derive the Bresch--Desjardins inequality for $(\mathbf{u}_\delta,\rho_\delta)$. 
\begin{lemma}
Let $(\bu_{\delta},\rho_{\delta})$ be an approximate solution given by Proposition \ref{prop:weak-solution-delta}. 
    Then we have 
\begin{align}
&\sup_{0 \leq t \leq T}  E^{\delta}_{\mathrm{BD}}(\bu_\delta(t),\rho_\delta(t)) 
+ \frac{\delta}{2} \int_0^T \int_{\bbt^3} \absm{\Delta \bu_\delta}^2 \,\dxdt  
+  \frac{1}{\alpha^2} \int_0^T \int_{\bbt^3} \absm{\nabla \Delta^{-1} \Div \bu_\delta}^2 \,\dxdt
\notag\\
&\qquad +
\frac{1}{4} \int_0^T\int_{\bbt^3} \rho_\delta \absm{\nabla \bu_\delta - (\nabla \bu_\delta)^\top}^2 \,\dxdt 
+ \frac{\zeta^2}{2} \int_0^T \int_{\bbt^3} \absm{\Delta \rho_\delta}^2 \,\dxdt  
 \notag\\
&\qquad 
    + \frac{\delta}{2} \int_0^T \int_{\bbt^3} \absm{\nabla\Delta \rho_\delta}^2 \,\dxdt 
     + \frac{\delta}{2} \int_0^T\int_{\mathbb{T}^3} \frac{|\Delta\rho_\delta|^2}{\rho_\delta}\,\dxdt  
     + \delta \int_0^T \int_{\bbt^3} W''(\rho_\delta) \abs{\nabla \rho_\delta}^{2} \,\dx\dt
\notag \\
&\quad \leq E^{\delta}_{\mathrm{BD}}(\bu_0,\rho_{0,\delta}) +  C(T,\omega, \underline{\rho}, E_{\delta}(\bu_0,\rho_{0,\delta}),C_*),  
    \label{eqs:BD-entropy-delta-1}
\end{align}
    where the approximate BD-entropy $E_{\mathrm{BD}}^\delta$ is defined as 
\begin{align*}
	E_{\mathrm{BD}}^\delta(\bu_\delta,\rho_\delta)
	\coloneqq \int_{\bbt^3}
	\Big( 
	\frac{\rho_\delta}{2} \abs{\bu_\delta + \nabla \ln \rho_\delta}^2
	+ \frac{\zeta^2}{2}|\nabla \rho_\delta|^2 
    + \widetilde{F}(\rho_\delta) 
	  + \frac{\delta}{2} \absm{\Delta \rho_\delta}^2 
	+ \delta W(\rho_\delta)
	\Big) \,\dx.
\end{align*}

\end{lemma}
\begin{proof}
We start with the approximate solution $(\mathbf{u}_\sigma,\rho_\sigma)$ given by  Proposition \ref{prop:weak-solution-sigma}. Using the regularized equation \eqref{eqs:app-mass-P-sigma-delta}, we have (cf. \cite[(A.2)]{MSZ2025})
\begin{align*}
\int_{\mathbb{T}^3}(\rho_\sigma \mathbf{u}_\sigma)\partial_t\nabla \ln \rho_\sigma \,\dx 
= \int_{\mathbb{T}^3} \frac{1}{\rho_\sigma} (\Div(\rho_\sigma\mathbf{u}_\sigma))^2\,\dx 
-\delta\int_{\mathbb{T}^3} \frac{\Delta\rho_\sigma}{\rho_\sigma}  \Div(\rho_\sigma\mathbf{u}_\sigma)\,\dx,\quad \text{a.e. in}\ \ (0,T).
\end{align*}
Testing \eqref{eqs:weak-formulation-ddt-sigma} by $\nabla \ln \rho_\sigma$, following the calculations in Lemma \ref{lem:BD-entropy} and those in \cite[Appendix A]{MSZ2025} for additional regularizing terms involving $\delta$, we obtain 
 \begin{align*}
     & \int_{\mathbb{T}^3} \partial_t(\rho_\sigma\mathbf{u}_\sigma)\cdot\nabla \ln \rho_\sigma \,\mathrm{d}x -\frac{1}{\alpha^2} \ddt\int_{\bbt^3} \ln \rho_\sigma \,\dx
     \\
     &\quad = -\int_{\mathbb{T}^3} (\rho_\sigma \mathbf{u}_\sigma\otimes \mathbf{u}_\sigma)\cdot\nabla \ln \rho_\sigma \,\mathrm{d}x 
     + \int_{\mathbb{T}^3} \nabla \mathbf{u}_\sigma : (\nabla \rho_\sigma\otimes \nabla \ln \rho_\sigma)\,\mathrm{d}x \\
     &\qquad -\int_{\mathbb{T}^3} \Delta\rho_\sigma\Div \mathbf{u}_\sigma\, \dx 
     - \int_{\mathbb{T}^3}   (\widetilde{F}''_\sigma(\rho_\sigma) +\delta W_\delta''(\rho_\sigma))|\nabla \rho_\sigma|^2 \,\mathrm{d}x \\
     & \qquad - \zeta^2 \int_{\mathbb{T}^3} |\Delta \rho_\sigma|^2\, \dx -\delta \int_{\mathbb{T}^3} |\nabla\Delta\rho_\sigma|^2\,\dx \\
     & \qquad -\delta\int_{\mathbb{T}^3} (\nabla \rho_\sigma\cdot\nabla \mathbf{u}_\sigma)\cdot\nabla \ln \rho_\sigma\,\dx 
     -\delta\int_{\mathbb{R}}\Delta \mathbf{u}_\sigma \cdot\nabla \Delta \ln \rho_\sigma\,\dx \\
     &\qquad - \frac{\delta}{\alpha^2} \int_{\bbt^3} \frac{\Delta \rho_\sigma}{\rho_\sigma} \,\dx, \quad \text{a.e. in}\ \ (0,T).
 \end{align*}
Similarly to Lemma \ref{lem:BD-entropy-logrho}, using \eqref{eqs:app-mass-P-sigma-delta}, we can deduce that  
\begin{align*}
		& \frac{1}{2}\frac{\mathrm d}{\mathrm dt}\int_{\mathbb{T}^{3}} \rho_\sigma|\nabla \ln \rho_\sigma|^2\, \mathrm{d} x+\int_{\mathbb{T}^{3}} \nabla\Div\mathbf{u}_\sigma\cdot\nabla\rho_\sigma\, \mathrm{d} x+\int_{\mathbb{T}^{3}} \rho_\sigma\bbd\mathbf{u}_\sigma:\nabla \ln \rho_\sigma\otimes \nabla \ln \rho_\sigma\, \mathrm{d} x\\
        &\quad -\frac{\delta}{2}\int_{\mathbb{T}^3} \Delta \rho_\sigma|\nabla \ln \rho_\sigma|^2\,\mathrm{d}x + \delta \int_{\mathbb{T}^3} \frac{|\Delta\rho_\sigma|^2}{\rho_\sigma}\,\mathrm{d}x =0, \quad \text{a.e. in}\ \ (0,T).
	\end{align*}    
Combining the above identities, we find 
\begin{align}
& \frac{\mathrm{d}}{\dt}\int_{\mathbb{T}^3} \Big(\rho_\sigma\mathbf{u}_\sigma\cdot \nabla \ln \rho_\sigma + \frac12  \rho_\sigma|\nabla \ln \rho_\sigma|^2- \frac{1}{\alpha^2} \ln \rho_\sigma 
\Big)\,\mathrm{d}x + \frac{1}{4} \int_{\bbt^3} \rho_\sigma \absm{\nabla \bu_\sigma - (\nabla \bu_\sigma)^\top}^2 \,\dx
\notag \\
&\qquad + \zeta^2 \int_{\mathbb{T}^3} |\Delta \rho_\sigma|^2\, \dx + \delta \int_{\mathbb{T}^3} |\nabla\Delta\rho_\sigma|^2\,\dx + \delta \int_{\mathbb{T}^3} \frac{|\Delta\rho_\sigma|^2}{\rho_\sigma}\,\mathrm{d}x 
+ \delta \int_{\mathbb{T}^3}  W_\delta''(\rho_\sigma)|\nabla \rho_\sigma|^2 \,\mathrm{d}x
\notag \\
&\quad 
= \int_{\mathbb{T}^3} \rho_\sigma|\mathbb{D}\mathbf{u}_\sigma|^2\,\dx 
-\delta\int_{\mathbb{T}^3} (\nabla \rho_\sigma\cdot\nabla \mathbf{u}_\sigma)\cdot\nabla \ln \rho_\sigma\,\dx 
+ \frac{\delta}{2}\int_{\mathbb{T}^3} \Delta \rho_\sigma|\nabla \ln \rho_\sigma|^2\,\mathrm{d}x 
\notag \\
&\qquad 
-\delta\int_{\mathbb{T}^3} \frac{\Delta\rho_\sigma}{\rho_\sigma}  \Div(\rho_\sigma\mathbf{u}_\sigma)\,\dx
-\delta\int_{\mathbb{R}}\Delta \mathbf{u}_\sigma\cdot\nabla \Delta \ln \rho_\sigma\,\dx 
\notag \\
&\qquad - \int_{\mathbb{T}^3}    \widetilde{F}''_\sigma(\rho_\sigma)  |\nabla \rho_\sigma|^2 \,\mathrm{d}x 
- \frac{\delta}{\alpha^2} \int_{\bbt^3} \frac{\Delta \rho_\sigma}{\rho_\sigma} \,\dx.
\label{BD-pat-sigma}
\end{align}
Define 
\begin{align*}
	E_{\mathrm{BD}}^\sigma(\bu_\sigma, \rho_\sigma)
	& \coloneqq E_{\sigma,\delta} (\mathbf{u}_\sigma,\rho_\sigma) + \int_{\mathbb{T}^3} \Big(\rho_\sigma\mathbf{u}_\sigma\cdot \nabla \ln \rho_\sigma + \frac12  \rho_\sigma|\nabla \ln \rho_\sigma|^2\Big)\,\mathrm{d}x 
    \\
    & \ = \int_{\bbt^3}
	\Big( 
	\frac{\rho_\sigma}{2} \abs{\bu_\sigma + \nabla \ln \rho_\sigma}^2
	+ \frac{\zeta^2}{2}|\nabla \rho_\sigma|^2 
    + \widetilde{F}_\sigma(\rho_\sigma) 
	  + \frac{\delta}{2} \absm{\Delta \rho_\sigma}^2 
	+ \delta W_\delta(\rho_\sigma)
	\Big) \,\dx.
\end{align*}
Integrating \eqref{BD-pat-sigma} with respect to time and combining the resultant with the energy estimates
\begin{align}
 \sup_{0 \leq t \leq T} \left(E_{\sigma,\delta}(\bu_\sigma(t),\rho_\sigma(t)) + C_* \right) \leq 
 \left(E_{\sigma,\delta}(\bu_0,\rho_{0,\delta}) + C_* \right) e^{2\omega\delta T}, 
 \notag 
\end{align}
and 
\begin{align}
	\nonumber
	& \sup_{0 \leq t \leq T} \left(E_{\sigma,\delta}(\bu_\sigma(t),\rho_\sigma(t)) \right)
	+ \int_0^T \int_{\bbt^3} \rho_\sigma \absm{\bbd \bu_\sigma}^2 \,\dxdt
	+ \frac{1}{\alpha^2} \int_0^T \int_{\bbt^3} \absm{\nabla \Delta^{-1} \Div \bu_\sigma}^2 \,\dxdt \\
	\nonumber
	&\qquad  + \delta \int_0^T \int_{\bbt^3} \absm{\Delta \bu_\sigma}^2 \,\dxdt
	+ \delta^2 \int_0^T \int_{\bbt^3} \absm{\nabla\Delta \rho_\sigma}^2 \,\dxdt
	+ \delta \zeta^2 \int_0^T \int_{\bbt^3} \absm{\Delta \rho_\sigma}^2 \,\dxdt 
    \\
	& \quad \leq  E_{\sigma,\delta}(\bu_{0},\rho_{0,\delta})
    + \delta \omega \zeta^2 \int_0^T\int_{\mathbb{T}^3} |\nabla \rho_\sigma|^2\,\dxdt,     
	\notag 
    \end{align}
we can deduce that 
\begin{align}
&\sup_{0 \leq t \leq T} \left(E^{\sigma}_{\mathrm{BD}}(\bu_\sigma(t),\rho_\sigma(t))- \frac{1}{\alpha^2} \int_{\bbt^3} \ln \rho_\sigma(t) \,\dx\right)
+ \delta \int_0^T \int_{\bbt^3} \absm{\Delta \bu_\sigma}^2 \,\dxdt
\notag\\
&\qquad +  \frac{1}{\alpha^2} \int_0^T \int_{\bbt^3} \absm{\nabla \Delta^{-1} \Div \bu_\sigma}^2 \,\dxdt +
\frac{1}{4} \int_0^T\int_{\bbt^3} \rho_\sigma \absm{\nabla \bu_\sigma - (\nabla \bu_\sigma)^\top}^2 \,\dxdt
 \notag\\
&\qquad 
    + (1+\delta) \zeta^2 \int_0^T \int_{\bbt^3} \absm{\Delta \rho_\sigma}^2 \,\dxdt  
    + (\delta+\delta^2) \int_0^T \int_{\bbt^3} \absm{\nabla\Delta \rho_\sigma}^2 \,\dxdt 
    \notag \\
&\qquad + \delta \int_0^T\int_{\mathbb{T}^3}  \frac{|\Delta\rho_\sigma|^2}{\rho_\sigma}\,\dxdt 
+ \delta \int_0^T \int_{\bbt^3} W_\delta''(\rho_\sigma) \abs{\nabla \rho_\sigma}^{2} \,\dx\dt
\notag \\
&\quad \leq E^{\sigma}_{\mathrm{BD}}(\bu_0,\rho_{0,\delta}) - \frac{1}{\alpha^2} \int_{\bbt^3} \ln \rho_{0,\delta} \,\dx 
\notag \\
& \qquad 
\underbrace{-\delta \int_0^T \int_{\bbt^3} (\nabla \rho_\sigma \cdot \nabla \bu_\sigma) \cdot \frac{\nabla \rho_\sigma}{\rho_\sigma} \,\dx\dt}_{=:R_1}
	+ \underbrace{\frac{\delta}{2} \int_0^T \int_{\bbt^3} \Delta \rho_\sigma \frac{\absm{\nabla \rho_\sigma}^2}{ \rho_\sigma^2} \,\dx\dt}_{=:R_2} 
    \notag \\
	\nonumber
	& \qquad 
	 \underbrace{-\delta \int_0^T \int_{\bbt^3} \frac{\Delta \rho_\sigma}{\rho_\sigma} \Div(\rho_\delta \bu_\sigma) \,\dx\dt}_{=:R_3}
	 \underbrace{-\delta \int_0^T \int_{\bbt^3} \Delta \bu_\sigma \cdot \nabla \Delta \ln \rho_\sigma \,\dx\dt}_{=:R_4} \\
    &\qquad + \underbrace{(1+\delta) \omega \zeta^2 \int_0^T\int_{\mathbb{T}^3} |\nabla \rho_\sigma|^2\,\dxdt}_{=:R_5}
    \underbrace{- \frac{\delta}{\alpha^2} \int_0^T\int_{\bbt^3} \frac{\Delta \rho_\sigma}{\rho_\sigma} \,\dxdt}_{=:R_6}. 
    \label{eqs:BD-sigma-0}
\end{align}
Here for $R_5$, we have used the fact that $\widetilde{F}_\sigma''\geq -\omega\zeta^2$. 

In the following, we estimate the terms on the right-hand side of \eqref{eqs:BD-sigma-0}. Using \eqref{UL-rho-sigma}, H\"older's inequality, Young's inequality, the Gagliardo--Nirenberg inequality
$$
\normm{\nabla \rho_\sigma}_{L^4(\bbt^3)} \leq \normm{ \rho_\sigma}_{L^\infty(\bbt^3)}^\frac{5}{6} \normm{ \rho_\sigma}_{H^3(\bbt^3)}^\frac{1}{6}, 
$$
and the elliptic estimate 
$$
\|\rho_\sigma\|_{H^3(\mathbb{T}^3)}
\leq C(\|\nabla \Delta \rho_\sigma\|_{L^2(\mathbb{T}^3)} +\|\rho_\sigma\|_{L^2(\mathbb{T}^3)}),
$$
we find that   
\begin{align}
	\absm{R_1} & \leq \frac{2\delta}{\underline{\rho}} \int_0^T \normm{\nabla \bu_\sigma}_{L^2(\mathbb{T}^3)} \norm{\nabla \rho_\sigma}_{L^4(\mathbb{T}^3)}^2 \,\dt\nonumber\\
    & \leq \frac{\delta}{8}\int_0^T \normm{\nabla\Delta \rho_\sigma }_{L^2(\mathbb{T}^3)}^2\,\dt +  C(T,\underline{\rho}, E_{\sigma,\delta}(\bu_0,\rho_{0,\delta}),C_*).
    \notag 
\end{align}
In a similar way, we get 
\begin{align}
	\absm{R_2} & \leq \frac{2\delta }
   {\underline{\rho}^2}
    \int_0^T \normm{\Delta \rho_\sigma}_{L^2(\mathbb{T}^3)} \norm{\nabla \rho_\sigma}_{L^4(\mathbb{T}^3)}^2\,\dt\nonumber\\
	& \leq  \frac{\delta}{8}\int_0^T \normm{\nabla\Delta \rho_\sigma }_{L^2(\mathbb{T}^3)}^2\,\dt 
    + \frac{\zeta^2}{4} \int_0^T \normm{\Delta \rho_\sigma}_{L^2(\mathbb{T}^3)}^2 \,\dt + C(\underline{\rho},\zeta)T.\notag 
\end{align}
In the above estimates for $R_1$, $R_2$, we have used the fact that $\delta\in (0,1)$. Next, employing the following interpolation inequalities,
\begin{align*}
  &\normm{\rho_\sigma}_{H^2(\bbt^3)}  \leq \normm{\rho_\sigma}_{L^\infty(\bbt^3)}^\frac23 \normm{\rho_\sigma}_{H^3(\bbt^3)}^\frac13,
  \\
  &\normm{\rho_\sigma}_{W^{2,3}(\bbt^3)}  \leq \normm{\rho_\sigma}_{L^\infty(\bbt^3)}^\frac13 \normm{\rho_\sigma}_{H^3(\bbt^3)}^\frac23,
\end{align*}
we can deduce that 
\begin{align}
	\absm{R_3} 
	& \leq \abs{\delta \int_0^T \int_{\bbt^3} \frac{\Delta \rho_\sigma}{\rho_\sigma} (\bu_\sigma \cdot \nabla \rho_\sigma )\,\dx\dt}
	+ \abs{\delta \int_0^T \int_{\bbt^3} \Delta \rho_\sigma \Div \bu_\sigma  \,\dx\dt} \nonumber
    \\
	& \leq \frac{2\delta}{\underline{\rho}}\int_0^T \normm{\bu_\sigma}_{L^4(\mathbb{T}^3)} \normm{\nabla\rho_\sigma}_{L^4(\mathbb{T}^3)} \normm{\Delta \rho_\sigma}_{L^2(\mathbb{T}^3)} \,\dt 
	+ \delta\int_0^T \normm{\Div \bu_\sigma}_{L^2(\mathbb{T}^3)}\normm{\Delta \rho_\sigma}_{L^2(\mathbb{T}^3)}\,\dt \nonumber
    \\
	& \leq C(\underline{\rho}) \delta \int_0^T \normm{\bu_\sigma}_{L^2(\mathbb{T}^3)}^\frac58 \normm{\bu_\sigma}_{H^2(\mathbb{T}^3)}^\frac38 \normm{\rho_\sigma}_{L^\infty(\mathbb{T}^3)}^\frac{3}{2} \normm{ \rho_\sigma}_{H^3(\mathbb{T}^3)}^\frac{1}{2} \,\dt
    \notag\\
    &\quad 
	+ C \delta\int_0^T \normm{\nabla  \bu_\sigma}_{L^2(\mathbb{T}^3)} \normm{\Delta \rho_\sigma}_{L^2(\mathbb{T}^3)}\,\dt 
    \nonumber\\
	& \leq \frac{\delta}{4} \int_0^T \normm{\Delta \bu_\sigma}_{L^2(\mathbb{T}^3)}^2\, \dt 
	+\frac{\zeta^2}{4} \int_0^T \normm{\Delta \rho_\sigma}_{L^2(\mathbb{T}^3)}^2 \,\dt + \frac{\delta}{8}\int_0^T \normm{\nabla\Delta \rho_\sigma }_{L^2(\mathbb{T}^3)}^2\,\dt  
    \notag \\
	&\quad + C(T,\underline{\rho}, E_{\sigma,\delta}(\bu_0,\rho_{0,\delta}),C_*),
    \notag 
\end{align}
and 
\begin{align}
	\absm{R_4} 
	& \leq \delta  \int_0^T \int_{\bbt^3} \|\Delta \bu_\sigma\|_{L^2(\mathbb{T}^3)} 
    \left\|
    \frac{\nabla\Delta \rho_\sigma}{\rho_\sigma}
    - \frac{\Delta\rho_\sigma\nabla \rho_\sigma}{\rho_\sigma^2}-\frac{\nabla |\nabla \rho_\sigma|^2}{\rho_\sigma^2}+ \frac{2|\nabla \rho_\sigma|^2\nabla \rho_\sigma}{\rho_\sigma^3}\right\|_{L^2(\mathbb{T}^3)}\,\dx\dt
	\notag \\
    &
    \leq C(\underline{\rho}) \delta \int_0^T \int_{\bbt^3} \|\Delta \bu_\sigma\|_{L^2(\mathbb{T}^3)} 
    \left(
    \|\nabla\Delta \rho_\sigma\|_{L^2(\mathbb{T}^3)}
    + \|\nabla \rho_\sigma\|^3_{L^6(\mathbb{T}^3)}\right)\,\dx\dt\notag 
    \\
     &
    \quad + C(\underline{\rho}) \delta \int_0^T \int_{\bbt^3} \|\Delta \bu_\sigma\|_{L^2(\mathbb{T}^3)} 
    \|\rho_\sigma\|_{W^{2,3}(\mathbb{T}^3)}\|\nabla \rho_\sigma\|_{L^6(\mathbb{T}^3)} \,\dx\dt\notag 
    \\ 
    &	\leq \frac{\delta}{4} \int_0^T \normm{\Delta \bu_\sigma}_{L^2(\mathbb{T}^3)}^2\, \dt 
    + \frac{\delta}{8}\int_0^T \normm{\nabla\Delta \rho_\sigma }_{L^2(\mathbb{T}^3)}^2\,\dt 
	 + C(T,\underline{\rho}, E_{\sigma,\delta}(\bu_0,\rho_{0,\delta}),C_*).
    \notag 
\end{align}
The term $R_5$ can easily be estimated by  
\begin{align}
 |R_5|\leq  C(T,\omega, E_{\sigma,\delta}(\bu_0,\rho_{0,\delta}),C_*),
 \notag 
\end{align}
while for $R_6$, we have 
\begin{align*}
	\abs{R_6}
	\leq \frac{\delta}{\alpha^2} \left(\int_{\bbt^3} \frac{\absm{\Delta \rho_\sigma}^2}{\rho_\sigma} \,\dx\right)^{\onehalf} \left\|\rho_\delta^{-\frac12}\right\|_{L^2(\bbt^3)}
	\leq \frac{\delta}{2} \int_{\bbt^3} \frac{\absm{\Delta \rho_\sigma}^2}{\rho_\sigma} \,\dx 
	+ \frac{\sqrt{2}\delta}{2\alpha^4\sqrt{\underline{\rho}}}.
\end{align*}
Substituting the above estimates for $R_1$, ..., $R_6$ into \eqref{eqs:BD-sigma-0}, we obtain 
\begin{align}
&\sup_{0 \leq t \leq T}  E^{\sigma}_{\mathrm{BD}}(\bu_\sigma(t),\rho_\sigma(t)) 
+ \frac{\delta}{2} \int_0^T \int_{\bbt^3} \absm{\Delta \bu_\sigma}^2 \,\dxdt  
+  \frac{1}{\alpha^2} \int_0^T \int_{\bbt^3} \absm{\nabla \Delta^{-1} \Div \bu_\sigma}^2 \,\dxdt
\notag\\
&\qquad +
\frac{1}{4} \int_0^T\int_{\bbt^3} \rho_\sigma \absm{\nabla \bu_\sigma - (\nabla \bu_\sigma)^\top}^2 \,\dxdt 
+ \frac{\zeta^2}{2} \int_0^T \int_{\bbt^3} \absm{\Delta \rho_\sigma}^2 \,\dxdt  
 \notag\\
&\qquad 
    + \frac{\delta}{2} \int_0^T \int_{\bbt^3} \absm{\nabla\Delta \rho_\sigma}^2 \,\dxdt 
     + \frac{\delta}{2} \int_0^T\int_{\mathbb{T}^3} \frac{|\Delta\rho_\sigma|^2}{\rho_\sigma}\,\dxdt  
     + \delta \int_0^T \int_{\bbt^3} W_\delta''(\rho_\sigma) \abs{\nabla \rho_\sigma}^{2} \,\dx\dt
\notag \\
&\quad \leq E^{\sigma}_{\mathrm{BD}}(\bu_0,\rho_{0,\delta}) +  C(T,\omega, \underline{\rho}, E_{\sigma,\delta}(\bu_0,\rho_{0,\delta}),C_*).  
    \label{eqs:BD-entropy-delta-0}
\end{align}
By the construction of $\rho_{0,\delta}$ and $\widetilde{F}_\sigma$, the right-hand side of \eqref{eqs:BD-entropy-delta-0} is bounded by a positive constant independent of $\sigma, \delta$. 

Since the solution $(\mathbf{u}_\delta,\rho_\delta)$ is the limit of a sequence of approximate solutions $(\bu_\sigma,\rho_\sigma)$ constructed in Proposition \ref{prop:weak-solution-sigma}, after passing to the limit as $\sigma \to 0$ in \eqref{eqs:BD-entropy-delta-0}, we arrive at the conclusion \eqref{eqs:BD-entropy-delta-1}. 
\end{proof}

Using the construction of $\rho_{0,\delta}$ again (especially \eqref{eqs:rho0-bounds} and \eqref{eqs:rho0-convergence}), we find that $E^{\delta}_{\mathrm{BD}}(\bu_0,\rho_{0,\delta})$ can be bounded by a constant independent of $0<\delta\ll 1$. 
Then the BD-entropy estimate \eqref{eqs:BD-entropy-delta-1} yields that 
\begin{align}
	\label{eqs:uniform-rho}
	\rho_\delta \ \ \text{is uniformly bounded in} \ \ L^2(0,T;H^2(\bbt^3)). 
\end{align}
Applying \eqref{eqs:app-mass-sigma1}, \eqref{eqs:uniform-rho}, H\"{o}lder's inequality and the Sobolev theorem, we get   %
\begin{align}
\|\partial_t\rho_\delta\|_{L^2(\mathbb{T}^3)}
& \lesssim \|\nabla \rho_\delta\|_{L^6(\mathbb{T}^3)}\|\mathbf{u}_\delta\|_{L^3(\mathbb{T}^3)}
+\|\rho_\delta\|_{L^\infty(\mathbb{T}^3)}\|\Div \mathbf{u}_\delta\|_{L^2(\mathbb{T}^3)}
+\delta\|\Delta \rho_\delta\|_{L^2(\mathbb{T}^3)}\notag\\
&\lesssim \|\rho_\delta\|_{H^2(\mathbb{T}^3)}\|\mathbf{u}_\delta\|_{H^1(\mathbb{T}^3)}^\frac12\|\mathbf{u}_\delta\|_{L^2(\mathbb{T}^3)}^\frac12+
\|\rho_\delta\|_{L^\infty(\mathbb{T}^3)}\| \mathbf{u}_\delta\|_{H^1(\mathbb{T}^3)}+ \delta \|\Delta \rho_\delta\|_{L^2(\mathbb{T}^3)},
\notag \\
\|\partial_t\rho_\delta\|_{L^\frac32(\mathbb{T}^3)}
& \lesssim \|\nabla \rho_\delta\|_{L^3(\mathbb{T}^3)}\|\mathbf{u}_\delta\|_{L^3(\mathbb{T}^3)}
+\|\rho_\delta\|_{L^\infty(\mathbb{T}^3)}\|\Div \mathbf{u}_\delta\|_{L^2(\mathbb{T}^3)}
+\delta\|\Delta \rho_\delta\|_{L^2(\mathbb{T}^3)}\notag\\
&\lesssim \|\rho_\delta\|_{H^2(\mathbb{T}^3)}^\frac12 \|\rho_\delta\|_{H^1(\mathbb{T}^3)}^\frac12\|\mathbf{u}_\delta\|_{H^1(\mathbb{T}^3)}^\frac12\|\mathbf{u}_\delta\|_{L^2(\mathbb{T}^3)}^\frac12+
\|\rho_\delta\|_{L^\infty(\mathbb{T}^3)}\| \mathbf{u}_\delta\|_{H^1(\mathbb{T}^3)}\notag\\
&\quad + \delta \|\Delta \rho_\delta\|_{L^2(\mathbb{T}^3)}.
\notag 
\end{align}
From the above estimates, we can deduce that   
\begin{align}
\partial_t\rho_\delta\ \ \text{is uniformly bounded in} \ \ L^\frac{4}{3}(0,T;L^2(\mathbb{T}^3))\cap L^{2}(0,T;L^{\frac{3}{2}}(\bbt^3)).
\label{eqs:uniform-pt-rho}
\end{align}
Nevertheless, the estimate for $\partial_t(\rho_\delta\mathbf{u}_\delta)$ requires a suitable control of the pressure $\widetilde{P}(\rho_\delta)$.     

\subsection{Integrability of the pressure}
\label{sec:integrability-pressure-delta}
We now apply the BD-entropy estimate to derive the uniform  integrability of $\widetilde{P}(\rho_\delta)$. Compared to Lemma \ref{lem:uniform-P}, the key point is that the subsequent estimate is independent of $\delta$. 
\begin{lemma}
	\label{lem:uniform-P-delta}
	Let $(\bu_\delta,\rho_\delta)$ be an approximate solution given by Proposition \ref{prop:weak-solution-delta}. Then we have 
	\begin{align}
		\label{eqs:uniform-P-delta}
		\int_0^{T} \int_{\bbt^{3}} \abs{\widetilde{P}(\rho_\delta)}\, \dxdt \leq C, 
	\end{align}
	where $C > 0$ is independent of $\delta$.
\end{lemma}
\begin{proof}
	Choosing the test function $\bphi$ in \eqref{eqs:weak-formulation-delta} as 
	\begin{align*}
		\bphi=\psi(t) \cB\big[\rho_{\delta}^{(0)}\big], \quad \psi \in C_c^\infty((0,T))\quad \text{and}\quad \rho_{\delta}^{(0)} \coloneqq \zeromean{\rho_\delta},
	\end{align*}
	we obtain 
	\begin{align*}
		&\int_0^T \psi(t) \int_{\mathbb{T}^{3}} \widetilde{P}(\rho_\delta)  \rho_{\delta}^{(0)} \,\dxdt 
        \\ 
        & = - \int_0^{T} \psi'(t) \int_{\mathbb{T}^{3}} \rho_\delta \bu_\delta \cdot \cB \big[\rho_{\delta}^{(0)}\big]\,\dxdt 
        - \int_0^{T} \psi(t) \int_{\mathbb{T}^{3}} \rho_\delta \bu_\delta \cdot \cB \big[\pt\rho_{\delta}^{(0)}\big]\,\dxdt 
        \\
		& \quad  
        - \int_0^{T} \psi(t) \int_{\mathbb{T}^{3}} \rho_\delta (\bu_\delta \otimes \bu_\delta) : \nabla \cB \big[\rho_{\delta}^{(0)}\big]\,\dxdt 
        + \int_0^{T} \psi(t) \int_{\mathbb{T}^{3}} \rho_\delta \bbd \bu_\delta : \nabla \cB \big[\rho_{\delta}^{(0)}\big]\,\dxdt 
        \\
		 & \quad -  \delta \int_0^{T} \psi(t) \int_{\mathbb{T}^{3}} H(\rho_\delta)  \rho_{\delta}^{(0)} \,\dxdt 
         +\delta \int_0^{T} \psi(t) \int_{\mathbb{T}^{3}}   \nabla^2\rho_{\delta}: \nabla\Delta\rho_{\delta} \otimes\cB \big[\rho_\delta^{(0)}\big]\,\dxdt
         \\
      	&\quad +\delta \int_0^{T} \psi(t) \int_{\mathbb{T}^{3}}   \nabla\Delta\rho_{\delta} \otimes\nabla\rho_{\delta} :\nabla \cB \big[\rho_\delta^{(0)}\big]\,\dxdt
        +\delta\int_0^{T} \psi(t) \int_{\mathbb{T}^{3}}   (\nabla\rho_{\delta} \cdot\nabla\Delta\rho_{\delta})  \rho_\delta^{(0)}\,\dxdt
       \\
       &\quad + \delta \int_0^{T} \psi(t) \int_{\mathbb{T}^{3}}   \rho_{\delta} \nabla\Delta\rho_{\delta} \cdot \nabla\rho_\delta^{(0)}\,\dxdt 
        +\delta\int_0^{T} \psi(t) \int_{\mathbb{T}^{3}}   (\nabla\rho_{\delta} \cdot\nabla \bu_\delta) \cdot \cB \big[\rho_\delta^{(0)}\big]\,\dxdt
      \\
      &\quad 
         +\delta \int_0^{T} \psi(t) \int_{\mathbb{T}^{3}} \Delta \bu_\delta  \cdot \Delta\cB \big[\rho_{\delta}^{(0)}\big]\,\dxdt 
         + \zeta^2  \int_0^{T} \psi(t) \int_{\mathbb{T}^{3}}  \rho_\delta \Delta \rho_\delta  \rho_{\delta}^{(0)} \,\dxdt 
         \\
		 & \quad  + \zeta^2 \int_0^{T} \psi(t) \int_{\mathbb{T}^{3}}  \Delta \rho_\delta \nabla \rho_\delta \cdot \cB  \big[\rho_{\delta}^{(0)}\big] \,\dxdt  
          - \frac{1}{\alpha^2}\int_0^{T} \psi(t) \int_{\mathbb{T}^{3}}  \Delta^{-1} \Div \bu_\delta  \rho_{\delta}^{(0)} \,\dxdt. 
	\end{align*}
Combining the estimates \eqref{eqs:uniform-u-del}, \eqref{eqs:uniform-rho-del} and \eqref{eqs:uniform-rho}, we can conclude 
\begin{align}
		\label{eqs:P-rho-delta}
		\abs{\int_0^T \psi(t)\int_{\mathbb{T}^{3}} \widetilde{P}_\delta(\rho_{\delta}) \rho_\delta^{(0)}\,\dxdt} 
        \leq C(\|\psi\|_{L^{\infty}(0,T)}+\|\psi'\|_{L^{1}(0,T)}),
	\end{align} 
where $C>0$ is independent of $0<\delta\ll 1$. Here we only mention estimates that are essentially based on the BD-entropy estimate  (more precisely, \eqref{eqs:uniform-rho}):  
\begin{align*}
 & \delta \left|\int_0^{T} \psi(t) \int_{\mathbb{T}^{3}}   \nabla^2\rho_{\delta}:\nabla\Delta\rho_{\delta} \otimes\cB \big[\rho_\delta^{(0)}\big]\,\dxdt \right|  \\
 &\quad \leq \|\psi\|_{L^\infty(0,T)}\norm{\cB \big[\rho_\delta^{(0)}\big]}_{L^\infty(0,T;L^\infty(\mathbb{T}^3))}  \|\nabla^2\rho_{\delta}\|_{L^2(0,T;L^2(\mathbb{T}^3))}
  \left(\delta \|\nabla\Delta\rho_{\delta}\|_{L^2(0,T;L^2(\mathbb{T}^3))}
 \right),
\end{align*}
\begin{align*}
& \zeta^2 \left| \int_0^{T} \psi(t) \int_{\mathbb{T}^{3}}  \Delta \rho_\delta \nabla \rho_\delta \cdot \cB  \big[\rho_{\delta}^{(0)}\big] \,\dxdt\right|
  \\
&\quad \leq \zeta^2 \|\psi\|_{L^\infty(0,T)}\norm{\cB \big[\rho_\delta^{(0)}\big]}_{L^\infty(0,T;L^\infty(\mathbb{T}^3))}  \|\Delta\rho_{\delta}\|_{L^2(0,T;L^2(\mathbb{T}^3))}
   \|\nabla \rho_{\delta}\|_{L^2(0,T;L^2(\mathbb{T}^3))}.
\end{align*}
Applying an argument similar to that used for Lemma \ref{lem:uniform-P}, we can conclude \eqref{eqs:uniform-P-delta} from \eqref{eqs:P-rho-delta}. The details are omitted for the sake of brevity. 
\end{proof}

\subsection{Equi-integrability of the pressure}
\label{sec:equi-integrability}
Next, we apply the BD-entropy estimate to establish the equi-integrability of $\{\widetilde{P}(\rho_\delta)\}$ for $0< \delta \ll 1$. For this purpose, let us introduce the function 
\begin{align}\notag 
    \chi(r) = \ln \left(\frac{r - \underline{\rho}}{1 - r}\right),\ \quad  r\in\big(\underline{\rho},1\big)
\end{align}
and consider the following test function
\begin{align*}
	\bw_\delta(x,t) 
	=\psi(t) \cB\left[\zeromean{\chi(\rho_\delta)}\right]=:\psi(t)\cB\big[\chi^{(0)}_{\delta}\big],\quad \psi \in C_c^\infty((0,T)).
\end{align*}
In view of the bound \eqref{rho-UpLo-delta} for $\rho_\delta$, here we do not need a truncation like that for $\chi_\sigma$ (cf. \eqref{cut function-1}). According to Assumption \ref{main assumption}, it follows that 
\begin{align*}
	\abs{\chi(\rho_\delta(x,t))}^{p} & \leq C \abs{\widetilde{F} (\rho_\delta(x,t))} + C, \quad \text{ for all }\ 1 \leq p < \infty, \\
	\abs{\chi' (\rho_\delta(x,t))}^{\beta} & \leq C\abs{\widetilde{F} (\rho_\delta(x,t))} + C, \\
	\abs{\chi' (\rho_\delta(x,t))}^{\beta+1} & \leq C \abs{\widetilde{P} (\rho_\delta(x,t))} + C.
\end{align*}
These estimates, together with \eqref{eqs:uniform-FWH-del} and \eqref{eqs:uniform-P-delta} imply that 
\begin{alignat}{3}
	\label{sigular1-delta}
	\chi(\rho_\delta)  & \text{ is uniformly bounded in } && L^\infty(0,T;L^p(\mathbb{T}^{3})), \quad \text{for all }\ 1 \leq p < \infty, \\
	\label{sigular2-delta}
	\chi'(\rho_\delta) & \text{ is uniformly bounded in } && L^\infty(0,T;L^\beta(\mathbb{T}^{3})) \cap  L^{\beta+1}(Q_T),\\
    \label{B-bound-chi} \cB\big[\chi^{(0)}_\delta\big]  
    & \text{ is uniformly bounded in } && L^{\infty}(0,T;W^{1,p}(\bbt^3)),\text{ for all $1\leq p<\infty$}.
\end{alignat}
In addition, the following estimate holds    
\begin{align}
\big\|\sqrt{\delta}\nabla\chi^{(0)}_\delta\big\|_{L^{2}(0,T;L^{2}(\mathbb{T}^{3}))}^2
&\leq\|\sqrt{\delta}\nabla\rho_\delta\|_{L^{4}(0,T;L^{\infty}(\mathbb{T}^{3}))}^2
\|\chi'(\rho_\delta)\|_{L^{4}(0,T;L^{2}(\mathbb{T}^{3}))}^2
\notag \\
&\leq\|\sqrt{\delta}\nabla\rho_\delta\|_{L^{\infty}(0,T;H^{1}(\mathbb{T}^{3}))}\|\sqrt{\delta}\nabla\rho_\delta\|_{L^{2}(0,T;H^{2}(\mathbb{T}^{3}))} 
\notag \\
&\quad\times \|\chi'(\rho_\delta)\|_{L^{\infty}(0,T;L^{\frac{3}{2}}(\mathbb{T}^{3}))}^{\frac{3}{4}}\|\chi'(\rho_\delta)\|_{L^{\frac{5}{2}}(0,T;L^{\frac{5}{2}}(\mathbb{T}^{3}))}^{\frac{5}{4}},
\label{B-bound-chi-n1}
\end{align}
where we have used \eqref{eqs:BD-entropy-delta-1}, \eqref{sigular2-delta}, Agmon's inequality (in three dimensions) and the assumption $\beta\geq 3/2$. The estimate \eqref{B-bound-chi-n1} also implies 
\begin{alignat}{3}
\label{B-bound-chi-1}
\sqrt{\delta}\cB\big[\chi^{(0)}_\delta\big]  & \text{ is uniformly bounded in }  && L^{2}(0,T;H^{2}(\bbt^3)).
\end{alignat}

Now we prove the equi-integrability of the pressure.
\begin{lemma}
	\label{lem:uniform-P-chi}
	Let  $(\bu_\delta,\rho_\delta)$ be an approximate solution given by Proposition \ref{prop:weak-solution-delta}. Then we have 
	\begin{align}
		\label{eqs:uniform-P-chi}
		\int_0^{T} \int_{\bbt^{3}} \abs{\widetilde{P}(\rho_\delta) \chi(\rho_\delta)} \,\dxdt \leq C,
	\end{align}
	where $C > 0$ is independent of $\delta$.
\end{lemma}
\begin{proof}
	Choosing the test function $\bphi=\bw_\delta$ in  \eqref{eqs:weak-formulation-delta}, we obtain
	\begin{align}
		\nonumber
		&\int_0^{T} \psi(t) \int_{\mathbb{T}^{3}} \widetilde{P}(\rho_\delta) \chi^{(0)}_{\delta} \,\dxdt 
        \\
		\nonumber
		& \quad = -\int_0^{T} \psi'(t) \int_{\mathbb{T}^{3}}\rho_\delta \bu_\delta \cdot \cB \big[\chi^{(0)}_{\delta}\big]\,\dxdt 
		- \int_0^{T} \psi(t) \int_{\mathbb{T}^{3}} \rho_\delta \bu_\delta \cdot \cB \big[\pt \chi^{(0)}_{\delta}\big]\,\dxdt 
        \\
		& \nonumber \qquad
		- \int_0^{T} \psi(t) \int_{\mathbb{T}^{3}} \rho_\delta (\bu_\delta \otimes \bu_\delta) :\nabla \cB\big[\chi^{(0)}_{\delta}\big]\,\dxdt
		+ \int_0^{T} \psi(t) \int_{\mathbb{T}^{3}} \rho_\delta \bbd \bu_\delta : \nabla \cB \big[\chi^{(0)}_{\delta}\big]\,\dxdt 
		\\
        &\nonumber\qquad
		- \delta  \int_0^{T} \psi \int_{\mathbb{T}^{3}}  H(\rho_\delta)  \chi^{(0)}_{\delta} \,\dxdt
		+ \delta\int_0^{T} \psi(t) \int_{\mathbb{T}^{3}}   \nabla^2\rho_{\delta}:\nabla\Delta\rho_{\delta} \otimes \cB\big[\chi_\delta^{(0)}\big] \,\dxdt 
		\\
        &\qquad 
		  + \delta\int_0^{T} \psi(t) \int_{\mathbb{T}^{3}}   \nabla\Delta\rho_{\delta} \otimes\nabla\rho_{\delta} : \nabla\cB\big[\chi_\delta^{(0)}\big]  \,\dxdt
        + \delta\int_0^{T} \psi(t) \int_{\mathbb{T}^{3}}   \nabla\rho_{\delta} \cdot\nabla\Delta\rho_{\delta} \chi_\delta^{(0)}\,\dxdt
        \notag \\
        &\qquad +\delta \int_0^{T} \psi(t) \int_{\mathbb{T}^{3}}    \rho_{\delta} \nabla\Delta\rho_{\delta} \cdot \nabla\chi_\delta^{(0)}\,\dxdt
        + \delta\int_0^{T} \psi(t) \int_{\mathbb{T}^{3}} (\nabla \rho_\delta \cdot\nabla\bu_\delta) \cdot \cB\big[\chi_\delta^{(0)}\big] \,\dxdt
		\notag \\
        &\nonumber\qquad	
        + \delta\int_0^{T} \psi(t) \int_{\mathbb{T}^{3}} \Delta \bu_\delta  \cdot \Delta\cB\big[\chi_\delta^{(0)}\big] \,\dxdt
        + \zeta^2\int_0^{T} \psi(t) \int_{\mathbb{T}^{3}}\Delta \rho_\delta \nabla\rho_\delta   \cdot\cB\big[\chi_\delta^{(0)}\big]\, \dxdt
        \\
        &\nonumber\qquad
		  + \zeta^2\int_0^{T} \psi(t) \int_{\mathbb{T}^{3}}\rho_\delta \Delta \rho_\delta \chi_\delta^{(0)}\, \dxdt
        -\frac{1}{\alpha^2}\int_0^{T} \psi(t) \int_{\mathbb{T}^{3}} \Delta^{-1} \Div \bu_\delta \chi_\delta^{(0)}\,\dxdt. 
		\\
        &\label{eqs:P-chi-integral}
		\quad \eqqcolon \sum_{i = 1}^{14} K_i.
	\end{align}
In the following, we use the estimates \eqref{rho-UpLo-delta}, \eqref{eqs:uniform-u-del}, \eqref{eqs:uniform-rho-del}, \eqref{eqs:BD-entropy-delta-1}, \eqref{eqs:uniform-rho} and \eqref{sigular1-delta}--\eqref{B-bound-chi-1} to derive bounds for the right-hand side of \eqref{eqs:P-chi-integral}, which are uniform with respect to $\delta$.  

First, we infer from \eqref{rho-UpLo-delta}, \eqref{eqs:uniform-u-del} and \eqref{B-bound-chi} that 
	\begin{align}
		|K_1|& \nonumber\leq \|\psi'\|_{L^{1}(0,T)}\|\rho_\delta \bu_\delta\|_{L^{\infty}(0,T;L^{2}(\mathbb{T}^{3}))} 
\norm{\cB\big[\chi_\delta^{(0)}\big]}_{L^{\infty}(0,T;L^{2}(\mathbb{T}^{3}))}
		\notag 
		\\
        &\leq C\|\psi'\|_{L^{1}(0,T)}. \notag 
	\end{align}
Next, with the aid of equation \eqref{eqs:app-mass-sigma1}, we rewrite $K_2$ as 
	\begin{align*}
		K_2 
        & = -\int_0^{T}\psi (t)\int_{\mathbb{T}^{3}} \rho_\delta \bu_\delta \cdot \cB\left[\Div(\chi(\rho_\delta) \bu_\delta)\right]\,\dxdt 
        \\
        &\quad 
        +\int_0^{T}\psi(t) \int_{\mathbb{T}^{3}} \rho_\delta \bu_\delta \cdot \cB \left[\chi(\rho_\delta) \Div \bu_\delta-\mean{\chi(\rho_\delta) \Div \bu_\delta} \right]\,\dxdt \\
		& \quad - \int_0^{T}\psi(t)\int_{\mathbb{T}^{3}} \rho_\delta \bu_\delta \cdot \cB \left[\chi'(\rho_\delta)\rho_\delta \Div \bu_\delta - \mean{\chi'(\rho_\delta)\rho_\delta \Div \bu_\delta}\right]\,\dxdt 
        \\
        &\quad +  \delta\int_0^{T}\psi(t)\int_{\mathbb{T}^{3}} \rho_\delta \bu_\delta \cdot \cB \left[\chi'(\rho_\delta)\Delta \rho_\delta-\mean{\chi'(\rho_\delta)\Delta \rho_\delta} \right]\,\dxdt 
		 \\
		& = :\sum_{j=1}^4 K_{2,j}.
	\end{align*}
Then applying \eqref{eqs:Bogovskii-Lp-div}, \eqref{eqs:uniform-u-del}, \eqref{sigular1-delta} and \eqref{B-bound-chi}, we get 
	\begin{align}
        \notag 
		|K_{2,1}|& \nonumber\leq \|\psi\|_{L^{\infty}(0,T)}\|\rho_\delta \mathbf{u}_{\delta}\|_{L^{2}(0,T;L^{6}(\mathbb{T}^{3}))}
		\|\chi(\rho_\delta)\|_{L^{\infty}(0,T;L^{\frac{3}{2}}(\mathbb{T}^{3}))}\|\mathbf{u}_{\delta}\|_{L^{2}(0,T;L^{6}(\mathbb{T}^{3}))}
		\\& 
        \leq C\|\psi\|_{L^{\infty}(0,T)}.
        \notag 
	\end{align}
	In view of \eqref{eqs:Bogovskii-Lp} and \eqref{sigular1-delta}, we have
	\begin{align}
		\nonumber
		|K_{2,2}|&\leq \|\psi\|_{L^{\infty}(0,T)} \|\rho_{\delta}\mathbf{u}_{\delta}\|_{L^{2}(0,T;L^\frac{3}{2}(\mathbb{T}^{3}))}
		  \left\|\cB\left[ \chi(\rho_\delta) \Div\bu_\delta-\mean{\chi(\rho_\delta) \Div\bu_\delta} \right]\right\|_{L^{2}(0,T;L^{3}(\mathbb{T}^{3}))}
		\\ \nonumber
        & \leq \|\psi\|_{L^{\infty}(0,T)} 
        \|\rho_\delta\|_{L^\infty(0,T;L^\infty(\mathbb{T}^3))} 
        \|\mathbf{u}_{\delta}\|_{L^{2}(0,T;L^{2}(\mathbb{T}^{3}))}
        \|\chi(\rho_\delta) \Div\bu_\delta \|_{L^{2}(0,T;L^\frac{3}{2}(\mathbb{T}^{3}))}
		\\ \nonumber
        &\leq C\|\psi\|_{L^{\infty}(0,T)} \|\mathbf{u}_{\delta}\|_{L^{2}(0,T;L^{6}(\mathbb{T}^{3}))}
		\|\chi(\rho_\delta) 
        \|_{L^{\infty}(0,T;L^{6}(\mathbb{T}^{3}))}\| \Div\bu_\delta \|_{L^{2}(Q_T)}
        \\ 
        \notag 
		&\leq C\|\psi\|_{L^{\infty}(0,T)}.
	\end{align}
Using the Sobolev embedding theorem $W^{1,\frac{10}{9}}(\mathbb{T}^{3}) \hookrightarrow L^\frac{30}{17}(\mathbb{T}^{3})$, we infer from  \eqref{eqs:Bogovskii-Lp} that
	\begin{align*}
		& \norm{\cB\left[\chi'(\rho_\delta) \rho_\delta \Div \bu_\delta
        -\mean{\chi'(\rho_\delta) \rho_\delta \Div \bu_\delta} \right]}_{L^{\frac{30}{17}}(\mathbb{T}^{3})} \\
		& \quad \leq  C\norm{\cB\left[\chi'(\rho_\delta)\rho_\delta \Div \bu_\delta-\mean{\chi'(\rho_\delta) \rho_\delta \Div \bu_\delta}\right]}_{W^{1,\frac{10}{9}}(\mathbb{T}^{3})} \\
		&\quad  \leq C \norm{\rho_\delta}_{L^\infty(\mathbb{T}^{3})}\norm{\chi'(\rho_\delta)}_{L^{\frac{5}{2}}(\mathbb{T}^{3})} \norm{\Div \bu_\delta}_{L^2(\mathbb{T}^{3})}.
	\end{align*}
In addition, by interpolation, we have
$$
\|\mathbf{u}_\delta\|_{L^\frac{30}{13}(\mathbb{T}^3)}\leq C\|\mathbf{u}_\delta\|_{L^2(\mathbb{T}^3)}^\frac45\|\mathbf{u}_\delta\|_{L^6(\mathbb{T}^3)}^\frac15. 
$$
These estimates together with \eqref{eqs:Bogovskii-Lp-div}, \eqref{eqs:uniform-u-del} and  \eqref{sigular2-delta} yield  
	\begin{align}
		\nonumber
		\abs{K_{2,3}} 
        & \leq \|\psi\|_{L^{\infty}(0,T)} \|\rho_{\delta}\mathbf{u}_{\delta}\|_{L^{10}(0,T; L^{\frac{30}{13}}(\mathbb{T}^{3}))} \\
        &\quad \times \norm{\cB\left[\chi'(\rho_\delta) \rho_\delta \Div \bu_\delta-\mean{\chi'(\rho_\delta) \rho_\delta \Div \bu_\delta}\right]}_{L^\frac{10}{9}(0,T;L^{\frac{30}{17}}(\mathbb{T}^{3}))}
        \notag 
        \\
        & \notag \leq \|\psi\|_{L^{\infty}(0,T)}\|\rho_{\delta}\|_{L^\infty(0,T;L^\infty(\mathbb{T}^3))}^2 \|\mathbf{u}_{\delta}\|_{L^{\infty}(0,T; L^2(\mathbb{T}^{3}))}^\frac45
        \|\mathbf{u}_\delta \|_{L^2(0,T;H^1(\mathbb{T}^3))}^\frac65 \|\chi'(\rho_\delta)\|_{L^{\frac{5}{2}}(Q_T)} \\
        \notag 
		& \leq C\|\psi\|_{L^{\infty}(0,T)},
	\end{align}
    where we have used again the assumption $\beta\geq 3/2$. 
	Similarly, with \eqref{eqs:uniform-rho-del} and \eqref{eqs:uniform-rho} we can conclude 
    $$|K_{2,4}|\leq C\|\psi\|_{L^{\infty}(0,T)}.$$
    Combining the above estimates yields \begin{align}
	  |K_{2}|\leq C\|\psi\|_{L^{\infty}(0,T)}.\notag 
	\end{align}
    Next, using \eqref{sigular1-delta}, \eqref{B-bound-chi}, we easily find 
	\begin{align}
        \notag 
		|K_3| + |K_4|+ |K_{12}| + |K_{13}| + |K_{14}| 
        \leq 
 C\|\psi\|_{L^{\infty}(0,T)}.
	\end{align}
The term $K_5$ can be bounded using \eqref{H def}, \eqref{rho-UpLo-delta}, and \eqref{sigular1-delta}, i.e., 
$$|K_5|\leq C\|\psi\|_{L^{\infty}(0,T)}.
$$ 
Concerning $K_6$, we infer from \eqref{eqs:uniform-rho-del}, \eqref{eqs:uniform-rho} and \eqref{B-bound-chi} that 
\begin{align*}
|K_6|& \leq 
\|\psi(t)\|_{L^{\infty}(0,T)} \| \delta\nabla\Delta\rho_\delta\|_{L^{2}(Q_T)}\|\nabla^2\rho_\delta\|_{L^{2}(Q_T)}
\norm{\cB\big[\chi_\delta^{(0)}\big]}_{L^{\infty}(0,T;L^{\infty}(\mathbb{T}^{3}))}
\\
&\leq C \|\psi(t)\|_{L^{\infty}(0,T)} \| \delta\nabla\Delta\rho_\delta\|_{L^{2}(Q_T)}\|\nabla^2\rho_\delta\|_{L^{2}(Q_T)}
\norm{\cB\big[\chi_\delta^{(0)}\big]}_{L^{\infty}(0,T;W^{1,4}(\mathbb{T}^{3}))}
\\
&\leq C\|\psi(t)\|_{L^{\infty}(0,T)}.  
\end{align*}
In a similar manner, we can conclude 
\begin{align*}
|K_7|+|K_8|
&\leq \|\psi(t)\|_{L^{\infty}(0,T)}    \|\delta \nabla\Delta\rho_\delta\|_{L^{2}(Q_T)}\| \nabla\rho_\delta\|_{L^{2}(0,T;L^{6}(\mathbb{T}^{3}))}
		\| \chi_\sigma^{(0)} \|_{L^{\infty}(0,T;L^{3}(\mathbb{T}^{3}))}\\
        &\leq C \|\psi(t)\|_{L^{\infty}(0,T)}  \|\delta \nabla\Delta\rho_\delta\|_{L^{2}(Q_T)}\| \rho_\delta\|_{L^{2}(0,T;H^{2}(\mathbb{T}^{3}))}
		\| \chi_\sigma^{(0)} \|_{L^{\infty}(0,T;L^{3}(\mathbb{T}^{3}))}\\
    &\leq C \|\psi(t)\|_{L^{\infty}(0,T)}.
\end{align*}
For $K_9$, exploiting \eqref{rho-UpLo-delta}, \eqref{eqs:BD-entropy-delta-1} and \eqref{B-bound-chi-n1}, we get 
	\begin{align}\nonumber
		|K_9|& \leq \|\psi(t)\|_{L^{\infty}(0,T)} \|\rho_\delta\|_{L^{\infty}(0,T;L^{\infty}(\mathbb{T}^{3}))}\|\sqrt{\delta}\nabla\Delta\rho_\delta\|_{L^{2}(Q_T)}
		\|\sqrt{\delta}\nabla \chi_\sigma^{(0)}\|_{L^{2}(0,T;L^{2}(\mathbb{T}^{3}))}
		\\
        &\leq C\|\psi(t)\|_{L^{\infty}(0,T)}.
        \notag 
	\end{align}
Finally, $K_{10}$, $K_{11}$ can be estimated in the following way:
	\begin{align}
				|K_{10}|
        & \leq \|\psi(t)\|_{L^{\infty}(0,T)}  \|\sqrt{\delta}\nabla \rho_\delta\|_{L^2(0,T;L^6(\mathbb{T}^3))} \| \sqrt{\delta}\nabla\bu_\delta\|_{L^2(0,T;L^2(\mathbb{T}^3))} 
        \norm{\cB\big[\chi_\delta^{(0)} \big]}_{L^\infty(0,T;L^3(\mathbb{T}^3))}  
		\notag \\
       & \leq C \|\psi(t)\|_{L^{\infty}(0,T)} 
       \notag 
	\end{align}
and 
    \begin{align}
        |K_{11}| & \leq \|\psi(t)\|_{L^{\infty}(0,T)}	
        \|\sqrt{\delta}  \Delta \bu_\delta  \|_{L^2(0,T;L^2(\mathbb{T}^3))}  \norm{\sqrt{\delta} \Delta\cB\big[\chi_\delta^{(0)}\big]}_{L^2(0,T;L^2(\mathbb{T}^3))}, \notag 
    \end{align}
where in the last estimate, we have used \eqref{eqs:BD-entropy-delta-1} and \eqref{B-bound-chi-1}. Substituting the above estimates for $K_1, ..., K_{14}$ into \eqref{eqs:P-chi-integral} leads to 
\begin{align}
		\label{K1-11}
		\bigg|\int_0^{T} \psi(t) \int_{\mathbb{T}^{3}} \widetilde{P}(\rho_\delta) \chi^{(0)}_{\delta} \,\dxdt\bigg|\leq C( \|\psi'(t)\|_{L^{1}(0,T)}+\|\psi(t)\|_{L^{\infty}(0,T)}).
	\end{align}
Choosing $\psi$ as in \eqref{cut-off function} and 
then sending $ m \to \infty $ in \eqref{K1-11}, we get 
\begin{align}
		\left|\int_0^{T}\int_{\mathbb{T}^{3}} \widetilde{P}(\rho_\delta) \chi_\delta^{(0)}\,\dxdt\right| 
        \leq C.
        \label{K1-11-00}
	\end{align}
    
Applying \eqref{sigular1-delta} with $p=1$, we can check that $\sup_{0\leq t\leq T} \abs{\int_{\mathbb{T}^{3}} \chi(\rho_\delta)\, \dx}$ is bounded by a positive constant independent of $\delta$. Consequently, using the same argument as that for Lemma \ref{lem:uniform-P-chi-sigma}, we can conclude \eqref{eqs:uniform-P-chi} from \eqref{K1-11-00}. This completes the proof. 
\end{proof}
\begin{remark}\rm 
\label{rem:2D-beta}
The assumption $\beta\geq 3/2$ has been used to verify \eqref{B-bound-chi-n1}, \eqref{B-bound-chi-1} and the estimates for $K_{2,3}$, $K_{2,4}$. When the spatial dimension is two, this condition can be relaxed to $\beta>1$. Recall the Br\'{e}zis--Gallouet inequality in two dimensions
\begin{equation}
 \|g\|_{L^\infty(\mathbb{T}^2)}\leq C\|g\|_{H^1(\mathbb{T}^2)} \sqrt{\ln (\mathrm{e}+\|g\|_{H^2(\mathbb{T}^2)})},\quad \forall\, g\in H^2(\mathbb{T}^2).
 \label{BG}
\end{equation}
Applying \eqref{BG} with $g= \sqrt{\delta} \nabla \rho_\delta$, for any $q\in (2,\infty)$, we have 
\begin{align}
\big\|\sqrt{\delta}\nabla\chi^{(0)}_\delta\big\|_{L^{2}(0,T;L^{2}(\mathbb{T}^{3}))}^2
&\leq\|\sqrt{\delta}\nabla\rho_\delta\|_{L^{\frac{2q}{q-2}}(0,T;L^{\infty}(\mathbb{T}^{3}))}^2
\|\chi'(\rho_\delta)\|_{L^{q}(0,T;L^{2}(\mathbb{T}^{3}))}^2
\notag \\
&\leq C \|\sqrt{\delta}\nabla\rho_\delta\|_{L^{\infty}(0,T;H^{1}(\mathbb{T}^{3}))}^2 \left[\int_0^T\big(\mathrm{e} + \|\sqrt{\delta}\nabla\rho_\delta\|_{H^{2}(\mathbb{T}^{3})}\big)^2\,\dt\right]^{\frac{q-2}{q}} 
\notag \\
&\quad\times \|\chi'(\rho_\delta)\|_{L^{\infty}(0,T;L^{\beta}(\mathbb{T}^{3}))}^{\beta(\beta-1)}
\left[ \int_0^T \|\chi'(\rho_\delta)\|_{L^{\beta+1}(\mathbb{T}^{3})}^{\frac{(2-\beta)(1+\beta)q}{2}}\,\dt\right]^{\frac{2}{q}}.
\label{2D-1}
\end{align}
Taking $\beta=2-2/q>1$, then  
$$
\frac{(2-\beta)(1+\beta)q}{2}=1+\beta, 
$$
so that the right-hand side of \eqref{2D-1} can be uniformly bounded in view of \eqref{eqs:BD-entropy-delta-1}, \eqref{sigular1-delta} and \eqref{sigular2-delta}. This recovers the estimates \eqref{B-bound-chi-n1} and \eqref{B-bound-chi-1}. Since $q>2$ can be arbitrarily close to $2$, then $\beta>1$ can be arbitrarily close to $1$. Next, for any $0<\epsilon\ll 1$, it holds 
	\begin{align*}
		& \norm{\cB \left[\chi'(\rho_\delta) \rho_\delta \Div \bu_\delta-\mean{\chi'(\rho_\delta)\rho_\delta  \Div \bu_\delta} \right]}_{L^{2+\epsilon}(\mathbb{T}^{2})}\\ 
		&\quad \leq C\norm{\cB \left[\chi'(\rho_\delta) \rho_\delta \Div \bu_\delta-\mean{\chi'(\rho_\delta)\rho_\delta  \Div \bu_\delta} \right]}_{W^{1,1+\frac{\epsilon}{4+\epsilon}}(\mathbb{T}^{2})}
       \\ 
       &\quad \leq C\norm{\chi'(\rho_\delta)\rho_\delta \Div \bu_\delta}_{L^{1+\frac{\epsilon}{4+\epsilon}}(\mathbb{T}^{2})}
		\\
        &\quad \leq C\norm{\chi'(\rho_\delta)}_{L^{2+\epsilon}(\mathbb{T}^{2})} \norm{\rho_\delta}_{L^{\infty}(\mathbb{T}^{2})}\norm{\Div \bu_\delta}_{L^2(\mathbb{T}^{2})}. 
	\end{align*}
 Consequently, taking $\beta=1+\epsilon$, we infer from \eqref{eqs:uniform-u-del} and \eqref{sigular2-delta} that 
	\begin{align}
		\notag 
		|K_{2,3}|
		& \nonumber\leq \|\psi\|_{L^{\infty}(0,T)}
       \norm{\rho_\delta}_{L^{\infty}(0,T;L^\infty(\mathbb{T}^2))} ^2\|\mathbf{u}_{\delta}\|_{L^{\frac{4+2\epsilon}{\epsilon}}\Big(0,T;L^{\frac{2+\epsilon}{1+\epsilon}}(\mathbb{T}^{2})\Big)}\| \Div \bu_\delta\|_{L^{2}(Q_T)} 		
		\|\chi'(\rho_\delta)\|_{L^{\beta+1}(Q_T)}
		\\
        &\leq C\|\psi\|_{L^{\infty}(0,T)}.
        \notag 
	\end{align}
    The term $K_{2,4}$ can be handled in a similar way by \eqref{eqs:uniform-rho-del}, \eqref{eqs:uniform-rho} and \eqref{sigular2-delta}. Since $\epsilon>0$ can be arbitrarily small, then $\beta>1$ can be arbitrarily close to $1$. 
\end{remark}

\subsection{The limit procedure as $\delta\to 0$}
\label{weak-formulation-sigma} 

First, we infer from \eqref{eqs:uniform-u-del}, \eqref{eqs:uniform-rho-del}, \eqref{eqs:BD-entropy-delta-1} and \eqref{eqs:uniform-rho} that as $\delta\to 0$, 
\begin{align*}
	& \delta\abs{ \int_0^T \int_{\bbt^3} H(\rho_\delta) \Div \bphi\,\dxdt}
	\leq \delta \normm{H(\rho_\delta)}_{L^2(Q_T)} \normm{\Div \bphi}_{L^2(Q_T)} \to 0, 
    \\
    & \delta\abs{  \int_0^T \int_{\bbt^3} (\nabla \rho_\delta \cdot \nabla \bu_\delta) \cdot \bphi \,\dxdt }
    \leq \delta\norm{\nabla \rho_\delta}_{L^\infty(0,T;L^2(\bbt^3)} \norm{\nabla \bu_\delta}_{L^2(Q_T)} \norm{\bphi}_{L^2(0,T;L^\infty(\bbt^3))} \to 0,
    \\
    & \delta \abs{ \int_0^T \int_{\bbt^3} \Delta \bu_\delta \cdot \Delta \bphi \,\dxdt} 
    \leq \sqrt{\delta} \norm{\sqrt{\delta}\Delta \bu_\delta}_{L^2(Q_T)} \norm{\Delta \bphi}_{L^2(Q_T)}  \to 0,
    \\	
	& \delta\abs{ \int_0^{T} \int_{\mathbb{T}^{3}}   \nabla^2\rho_\delta :(\nabla\Delta\rho_\delta  \otimes \bphi)\,\dxdt}
	\leq \sqrt{\delta} \normm{ \rho_\delta}_{L^2(0,T;H^2(\mathbb{T}^3))}  \normm{\sqrt{\delta} \nabla \Delta \rho_\delta}_{L^2(Q_T)} \normm{\bphi }_{L^\infty(Q_T)}\to 0,
    \\
    & \delta \abs{\int_0^{T} \int_{\mathbb{T}^{3}}   (\nabla\Delta\rho_\delta\otimes \nabla\rho_\delta) : \nabla\bphi\,\dxdt} 
    + \delta\abs{\int_0^{T} \int_{\mathbb{T}^{3}}   \nabla\rho_\delta\cdot\nabla\Delta\rho_\delta   \Div\bphi\,\dxdt}
    \\
    &\quad 
    \leq \sqrt{\delta} \normm{ \nabla\rho_\delta}_{L^2(0,T;L^6(\mathbb{T}^3))}  \normm{\sqrt{\delta} \nabla \Delta \rho_\delta}_{L^2(Q_T)} \normm{\bphi }_{L^\infty(0,T;W^{1,3}(\mathbb{T}^3))}\to 0,
    \\
    & \delta \abs{\int_0^{T} \int_{\mathbb{T}^{3}}   \rho_\delta \nabla\Delta \rho_\delta \cdot \nabla\Div\bphi\,\dxdt}
    \leq \sqrt{\delta} \normm{  \rho_\delta}_{L^\infty(Q_T)} \normm{\sqrt{\delta} \nabla \Delta \rho_\delta}_{L^2(Q_T)} \normm{\bphi }_{L^2(0,T;H^2(\mathbb{T}^3))}\to 0.
\end{align*}

Next, the uniform bounds \eqref{eqs:uniform-u-del}, \eqref{eqs:uniform-rho-del}, \eqref{eqs:conservation-mass-delta-b}, \eqref{eqs:uniform-mup-H2}, \eqref{eqs:BD-entropy-delta-1}, \eqref{eqs:uniform-rho} and \eqref{eqs:uniform-pt-rho} yields that there exists a subsequence $\{(\bu_\delta,\rho_\delta,\mu_{p,\delta})\}$ (not relabeled for simplicity) and a triple $(\bu,\rho,\mu_{p})$ such that as $\delta \to 0$:
	\begin{alignat}{3}
		\label{eqs:conv-u-L2}
		\bu_{\delta} & \rightharpoonup \bu, \qquad && \text{ weakly in } L^2(0,T;H^1(\bbt^{3}))\ \ \text{and}\ \ \text{ weakly-star in } L^\infty(0,T;L^2(\bbt^{3})), 
        \\
        \label{eqs:conv-rho-L2}
		\rho_{\delta} & \rightharpoonup \rho, && \text{ weakly in } L^2(0,T;H^{2}(\bbt^{3})) \ \ \text{and}\ \ \text{ weakly-star in } L^\infty(0,T;H^{1}(\bbt^{3})),
        \\
		\label{eqs:conv-pt-rho}
		\pt \rho_{\delta} & \rightharpoonup \pt \rho, && \text{ weakly in } L^2(0,T;L^{\frac{3}{2}}(\bbt^{3})),\\
        \label{eqs:conv-mup}
		  \mu_{p,\delta} & \rightharpoonup \mu_{p}, && \text{ weakly in } L^2(0,T;H^2(\bbt^{3})).
	\end{alignat}
In the following, convergence as $\delta\to 0$ is always understood in the sense of a subsequence. 
Using the Aubin--Lions--Simon compactness theorem, we can deduce that  
\begin{alignat}{3}
	\label{eqs:strongConv-rho-L2}
	\rho_{\delta} & \rightarrow \rho,  \quad \text{ strongly in } L^{2}(0,T;W^{1,q}(\mathbb{T}^{3}))\cap C([0,T];H^s(\mathbb{T}^3)), \quad 1 \leq q < 6,\ \ 0\leq s<1,
\end{alignat}
which also implies 
\begin{alignat}{3}
	\label{eqs:conv-rho-a.e.}
	\rho_{\delta} & \rightarrow \rho,\quad  && \text{ a.e. in } Q_T.
\end{alignat}
Then from Lemma \ref{lem:uniform-P-chi} and \eqref{eqs:conv-rho-a.e.}, we can apply the Dunford--Pettis theorem to conclude
\begin{align}
    \notag 
	\widetilde{P}(\rho_\delta) \rightharpoonup \widetilde{P}(\rho), \quad \text{ weakly in } L^1(Q_T).
\end{align}

Using the weak formulation \eqref{eqs:app-momentum-delta} and the previous uniform estimates for $\mathbf{u}_\delta$, $\rho_\delta$, $\widetilde{P}_\delta$, we can check that 
\begin{align*}
& \left|\big\langle \partial_t(\rho_\delta\mathbf{u}_\delta),\bphi\big\rangle_{L^1(0,T;(W^{2,4}(\mathbb{T}^3))^*),L^\infty(0,T;W^{2,4}(\mathbb{T}^3))}\right| 
\leq C \|\bphi\|_{L^\infty(0,T;W^{2,4}(\mathbb{T}^3))},
\end{align*}
which implies 
$$
\partial_t(\rho_\delta \mathbf{u}_\delta)\ \ \text{is uniformly bounded in} \ \ L^1(0,T;(W^{2,4}(\mathbb{T}^3))^*). 
$$
On the other hand, by interpolation, we can check that
$$
 \rho_\delta\mathbf{u}_\delta\quad \text{is uniformly bounded in }\ \ L^2(0,T;W^{1,\frac32}(\mathbb{T}^3))\cap L^{\frac{10}{3}}(Q_T). 
$$
Following the argument in \cite{ADG20131,AGW2026,FS2021}
we can conclude 
\begin{align}
    \notag 
	\rho_\delta \bu_\delta \to \rho \bu \quad \text{and}\quad \bu_\delta \to \bu \quad  \text{ strongly in } L^2(0,T; L^2(\bbt^3)).
\end{align}

Invoking the above convergence results, for any $\bphi \in C_c^\infty([0,T);C^\infty(\bbt^3))$, we have 
\begin{align*}
	\int_0^{T}\int_{\mathbb{T}^{3}}\rho_{\delta} \bu_{\delta}\cdot\partial_{t}\bphi\,\dxdt 
	& \rightarrow \int_0^{T}\int_{\mathbb{T}^{3}}\rho \bu\cdot\partial_{t}\bphi\,\dxdt,
    \\
	\int_0^{T}\int_{\mathbb{T}^{3}}\rho_{\delta} \bu_{\delta} \otimes \bu_{\delta}: \nabla \bphi\,\dxdt 
	& \rightarrow \int_0^{T}\int_{\mathbb{T}^{3}}\rho \bu \otimes \bu: \nabla \bphi\,\dxdt, 
    \\
	\int_0^{T}\int_{\mathbb{T}^{3}} \rho_{\delta}\bbd\bu_{\delta}: \nabla \bphi\,\dxdt
	& \rightarrow \int_0^{T}\int_{\mathbb{T}^{3}}\rho\bbd\bu: \nabla \bphi\,\dxdt,
    \\
    \int_0^{T}\int_{\mathbb{T}^{3}}\widetilde{P}(\rho_{\delta}) \mathrm{div}\boldsymbol{\varphi}\dxdt
	&\rightarrow
	\int_0^{T}\int_{\mathbb{T}^{3}}\widetilde{P}(\rho) \mathrm{div}\boldsymbol{\varphi}\,\dxdt,
	\\ \int_0^{T}\int_{\mathbb{T}^{3}}(\rho_{\delta}\Delta\rho_{\delta})\mathrm{div}\boldsymbol{\varphi}\,\dxdt
	& \rightarrow\int_0^{T}\int_{\mathbb{T}^{3}}(\rho\Delta\rho)\mathrm{div}\boldsymbol{\varphi}\,\dxdt,
	\\ 
    \int_0^{T}\int_{\mathbb{T}^{3}}\Delta \rho_{\delta}(\nabla \rho_{\delta}\cdot\boldsymbol{\varphi})\,\dxdt
	& \rightarrow\int_0^{T}\int_{\mathbb{T}^{3}}\Delta \rho (\nabla \rho\cdot\boldsymbol{\varphi})\,\dxdt,
    \\
    \int_{\mathbb{T}^3} \rho_{0,\delta}\mathbf{u}_0\cdot \bphi(0)\,\dx
    & \to \int_{\mathbb{T}^3} \rho_{0}\mathbf{u}_0\cdot \bphi(0)\,\dx,
\end{align*}
as $\delta\to 0$. Consequently, we can pass to the limit as $\delta\rightarrow0$ in the weak formulation \eqref{eqs:weak-formulation-delta} to obtain 
\eqref{weak1-sigma}. 
Moreover, using \eqref{eqs:conv-u-L2}--\eqref{eqs:strongConv-rho-L2}, we can recover \eqref{weak2-sigma}
and \eqref{weak3-sigma} by taking the limit  $\delta\rightarrow0$ in  \eqref{eqs:app-mass-sigma1} and \eqref{eqs:app-mup-delta}, respectively. From \eqref{eqs:rho0-convergence}, \eqref{eqs:strongConv-rho-L2} and $\rho_{\delta}|_{t=0}=\rho_{0,\delta}$, we see that the initial condition \eqref{attain-initial-1-sigma} is satisfied. 

Finally, taking the limit $\delta\to 0$ in \eqref{eqs:energy-inequality-delta}, \eqref{eqs:BD-entropy-delta-1} and \eqref{conservation law-000}, we can obtain the energy equality \eqref{eqs:energy-inequality}, the BD-energy estimate \eqref{eqs:BD-entropy} and the conservation of mass \eqref{conservation law-1111}, respectively. The details are omitted. 

The proof of Theorem \ref{thm:NSK} is complete.
\qed

\section{Proof of Theorem \ref{thm:main}}
\label{sec:limit-final}
Now we are ready to prove our main result Theorem \ref{thm:main}. 
Let $(\mathbf{u},\rho,\mu_p)$ be the global weak solution to problem \eqref{eqs:app-limit} constructed in Theorem \ref{thm:NSK}. 
Define 
\begin{align}
	\label{eqs:phi=rho}
	\phi \coloneqq -\zeta\rho+\zeta-1\quad \text{a.e. in}\ \ Q_T.
\end{align}
It follows that 
\begin{align*}
	& \phi \in BC_w(0,T; H^1(\mathbb{T}^{3})) \cap L^{2}(0,T; H^{2}(\mathbb{T}^{3}))\cap W^{1,\frac43}(0,T;L^2(\mathbb{T}^3)),  \\ 
    &\phi\in L^\infty(Q_T)\quad \text{with}\ \ \absm{\phi(x,t)} < 1 \text{ a.e. in } Q_T.
\end{align*}
Since $\widetilde{P}(\rho)= \rho\widetilde{F}'(\rho)-\widetilde{F}(\rho)\in L^1(Q_T)$ and $\widetilde{F}(\rho)\in L^\infty(0,T;L^1(\mathbb{T}^3))$, we can conclude  
$$
F'(\phi) =\frac{1}{\zeta}\widetilde{F}'(\rho)\in L^1(Q_T).
$$
Define
\begin{alignat*}{3}
	\mu  \coloneqq - \Delta \phi  + F'(\phi) , \quad 
	p  \coloneqq \frac{1}{\alpha} (\mu_{p}-\mu),\quad \text{a.e. in}\ \ Q_T.
\end{alignat*}
Then we have $\mu,p \in L^1(Q_T)$. 

From \eqref{weak2-sigma}, \eqref{weak3-sigma} and the linear relation \eqref{eqs:phi=rho}, we obtain 
\begin{align*}
	\partial_{t}\phi
	+ \Div(\phi \mathbf{u})
	& = \frac{1}{\alpha} \Div\mathbf{u}=\Delta\mu_p,\quad \text{a.e.~in } Q_T.
\end{align*}
This implies that \eqref{model2-3} holds. Next, observing that 
\begin{align}
	&\nonumber-\zeta^2\int_0^{T}\int_{\mathbb{T}^{3}}\rho \Delta \rho\Div \bphi\,\dx\dt  
	- \zeta^2\int_0^{T}\int_{\mathbb{T}^{3}}\Delta \rho(\nabla \rho \cdot\bphi)\,\dx\dt
	\\
    &\quad= \zeta \int_0^{T}\int_{\mathbb{T}^{3}}\rho \Delta \phi\Div \bphi\,\dx\dt
	- \int_0^{T}\int_{\mathbb{T}^{3}} \Delta \phi(\nabla \phi \cdot\bphi)\,\dx\dt,
    \notag 
\end{align}
and
\begin{align}
	\int_0^{T}\int_{\mathbb{T}^{3}}\widetilde{P}(\rho)\Div\bphi\,\dx\dt
	\nonumber
    &= \int_0^{T}\int_{\mathbb{T}^{3}}\big(\rho \widetilde{F}'(\rho) - \widetilde{F}(\rho)\big)\Div\bphi\,\dx\dt
	\\
    \notag 
    &=\int_0^{T}\int_{\mathbb{T}^{3}}\big(-\zeta\rho F'(\phi) - F(\phi)\big)\Div\bphi\,\dx\dt,
\end{align}
we find   
\begin{align*}
	&  \int_0^{T}\int_{\mathbb{T}^{3}}\widetilde{P}(\rho)\Div\bphi\,\dx\dt
-\zeta^2\int_0^{T}\int_{\mathbb{T}^{3}}\rho \Delta \rho\Div \bphi\,\dx\dt  
	\\
	& \quad-\zeta^2 \int_0^{T}\int_{\mathbb{T}^{3}}\Delta \rho(\nabla \rho \cdot\bphi)\,\dx\dt
	-\frac{1}{\alpha} \int_0^{T}\int_{\mathbb{T}^{3}}\nabla \mu_p\cdot\bphi\,\dx\dt  
	\\
    &=\int_0^{T}\int_{\mathbb{T}^{3}}\big(-\zeta\rho F'(\phi) - F(\phi)\big)\Div\bphi\,\dx\d\tau 
    +\frac{1}{\alpha} \int_0^{T}\int_{\mathbb{T}^{3}}\mu_p\Div\bphi\,\dx\dt \\
    &\quad+ \zeta \int_0^{T}\int_{\mathbb{T}^{3}}\rho \Delta \phi\Div \bphi\,\dx\dt
	- \int_0^{T}\int_{\mathbb{T}^{3}} \Delta \phi (\nabla \phi \cdot\bphi)\,\dx\dt
    \\ & 
	= \int_0^{T}\int_{\mathbb{T}^{3}} p\Div\bphi\,\dx\dt
    +\int_0^{T}\int_{\mathbb{T}^{3}} \phi\mu\Div\bphi\,\dx\dt
    -\int_0^{T}\int_{\mathbb{T}^{3}}F(\phi)\Div\bphi\,\dx\dt
	\nonumber\\
	& \quad -\int_0^{T}\int_{\mathbb{T}^{3}}\frac{\absm{\nabla \phi}^2}{2}\Div\bphi\,\dx\dt
	+\int_0^{T}\int_{\mathbb{T}^{3}}(\nabla \phi \otimes \nabla \phi):\nabla \bphi\,\dx\dt,
\end{align*}
for any $\bphi\in C_c^\infty([0,T);C^\infty(\mathbb{T}^3))$. 
Inserting the above identity into \eqref{weak1-sigma}, 
we recover the weak formulation \eqref{original-weak1} of the momentum equation. 

The proof of Theorem \ref{thm:main} is complete. 
\qed

\appendix
\section{Technical Lemma}
\label{sec:technical-lemma}
	
We report some properties of 
mollifiers in the torus $\mathbb{T}^3$. 
	\begin{lemma}
		Let $\bbt^d = \bbr^d/\bbz^d$ be the $d$-dimensional flat torus and let
		$\eta\in C_c^\infty(\bbr^d)$ be a standard mollifier that is radial symmetric, i.e.,
		\[
		\eta\ge 0,\qquad \supp\eta\subset B_1(0),\qquad \int_{\bbr^d}\eta(x)\,\dx=1 .
		\]
		For $\varepsilon>0$, set
		\[
		\eta_\varepsilon(x) = \frac{1}{\varepsilon^{d}}\eta\left(\frac{x}{\varepsilon}\right),\qquad
		\eta^{\mathrm{per}}_\varepsilon(x) = \sum_{\mathbf{k}\in\bbz^d}\eta_\varepsilon(x+\mathbf{k}),
		\]
		so that $\eta^{\mathrm{per}}_\varepsilon$ is $\bbz^d$-periodic and can be regarded as a
		smooth function on $\bbt^d$ with $\int_{\bbt^d}\eta^{\mathrm{per}}_\varepsilon \,\dx=1$.
		For $f\in L^1(\bbt^d)$, we define the mollification
		\[
		(J_\varepsilon f)(x) = (f * \eta^{\mathrm{per}}_\varepsilon)(x)
		:= \int_{\bbt^d} f(y)\,\eta^{\mathrm{per}}_\varepsilon(x-y)\,\mathrm{d}y .
		\]
		Then the following properties hold:
		\begin{enumerate}
			\item[(a)] For every $f\in L^1(\bbt^d)$ we have $J_\varepsilon f\in C^\infty(\bbt^d)$.
			
			\item[(b)] For $f\in L^p(\bbt^d)$, $1\le p\le\infty$, we have
			$\|J_\varepsilon f\|_{L^p(\bbt^d)} \le \|f\|_{L^p(\bbt^d)}$.
			When $1\le p<\infty$ we have
			$\|J_\varepsilon f-f\|_{L^p(\bbt^d)}\to 0$ as $\varepsilon\to0$, while
			if $f\in C(\bbt^d)$ then $J_\varepsilon f \to f$ uniformly on $\bbt^d$.
			
			\item[(c)] Let $s\in\bbr$ and $f\in H^s(\bbt^d)$. Then we have
			\begin{align*}
			    \|J_\varepsilon f\|_{H^{s+1}(\bbt^d)}
                \leq C \varepsilon^{-1} \|f\|_{H^{s}(\bbt^d)}.
			\end{align*}
			
			
			
		\end{enumerate}
	\end{lemma}

	\section*{Acknowledgments}
	\noindent M. Fei is partially supported by the NSF of China under Grant No.12271004 and 12471222 and the NSF of Anhui Province of China
under Grant No.2308085J10. Y. Liu is partially supported by the NSF of Jiangsu Province of China under Grant No.~BK20240572, the NSF of Jiangsu Higher Education Institutions of China under Grant No.~24KJB110020, and the China Postdoctoral Science Foundation under Grant No.~2025M773078. H. Wu is partially supported by the NSF of Shanghai under Grant No.~25ZR1401023. H. Wu is a member of the Key Laboratory of Mathematics for Nonlinear Sciences (Fudan University), Ministry of Education of China.

	\section*{Compliance with Ethical Standards}
	\subsection*{Date avability}
	Data sharing not applicable to this article as no datasets were generated during the current study.
	\subsection*{Conflict of interest}
	The authors declare that there are no conflicts of interest.

	
\end{document}